\pdfoutput=1
\documentclass[10pt]{article}
\usepackage{amsmath,amssymb,amsthm,enumitem,graphicx,tikz}
\usepackage{mathrsfs}
\usepackage{amssymb,amsfonts,amsmath,color}

\newtheorem{theorem}{Theorem}
\newtheorem{lemma}[theorem]{Lemma}
\newtheorem{problem}[theorem]{Problem}
\newtheorem{corollary}[theorem]{Corollary}

\newtheorem{conjecture}[theorem]{Conjecture}
\newtheorem{prop}[theorem]{Proposition}

\newcommand{\bb}[1]{\raisebox{-2ex}[0pt][0pt]{\shortstack{#1}}}

\title{The smallest number of vertices in a 2-arc-strong digraph which has no good pair}

\author{Ran Gu$^1$, Gregory Gutin$^2$, Shasha Li$^3$,\ Yongtang Shi$^4$, Zhenyu Taoqiu$^4$  \\[2mm]
{\small $^1$ College of Science, Hohai University}\\
{\small Nanjing, Jiangsu Province 210098,
P.R. China}\\
{\small Email: rangu@hhu.edu.cn}\\
{\small $^2$ Department of Computer Science}\\
{\small Royal Holloway, University of London}\\
{\small Egham, Surrey, TW20 0EX, UK
}\\
{\small Email: g.gutin@rhul.ac.uk}\\
{\small $^3$ Department of Mathematics}\\
{\small Ningbo University, Ningbo 315211, Zhejiang, China}\\
{\small Email: yezi$\_$pg@163.com}\\
{\small $^4$  Center for Combinatorics and LPMC}\\
{\small Nankai University, Tianjin 300071, China}\\
{\small Emails: shi@nankai.edu.cn, tochy@mail.nankai.edu.cn}
}
\date{\today}

\begin{document}
	\maketitle
\begin{abstract}
Bang-Jensen, Bessy, Havet and Yeo showed that every digraph of independence number at most $2$ and arc-connectivity at least $2$ has
an out-branching $B^+$ and an in-branching $B^-$ which are arc-disjoint (such a pair of branchings is called a {\it good pair}), which settled a conjecture of Thomassen for digraphs of independence number $2$. They also proved that every digraph on at most $6$ vertices and arc-connectivity at least $2$ has a good pair and gave an example of a $2$-arc-strong digraph $D$ on $10$ vertices with independence number 4 that has no good pair. They asked for the smallest number $n$ of vertices in a $2$-arc-strong digraph which has no good pair.
In this paper, we prove that every digraph on at most $9$ vertices and arc-connectivity at least $2$ has a good pair, which solves this problem.	
\end{abstract}

\tableofcontents

\section{Introduction}

An {\em out-branching} ({\em in-branching}) of a
digraph $D=(V,A)$ is a spanning tree in the underlying graph of $D$ whose edges are oriented in $D$ such that every
vertex except one, called the {\em root}, has in-degree (out-degree) one. For a non-empty subset $X\subset V$, the \textit{in-degree} (resp. \textit{out-degree}) of the set $X$, denoted by $d_D^-(X)$ (resp. $d_D^+(X)$), is the number of arcs with head (resp. tail) in $X$ and tail (resp. head) in $V\setminus X$.
The {\it arc-connectivity} of $D$, denoted by $\lambda(D)$, is the
minimum out-degree of a proper subset of vertices. A digraph is \textit{$k$-arc-strongly connected} (or just \textit{$k$-arc-strong}) if $\lambda(D)\ge k$. In particular, a digraph is {\it strongly connected} (or just {\it strong}) if $\lambda(D)\geq 1$.

It is an interesting problem to characterize digraphs having an out-branching and an in-branching which are arc-disjoint. Such a pair of branchings are called a {\it good pair}. Thomassen \cite{Thomassen}
proved that it is {\sf NP}-complete to decide whether a given digraph $D$ has an out-branching
and an in-branching both rooted at the same vertex such that these are arc-disjoint. This implies that it is {\sf NP}-complete to decide if a given digraph has an out-branching and in-branching which are arc-disjoint \cite{BBHY}.
Thomassen also conjectured that every digraph of sufficiently high arc-connectivity has a good pair.

\begin{conjecture}[\cite{Thomassen}] \label{conj}
There is a constant $c$, such that every digraph with arc-connectivity at
least $c$ has an out-branching and an in-branching which are arc-disjoint.
\end{conjecture}

Conjecture \ref{conj} has been verified for semicomplete digraphs \cite{Bang-Jensen} and their genearlizations: locally semicomplete digraphs \cite{BH} and semicomplete compositions \cite{BGY} (follows from its main result in \cite{BGY}).
In \cite{BBHY}, Bang-Jensen, Bessy, Havet and Yeo showed that every digraph of independence number at most $2$ and arc-connectivity at least $2$ has
a good pair, which settles the conjecture for digraphs of independence number $2$.
\begin{theorem}[\cite{BBHY}]\label{thm1}
If $D$ is a digraph with $\alpha(D)\leq 2\leq \lambda(D)$, then $D$ has a good pair.
\end{theorem}
\noindent Moreover, they also proved that every
digraph on at most $6$ vertices and arc-connectivity at least $2$ has a good pair and gave an example of a
$2$-arc-strong digraph $D$ on $10$ vertices with independence number 4 that has no good pair. They asked for the smallest number $n$ of vertices in a $2$-arc-strong digraph which has no good pair.
The following is the first problem in Section 8 of [2].
\begin{problem}[\cite{BBHY}]
What is the smallest number $n$ of vertices in a $2$-arc-strong digraph which has no good pair?
\end{problem}
In this paper, we prove that every digraph on at most $9$ vertices and arc-connectivity at least $2$ has a good pair, which answers this problem. The main results of the paper are shown below.

\begin{theorem}\label{thm2}
	Every 2-arc-strong digraph on 7 vertices has a good pair.
\end{theorem}

\begin{theorem}\label{thm3}
	Every 2-arc-strong digraph on 8 vertices has a good pair.
\end{theorem}

\begin{theorem}\label{thm4}
	Every 2-arc-strong digraph on 9 vertices has a good pair.
\end{theorem}

This paper is organised as follows. In the rest of this section, we provide further terminology and notation on digraphs. Undefined terms can be found in \cite{BG,BG2}.
In Section~\ref{outline}, we outline the proofs of Theorem \ref{thm2}, \ref{thm3} and \ref{thm4} and state some auxiliary lemmas which we use in their proofs.
Section~\ref{pre} contains a number of technical lemmas which will be used in proofs of our main results. Then we respectively devote one section for proofs of each theorem and its relevant auxiliary lemmas.
A number of supplementary proofs are moved to Appendix.

\paragraph{Additional Terminology and Notation.} For a positive integer $n$, $[n]$ denotes the set $\{1,2,\ldots,n \}$. Throughout this paper, we will only consider digraphs without loops and multiple arcs. Let $D=(V,A)$ be a digraph.
We denote by $uv$ the arc whose \textit{tail} is $u$ and whose \textit{head} is $v$. Two vertices $u,v$ are \textit{adjacent} if at least one of $uv$ and $vu$ belongs to $A$. If $u$ and $v$ are adjacent, then we also say that $u$ is a \textit{neighbour} of $v$ and vice versa.
If $uv\in A$, then $v$ is called an \textit{out-neighbour} of $u$ and $u$ is called an \textit{in-neighbour} of $v$. Moreover, we say $uv$ is an \textit{out-arc} of $u$ and an \textit{in-arc} of $v$ and that $u$ {\em dominates} $v$. The {\em order} $|D|$ of $D$ is $|V|.$

In this paper, we will extensively use {\em digraph duality}, which is as follows. Let $D$ be a digraph and let $D^{\rm rev}$ be the {\em reverse} of $D$, i.e., the digraph obtained from $D$ by reversing every arc $xy$ to $yx.$
Clearly, $D$ contains a subdigraph $H$ if and only if $D^{\rm rev}$ contains $H^{\rm rev}.$ In particular, $D$ contains a good pair if and only if $D^{\rm rev}$ contains a good pair.

Let $N_D^-(X)=\{y: yx\in A, x\in A\}$ and $N_D^+(X)=\{y: xy\in A, x\in A\}$. Note that $X$ may be just a vertex. For two non-empty disjoint subsets $X,Y\subset V$, we use $N_Y^-(X)$ to denote $N_D^-(X)\cap Y$ and $d_Y^-(X)=|N_Y^-(X)|$. Analogously, we can define $N_Y^+(X)$ and $d_Y^+(X)$.
For two non-empty subsets $X_1,X_2\subset V$, define $(X_1,X_2)_D=\{(v_1,v_2)\in A\colon v_1\in X_1~\text{and}~v_2\in X_2 \}$ and $[X_1,X_2]_D=(X_1,X_2)_D\cup (X_2,X_1)_D$.
We will drop the subscript when the digraph is clear from the context.

We write $D[X]$ to denote the subdigraph of $D$ induced by $X$. A \textit{clique} in $D$ is an induced subdigraph $D[X]$ such that any two vertices of $X$ are adjacent. We say that $D$ contains $K_p$ if it has a clique on $p$ vertices. A vertex set $X$ of $D$ is {\em independent} if no pair of vertices in $X$ are adjacent.
A dipath (dicycle) of $D$ with $t$ vertices is denoted by $P_t$ ($C_t$).  We drop the subscript when the order is not specified. A dipath $P$ from $v_1$ to $v_2$, denoted by $P_{(v_1,v_2)}$, is often called a $(v_1,v_2)$-{\em dipath}. A dipath $P$ is a \textit{Hamilton} dipath if $V(P)=V(D)$.
We call $C_2$ a \textit{digon}. A digraph without digons is called an {\em oriented graph}.
If two digons have exactly one common vertex, then we call this structure a \textit{bidigon}.
A \textit{semicomplete} digraph is a digraph $D$ that each pair of vertices has an arc between them. A \textit{tournament} is a semicomplete oriented graph.

In- and out-branchings were defined above. An {\em out-tree} ({\em in-tree}) is an out-branching (in-branching) of a subdigraph of $D.$
We use $B_s^+$ ($B_t^-$) to denote an out-branching rooted at $s$ (an in-branching rooted at $t$). The root $s$ ($t$) is called \textit{out-generator }(\textit{in-generator}) of $D$. We denote by Out($D$) (In($D$)) the set of out-generators (in-denerators) of $D$.
If the root is not specified, then we drop the subscripts of $B_s^+$ and $B_t^-$. We also use $O_D$ ($I_D$) to denote an out-branching (in-branching) of a digraph $D$. If $O_D$ and $I_D$ are arc-disjoint, then we write $(O_D,I_D)$ to denote a good pair in $D$.

\section{Proofs Outline}\label{outline}
In this section, we outline constructions we use to prove our main results. We prove each of them by contradiction. We give the statements of some auxiliary lemmas. For simplicity, when outlining the proof of our main results, we assume that $|D_1|=7$, $|D_2|=8$ and $|D_3|=9$.

\subsection{Theorem~\ref{thm2}}
First we get that the largest clique in $D_1$ is a tournament by Lemma~\ref{lem6}, which is given in Section~\ref{pre}.
Next we prove that $D_1$ is an oriented graph in Claim~\ref{thm2}.1 by Lemma~\ref{lem7}, which is shown in Section~\ref{pre}.
The next proposition, which shows that $D_1$ has a Hamilton dipath, is proved in Section~\ref{7}.
\begin{description}
	\item[Proposition~\ref{h1}] A 2-arc-strong oriented graph $D$ on $n$ vertices has a $P_7$, where $7\le n\le9$.
\end{description}
After that, we prove that $D_1$ has a good pair by Propositon~\ref{prop6}, which is shown in Section~\ref{pre}.

\subsection{Theorem~\ref{thm3}}
Our proof will follow three steps.

Firstly, we get that the largest clique $R$ in $D_2$ has 3 vertices by Lemma~\ref{lem6}. And we show that $R$ is a tournament through Claim~\ref{thm3}.1, which is proved by Lemmas~\ref{lem6} and \ref{lem7}.

Our second step is to prove that $D_2$ is an oriented graph in Claim~\ref{thm3}.2 by Lemmas~\ref{lem8}, \ref{lem8-1} and \ref{lem8-2}, which are given in Section~\ref{pre}. For completeness, we give partial proofs of these three lemmas in Appendix.

In the last step, we proceed as follows. The next proposition, which implies that $D_2$ has a Hamilton dipath, is proved in Section~\ref{8}.
\begin{description}
	\item[Proposition~\ref{h2-2}] Let $D$ be a 2-arc-strong digraph on $n$ vertices without a good pair, where $8\le n\le9$. If $D$ is an oriented graph without $K_4$ as a subdigraph, then $D$ has a $P_8$.
\end{description}
To prove it, we first show the proposition below, which will also be given in Section~\ref{8}.
\begin{description}
	\item[Proposition~\ref{h2-1}] Let $D$ be a 2-arc-strong oriented graph on $n$ vertices without $K_4$ as a subdigraph, where $8\le n\le9$. If $D$ has two disjoint cycles $C^1$ and $C^2$ which cover 7 vertices, then $D$ contains a $P_8.$
\end{description}
After that, we prove that $D_2$ has a good pair by Propositon~\ref{prop6}.

\subsection{Theorem~\ref{thm4}}
Our proof will follow four steps.

Firstly, we show that the largest clique $R$ in $D_3$ has 3 vertices by Claim~\ref{thm4}.1, which is proved by Proposition~\ref{prop1} given in Section~\ref{pre}, Lemmas~\ref{lem6} and \ref{lem7}.

Next we show that $R$ has no digons by Claim~\ref{thm4}.2, which is proved analogously to Claim~\ref{thm3}.1 by Lemmas~\ref{lem7}, \ref{lem8}, \ref{lem8-1} and \ref{lem8-2}.

Our third step is to show that $D_3$ is an oriented graph in Claim~\ref{thm4}.3. To do this we need two lemmas below, which are proved in Section~\ref{9}.
\begin{description}
	\item[Lemma~\ref{lem11}] Let $D$ be a 2-arc-strong digraph on $9$ vertices that contains a digon $Q$. Assume that $D$ has no subdigraph with a good pair on 3 or 4 vertices. Set $X=N_D^-(Q)$ and $Y=N_D^+(Q)$ with $X\cap Y=\emptyset$. If $|X|=3$ and $|Y|=2$, then $D$ has a good pair.
	\item[Lemma~\ref{lem13}] Let $D$ be a 2-arc-strong digraph on $9$ vertices that contains a digon $Q$. Assume that $D$ has no subdigraph with a good pair on 3 or 4 vertices. Set $X=N_D^-(Q)$ and $Y=N_D^+(Q)$ with $X\cap Y=\emptyset$. If $|X|=2$ and $|Y|=2$, then $D$ has a good pair.
\end{description}
To prove Lemma~\ref{lem11}, we give a generalization of Proposition~\ref{prop2} in Proposition~\ref{prop10}, which is shown in Section~\ref{9}.
To prove Lemma~\ref{lem13}, we prove the lemma below first, which will also be given in Section~\ref{9}.
\begin{description}
	\item[Lemma~\ref{lem12}] Let $D=(V,A)$ be a 2-arc-strong digraph on $9$ vertices that contains a digon $Q$. Assume that $D$ has no subdigraph with a good pair on at least 3 vertices. Set $X=N_D^-(Q)$ and $Y=N_D^+(Q)$ with $X\cap Y=\emptyset$ and $W=V-V(Q)-X-Y$.
	Assume that $|X|=|Y|=2$ and there is an arc $e=st\in A$ such that $s\in Y$ and $t\in W$ (resp. $s\in W$ and $t\in X$). If there are at least three arcs in $D[Y\cup\{t\}]$ (resp. $D[X\cup\{s\}]$), then $D$ has a good pair.
\end{description}
The next lemma, which shows that $D_3$ has a Hamilton dipath, is proved in Section~\ref{9}.
\begin{description}
	\item[Lemma~\ref{h3-2}] Let $D$ be a 2-arc-strong digraph on 9 vertices without good pair. If $D$ is an oriented graph without $K_4$ as a subdigraph, then $D$ has a Hamilton dipath.
\end{description}
To prove it, we show the proposition below first, which will also be given in Section~\ref{9}.
\begin{description}
	\item[Proposition~\ref{h3-1}] Let $D$ be a 2-arc-strong oriented graph on 9 vertices without $K_4$ as a subdigraph. If $D$ has two cycles $C^1$ and $C^2$ with $C^1\cap C^2=\emptyset$ which cover 8 vertices, then $D$ contains a Hamilton dipath.
\end{description}
After that, we prove that $D_3$ has a good pair by Proposition~\ref{prop6}.
For completeness, we give some proofs of Lemmas~\ref{lem11}, \ref{lem12}, \ref{lem13} and \ref{h3-2} in Appendix.

\section{Preliminaries and useful lemmas}\label{pre}

\begin{prop}\label{prop1}
	Let $D$ be a digraph with $\lambda(D)\ge2$ and with a good pair $(B_s^+,B_s^-)$. If there exists a vertex $t$ in $D$ such that $D[\{s,t\}]$ is a digon, then $D$ has a good pair $(B_t^+,B_t^-)$.
\end{prop}
\begin{proof}
	Let $B_t^+=ts+B_s^+-e_1$ and $B_t^-=B_s^-+st-e_2$, where $e_1$ ($e_2$) is the only in-arc (out-arc)
	of $t$ in $B_s^+$ ($B_s^-$). Observe that $B_t^+$ ($B_t^-$) is an out-branching (in-branching) rooted at $t$ in $D$.
	Since the root of any out-branching has in-degree zero, if $ts \in B_s^+ \cup B_s^-$, then
	$ts$ must be in $B_s^-$ and moreover $ts$ is the only out-arc $e_2$ of $t$ in $B_s^-$. Similarly,
	if $st \in B_s^+ \cup B_s^-$, then $st$ must be in $B_s^+$ and moreover $st$ is the only in-arc $e_1$ of $t$ in $B_s^+$.
	Thus, $B_t^+$ and $B_t^-$ are arc-disjoint and so $(B_t^+,B_t^-)$ is a good pair of $D$.
\end{proof}

\begin{prop}\label{prop2}
	Let $D$ be a digraph with a subdigraph $Q$ that has a good pair $(O_Q, I_Q)$. Let $X=N_D^-(Q)$ and $Y=N_D^+(Q)$ with $X\cap Y=\emptyset$ and $X\cup Y=V-V(Q)$. Let $X_i$ ($Y_j$) be the initial (terminal) strong components in $D[X]$ ($D[Y]$), $i\in[a]$ ($j\in [b]$). If one of the following holds, then $D$ has a good pair.
	\begin{enumerate}
		\item $d_Y^-(X_1)\ge 1$, $d_Y^-(X_i)\ge 2,~ i\in \{2,\ldots, a\}$ and $d_X^+(Y_j)\ge 2,~ j\in[b]$.
		\item $d_X^+(Y_1)\ge 1$, $d_X^+(Y_j)\ge 2,~ j\in \{2,\ldots,b\}$ and $d_Y^-(X_i)\ge 2,~ i\in[a]$.
	\end{enumerate}
\end{prop}
\begin{proof}
	Let $B^+$ be an out-tree containing $O_Q$ and an in-arc of any vertex in $Y$ from $Q$. Let $B^-$ be an in-tree containing $I_Q$ and an out-arc of any vertex in $X$ to $Q$. Set $\mathcal{X}=\{X_i,~i\in[a]\}$ and $\mathcal{Y}=\{Y_j,~j\in[b]\}$.
	By the digraph duality, it suffices to prove that condition 1 implies that $D$ has a good pair.

	Now assume that $d_Y^-(X_1)\ge 1$, $d_Y^-(X_i)\ge 2,~ i\in \{2,\ldots, a\}$, and $d_X^+(Y_j)\ge 2,~ j\in[b].$ Then there are at least two arcs from $Y_j$ (for each $j\in[b]$) to $X$, at least two arcs from $Y$ to $X_i$ (for each $i\in\{2,\ldots,a\}$) and at least one arc from $Y$ to $X_1$. Set $X'_1=X_1$.
	If there is an arc $y^1x_1$ from $Y$ to $X'_1$ with $y^1$ in some $Y_j,~j\in[b]$, then we choose such an arc and let $Y'_1=Y_j$, otherwise we choose an arbitrary arc $y^1x_1$ from $Y$ to $X'_1$ and let $Y'_1$ be an arbitrary strong component in $\mathcal{Y}$. Let $\mathcal{P}_X=\{y^1x_1\}$. There now exists an arc, $y_1x^1$, out of $Y'_1$ ($x^1\in X$) which is different from $y^1x_1$ (as $Y'_1$ has at least two arcs out of it). If there is such an arc $y_1x^1$ with $x^1$ in some $X_i,~i\in\{2,\ldots,a\}$, then we choose one of these arcs and let $X'_2=X_i$, otherwise we choose such an arbitrary arc $y_1x^1$ out of $Y'_1$ ($x^1\in X$) and let $X'_2$ be an arbitrary strong component in $\mathcal{X}-X'_1$. Let $\mathcal{P}_Y=\{y_1x^1\}$.
	Likewise, for $t\ge2$, we get an arc $y^tx_t$ into $X'_t$ ($y^t\in Y$) which is different from $y_{t-1}x^{t-1}$ in $\mathcal{P}_Y$. If there is such an arc $y^tx_t$ with $y^t$ in some $Y_j\in \mathcal{Y}-\{Y'_1,\ldots,Y'_{t-1}\}$, then choose one of these arcs and let $Y'_t=Y_j$, otherwise we choose such an arbitrary  arc $y^tx_t$ and let $Y'_t$ be an arbitrary strong component in $\mathcal{Y}-\{Y'_1,\ldots,Y'_{t-1}\}$. Add $y^tx_t$ to $\mathcal{P}_X$.
	For $s\ge2$, we get an arc $y_sx^s$ out of $Y'_s$ ($x^s\in X$) which is different from $y^sx_s$ in $\mathcal{P}_X$. If there is such an arc $y_sx^s$ with $x^s$ in some $X_i\in \mathcal{X}-\{X'_1,\ldots,X'_{s-1}\}$, then we choose one of these arcs and let $X'_s=X_i$, otherwise we choose such an arbitrary arc $y_sx^s$ and let $X'_s$ be an arbitrary strong component in $\mathcal{X}-\{X'_1,\ldots,X'_{s-1}\}$. Add $y_sx^s$ to $\mathcal{P}_Y$.
	Hence we get two arc sets $\mathcal{P}_X$ and $\mathcal{P}_Y$ with $\mathcal{P}_X\cap \mathcal{P}_Y=\emptyset$.
	
	We will now show that $D$ has a good pair. Let $D_X$ be the digraph obtained from $D[X]$ by adding one new vertex $y^*$ and arcs from $y^*$ to $x_i$ for $i\in[a]$. Analogously let $D_Y$ be the digraph obtained from $D[Y]$ by adding one new vertex $x^*$ and arcs from $y_j$ to $x^*$ for $j\in[b]$. Since Out($D_X$)$=\{y^*\}$ and In($D_Y$)$=\{x^*\}$, there exists an out-branching $B_{y^*}^+$ in $D_X$ and an in-branching $B_{x^*}^-$ in $D_Y$. Set $T_X=B_{y^*}^+-y^*$ and $T_Y=B_{x^*}^--x^*$.
	
	By construction, $(O_D,I_D)$ is a good pair of $D$ with $O_D=B^++\mathcal{P}_X+T_X$ and $I_D=B^-+\mathcal{P}_Y+T_Y$.
\end{proof}

Note that we can always get two arc-disjoint $\mathcal{P}_X,\mathcal{P}_Y$ and respectively an out- and an in-forest $T_X$ and $T_Y$ as above in a digraph $D$ when one of the conditions in Proposition~\ref{prop2} holds for $D$.

\begin{corollary}\label{cor1}
	Let $D$ be a digraph with $\lambda(D)\ge2$ that contains a subdigraph $Q$ with a good pair. Set $X=N_D^-(Q)$ and $Y=N_D^+(Q)$. If $X\cap Y=\emptyset$ and $X\cup Y=V-V(Q)$, then $D$ has a good pair.
\end{corollary}
\begin{proof}
	Let $X_i$ be the initial strong components in $D[X]$ and $Y_j$ be the terminal strong components in $D[Y]$, $i\in[a]$ and $j\in [b]$. Since $\lambda(D)\ge2$, $d_Y^-(X_i)\ge2$ and $d_X^+(Y_j)\ge2$, for any $i\in[a]$ and $j\in[b]$, which implies that $D$ has a good pair by Proposition~\ref{prop2}.
\end{proof}

\begin{lemma}[\cite{BBHY}]\label{lem1}
	Let $D$ be a digraph and $X\subset V(D)$ be a set such that every vertex of $X$ has both an in-neighbour and an out-neighbour in $V-X$. If $D-X$ has a good pair, then $D$ has a good pair.
\end{lemma}

By Lemma \ref{lem1}, in this paper we will often use the fact that if $Q$ is a maximal subdigraph of $D$ with a good pair and $X=N_D^-(Q), Y=N_D^+(Q)$, then $X\cap Y=\emptyset$.

\begin{lemma}\label{lem2}
	Let $D$ be a 2-arc-strong digraph containing a subdigraph $Q$ with a good pair, $X=N_D^-(Q)$ and $Y=N_D^+(Q)$. If $X\cap Y=\emptyset$ and $X\cup Y=V-V(Q)-\{w\}$, where $w\in V-V(Q)$, then $D$ has a good pair.
\end{lemma}
\begin{proof}
	Assume that $Q$ has a good pair $(O_Q,I_Q)$.
	Let $B^+$ be an out-tree containing $O_Q$ with an in-arc of any vertex in $Y$ from $Q$, while $B^-$ be an in-tree containing $I_Q$ with an out-arc of any vertex in $X$ to $Q$.
	
First assume that either $(Y,w)_D\ne \emptyset$ or $(w,X)_D\ne \emptyset$. By the digraph duality, we may assume that $(Y,w)_D\ne \emptyset$, i.e., there exists an arc $e$ from $Y$ to $w$ in $D$. Let $D'=D-e$. Set $X'=N_{D'}^-(Q)=X$ and $Y'=N_{D'}^+(Q)\cup \{w\}=Y\cup\{w\}$. Let $X'_i$ be the initial strong components in $D'[X']$ and $Y'_j$ be the terminal strong components in $D'[Y']$, $i\in[a]$ and $j\in [b]$. If $w$ has an in-neighbour $v$ in $Y$ with $v$ in some $Y'_j,~ j\in [b]$, then let $e=vw$ and $Y^*_1=Y'_j$, otherwise we choose an arbitrary in-neighbour $v$ of $w$ in $Y$ and let $e=vw$ and $Y^*_1$ be an arbitrary terminal strong component of $D'[Y']$.
	Since $\lambda(D)\ge2$, $d_{X'}^+(Y^*_1)\ge1$, $d_{X'}^+(Y'_j)\ge2$ and	$d_{Y'}^-(X'_i)\ge2$, for any $Y'_j\neq Y^*_1$, $j\in[b]$ and $i\in[a]$, which implies that we get arc sets $\mathcal{P}_{X'}$ and $\mathcal{P}_{Y'}$ with $\mathcal{P}_{X'}\cap \mathcal{P}_{Y'} =\emptyset$, and digraphs $T_{X'}$ and $T_{Y'}$ by Proposition~\ref{prop2}.
	By construction, $D$ has a good pair $(B^++\mathcal{P}_{X'}+T_{X'}+e,B^-+\mathcal{P}_{Y'}+T_{Y'})$.
	
	Now assume that $(Y,w)_D=\emptyset$ and $(w,X)_D=\emptyset$, which implies that $d_X^-(w)\ge 2$ and $d_Y^+(w)\ge 2$. Let $X_i$ be the initial strong components in $D[X]$ and $Y_j$ be the terminal strong components in $D[Y]$, $i\in[a]$ and $j\in [b]$.
	Since $\lambda(D)\ge2$ and $(w,X)_D=(Y,w)_D=\emptyset$, $d_Y^-(X_i)\ge2$ and $d_X^+(Y_j)\ge2$ for any $i\in[a]$ and $j\in[b]$. By Proposition~\ref{prop2}, we get $\mathcal{P}_X,T_X$ and $\mathcal{P}_Y,T_Y$ with $\mathcal{P}_X\cap \mathcal{P}_Y=\emptyset$. It follows that $(B^++\mathcal{P}_X+T_X+w^-w,B^-+\mathcal{P}_Y+T_Y+ww^+)$ is a good pair of $D$, where $w^-\in X$ and $w^+\in Y$.
\end{proof}

\begin{prop}[\cite{BBHY}]\label{prop3}
	Every digraph on 3 vertices with at least 4 arcs has a good pair.
\end{prop}

Following \cite{BG}, we shall use $\delta_0(D)$ to denote the \textit{minimum semi-degree} of $D$, which is the minimum over all in- and out-degrees of vertices of $D$.

\begin{prop}[\cite{BBHY}]\label{prop4}
	Let $D$ be a digraph on 4 vertices with at least 6 arcs except $E_4$ (see Figure~\ref{fig3}). If $\delta^0(D)\ge 1$ or $D$ is a semicomplete digraph, then $D$ has a good pair.
\end{prop}

\begin{figure}[!htpb]
	\centering\includegraphics[scale=0.5]{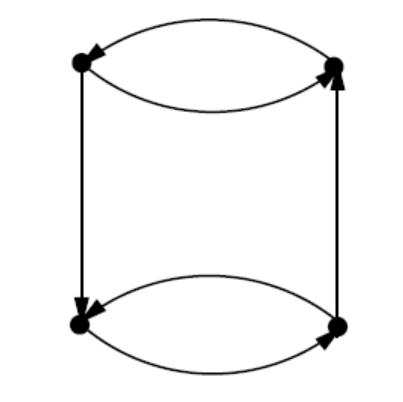}
	\caption{$E_4$.}
	\label{fig3}
\end{figure}

\begin{lemma}[\cite{BG2}, p.354]\label{lem3}
	Let $D=(V,A)$ be a digraph. Then $D$ is $k$-arc-strong if and only if it contains $k$ arc-disjoint $(s,t)$-paths for
	every choice of distinct vertices $s,t\in V$.
\end{lemma}

\begin{lemma}[Edmonds' branching theorem \cite{BG}]\label{lem4}
	A directed multigraph $D=(V,A)$ with a special vertex $z$ has $k$ arc-disjoint
	out-branchings rooted at $z$ if and only if $d^-(X)\geq k$ for all $\emptyset\neq X\subseteq V-z$.
\end{lemma}

\begin{lemma}[\cite{BBHY}]\label{lem5}
	If $D$ is a 2-arc-strong digraph on $n$ vertices that contains a subdigraph on $n-3$ vertices with a good pair, then $D$ has a good pair.
\end{lemma}

\begin{lemma}\label{lem6}
	If $D$ is a 2-arc-strong digraph on $n$ vertices that contains a subdigraph $Q$ on $n-4$ vertices with a good pair, then $D$ has a good pair.
\end{lemma}
\begin{proof}
	Let $(O_Q,I_Q)$ be a good pair of $Q$ and set $X=N^-_D(Q)$ and $Y=N^+_D(Q)$. Since $|Q|=n-4$,
	$|D-Q|=4$. If $X\cap Y\neq \emptyset$, then there is a vertex $v$ in $D-Q$ which has both an in-neighbour and an out-neighbour in $Q$.
	By Lemma \ref{lem1}, $D[V(Q)\cup\{v\}]$ has a good pair. Moreover, since $|Q\cup\{v\}|=n-3$, $D$ has a good pair by Lemma \ref{lem5}.
	Henceforth, we may assume that $X\cap Y= \emptyset$.
	
	If $X\cup Y=V-V(Q)$, namely $|X\cup Y|=4$, then $D$ has a good pair by Corollary \ref{cor1}.
	If $|X\cup Y|=3$, then $D$ has a good pair by Lemma \ref{lem2}.
	Therefore, we may assume that $|X\cup Y|=2$, i.e., $|X|=|Y|=1$.
	
	Set $Y=\{y\}$, $X=\{x\}$ and $V-V(Q)-X-Y=\{w_1,w_2\}$,
	and let $e_y$ (resp. $e_x$) be an arc from $Q$ to $y$ (resp. from $x$ to $Q$).
	Since $\lambda(D)\ge2$, $D$ contains $2$ arc-disjoint $(y,x)$-paths $P^1_{(y,x)}$ and $P^2_{(y,x)}$ by Lemma \ref{lem3}.
	Obviously, $(P^1_{(y,x)}\cup P^2_{(y,x)})\cap V(Q)=\emptyset$ and  $(P^1_{(y,x)}\cup P^2_{(y,x)})\cap \{w_1,w_2\}\neq \emptyset$
	(since there are no multiple arcs). W.l.o.g., assume that $P^1_{(y,x)}\cap \{w_1,w_2\}\neq \emptyset$.
	
	{\bf Case 1:} $P^2_{(y,x)}\cap \{w_1,w_2\}=\emptyset$.
	
	That is $P^2_{(y,x)}=yx$.
	Since $P^1_{(y,x)}\cup P^2_{(y,x)}$ uses at most one arc from $\{y,x\}$ to $\{w_1,w_2\}$, there is one arc
	$a_1 \notin P^1_{(y,x)}\cup P^2_{(y,x)}$ from $\{y,x\}$ to $\{w_1,w_2\}$. W.l.o.g., assume that the head of $a_1$ is $w_1$.
	Moreover, since $d^-_D(w_2)\geq 2$, there is one arc $a_2 \notin P^1_{(y,x)}\cup P^2_{(y,x)}$ with head $w_2$.
	Obviously, we can make $\{a_1\}\cup \{a_2\}\neq C_2$. Set $O_D=O_Q+e_y+P^2_{(y,x)}+a_1+a_2$.
	
	Now, we want to find $I_D$. If $P^1_{(y,x)}\supseteq \{w_1,w_2\}$, then let $I_D=P^1_{(y,x)}+e_x+I_Q$.
	If $P^1_{(y,x)}\cap \{w_1,w_2\}=w_1$, then there is one arc $a_3$
	with tail $w_2$. Note that $a_3 \neq a_1$ (since the tail of $a_1$ is in $\{y,x\}$), $a_3 \neq a_2$
	(since the head of $a_2$ is $w_2$) and obviously $a_3 \notin P^1_{(y,x)}\cup P^2_{(y,x)}$. Set $I_D=P^1_{(y,x)}+a_3+e_x+I_Q$.
	If $P^1_{(y,x)}\cap \{w_1,w_2\}=w_2$, then there is one arc $a_4 \neq a_2$
	with tail $w_1$ since $d^+_D(w_1)\geq 2$. Similarly, $a_4 \neq a_1$ and $a_4 \notin P^1_{(y,x)}\cup P^2_{(y,x)}$. Set $I_D=P^1_{(y,x)}+a_4+e_x+I_Q$.
	Therefore, $(O_D,I_D)$ is a good pair of $D$, a contradiction.\\
	
	The proof of the case when ``$P^2_{(y,x)}\cap \{w_1,w_2\}\neq \emptyset$'', which is analogous to Case 1, see Appendix.
\end{proof}

\begin{lemma}\label{lem7}
	Let $D$ be a 2-arc-strong digraph on $n$ vertices that contains a subdigraph $Q$ on $n-5$ vertices with a good pair, $X=N_D^-(Q),$ $Y=N_D^+(Q)$ and $X\cap Y=\emptyset$. If $|X|\ge 2$ or $|Y|\ge 2$, then $D$ has a good pair.
\end{lemma}
\begin{proof}
	Let $(O_Q,I_Q)$ be a good pair of $Q$.
	If $X\cup Y=V-V(Q)$, namely $|X\cup Y|=5$, then $D$ has a good pair by Corollary \ref{cor1}.
	If $|X\cup Y|=4$, then $D$ has a good pair by Lemma \ref{lem2}.
	Therefore, assume that $|X\cup Y|\leq 3$. Moreover, from the assumption that $|X|\ge 2$ or $|Y|\ge 2$, and the fact that $|X|\ge 1$ and $|Y|\ge 1$, we have $|X\cup Y|=3$. W.l.o.g., assume that $|X|=1$ and $|Y|=2$.
	
	Set $X=\{x\}$, $Y=\{y_1,y_2\}$ and $V-V(Q)-X-Y=\{w_1,w_2\}$,
	and let $e_x$ (resp. $e_{y_1}$, $e_{y_2}$) be an arc from $x$ to $Q$ (resp. from $Q$ to $y_1$, $y_2$).
	Let $B^+=O_Q+e_{y_1}+e_{y_2}$ and $B^-=I_Q+e_x$.
	Since $\lambda(D)\ge2$, by the definition of arc-connectivity and Lemma \ref{lem4}, $D$ contains two arc-disjoint in-branchings rooted at $x$, denoted by $B^-_1$ and $B^-_2$. Since $X=N_D^-(Q)$ and $Y=N_D^+(Q)$, clearly
	the restriction of $B^-_1, B^-_2$ to $V-V(Q)$ are still two arc-disjoint in-branchings rooted at $x$.
	We denote them by $\hat{B}^-_1$ and $\hat{B}^-_2$.
	
	Now, in $\hat{B}^-_i$ ($i=1$ or 2), if there exists a $(y_1,x)$-path or a $(y_2,x)$-path containing $w_j$ ($j=1$ or 2), then we let $f_i(w_j)=1$; otherwise, let $f_i(w_j)=0$. Note that, in any in-branching rooted at $x$, there is a unique $(s,x)$-path for any vertex $s$. Hence, for $\hat{B}^-_i$, if $f_i(w_1)=1$ and $f_i(w_2)=0$, then it is impossible that $w_2$ is the out-neighbor
	of $y_1,y_2$ or $w_1$. So $w_2$ must be a leaf in $\hat{B}^-_i$.
	
	Next, we consider the following cases:
	
	{\bf Case 1:} $f_i(w_1)+f_i(w_2)=2$ for some $i\in[2]$.
	
	W.l.o.g., assume that $i=1$, i.e., $f_1(w_1)=f_1(w_2)=1$. Then, $\hat{B}^-_1$ contains
	a $P_{(y_i,x)}$ containing $w_1$ and a $P_{(y_j,x)}$ containing $w_2$, where $i,j\in \{1,2\}$.
	Since $P_{(y_i,x)}\ne P_{(y_j,x)}$, $i\ne j$. Obviously,
	$B^++P_{(y_i,x)}+P_{(y_j,x)}$ contains an out-branching of $D$, denoted by $O_D$. Now, set $I_D=\hat{B}^-_2+B^-$. Since
	$P_{(y_i,x)}\cup P_{(y_j,x)}\subseteq \hat{B}^-_1$ and $\hat{B}^-_1$ and $\hat{B}^-_2$ are arc-disjoint, $(O_D,I_D)$ is a good pair.
	
	{\bf Case 2:} $f_i(w_1)+f_i(w_2)=1$ for any $i\in[2]$.
	
	{\bf Subcase 2.1:} $f_1(w_i)=f_2(w_i)=1$ and $f_1(w_{3-i})=f_2(w_{3-i})=0$, where $i\in[2]$.
	
	W.l.o.g., assume that $i=1$. Since $f_1(w_1)=1$ and $f_1(w_2)=0$,
	$w_2$ is a leaf of $\hat{B}^-_1$. Similarly, $w_2$ is also a leaf of $\hat{B}^-_2$. Thus, $\hat{B}^-_1\cup \hat{B}^-_2$ does not use any in-arc to $w_2$. Let $a$ be an arc with head $w_2$ and then $a\notin \hat{B}^-_1\cup \hat{B}^-_2$.
	Since $f_1(w_1)=1$, $\hat{B}^-_1$ contains a $P_{(y_j,x)}$ containing $w_1$ ($j=1$ or $2$).
	Now, $D$ has a good pair $(O_D,I_D)$ with $O_D=B^++P_{(y_j,x)}+a$ and $I_D=\hat{B}^-_2+B^-$.
	
	\vspace{2mm}
	Discussions of the subcase when ``$f_1(w_i)=f_2(w_{3-i})=1$ and $f_1(w_{3-i})=f_2(w_i)=0$, where $i\in[2]$'' and cases when ``$f_i(w_1)+f_i(w_2)=1$ and $f_{3-i}(w_1)+f_{3-i}(w_2)=0$, where $i=1$ or $2$'' and ``$f_i(w_1)+f_i(w_2)=0$ for any $i\in[2]$'' are given in Appendix.
\end{proof}

\begin{lemma}\label{lem8}
	Let $D$ be a 2-arc-strong digraph on $n$ vertices that contains a subdigraph $Q$ on $n-6$ vertices with a good pair. Let $X=N_D^-(Q)$ and $Y=N_D^+(Q)$ with $X\cap Y=\emptyset$. If $|X|=|Y|=2$ and at most one of $X$ and $Y$ is an independent set, then $D$ has a good pair.
\end{lemma}
\begin{proof}
	Let $D=(V,A)$ and $W=V-X-Y-V(Q)=\{w_1,w_2\}$. Set $X=\{x_1,x_2\}$ and $Y=\{y_1,y_2\}$. By contradiction, suppose that $D$ has no good pair. Assume that $(B_Q^+,B_Q^-)$ is a good pair of $Q$. Let $B^+$ be an out-tree containing $B_Q^+$ and an in-arc of $y_j$ from $Q$ for any $j\in[2]$. Let $B^-$ be an in-tree contraining $B_Q^-$ and an out-arc of $x_i$ into $Q$ for any $i\in[2]$. By construction, $B^+$ and $B^-$ are arc-disjoint.
	
	\begin{description}
		\item[Claim \ref{lem8}.1] {\it None of the following holds:
			\begin{enumerate}
				\item $|(Y,X)_D|=4$;
				\item $X$ is not independent and there exists a vertex in $Y$ which dominates each vertex in $X$.
				Analogously $Y$ is not independent and there exists a vertex in $X$ which is dominated by each vertex in $Y$;
				\item Both $X$ and $Y$ are not independent, say $x_ix_{3-i},y_jy_{3-j}\in A$, and $y_jx_i,y_{3-j}x_{3-i}\in A$, $i,j\in[2].$
			\end{enumerate}
	}
	\end{description}
	{\it Proof.}
	We will show that for each of the three cases $D$ has a good pair.
	\begin{enumerate}
		\item Since $|(Y,X)_D|=4$, for each $y\in Y$, $y$ dominates each vertex in $X$. Thus, $(B^++y_1x_1+y_2x_2,B^-+y_1x_2+y_2x_1)$ is a good pair of $D\setminus W$, which implies that $D$ has a good pair by Lemma~\ref{lem1}.
		\item By the digraph duality, it suffices to prove the first part.
		Assume that $x_1x_2\in A$ and $y_1$ dominates each vertex in $X$. Then $(B^++y_1x_1x_2,B^-+y_1x_2)$ is a good pair of $D\setminus \{W\cup \{y_2\}\}$, which implies that $D$ has a good pair by Lemma~\ref{lem5}.
		\item W.l.o.g., assume that $x_1x_2,y_1y_2,y_1x_1,y_2x_2\in A$. Then $(B^++y_1x_1x_2,B^-+y_1y_2x_2)$ is a good pair of $D\setminus W$, which implies that $D$ has a good pair by Lemma~\ref{lem1}. 	\hfill $\lozenge$
	\end{enumerate}
	
	Since at most one of $X$ and $Y$ is independent and $|X|=|Y|=2$, it suffices to consider the case when $X$ is not independent by the digraph duality. W.l.o.g., assume $x_1x_2\in A$. We now prove the following claims.
	
	\begin{description}
		\item[Claim \ref{lem8}.2] {\it No vertex in $W$ dominates both vertices of $X$. Analogously, if $Y$ is not an independent set, then no vertex in $W$ is dominated by both vertices in $Y$.}
	\end{description}
	{\it Proof.}
	By the digraph duality, it suffices to prove the first part.
	Suppose that $w_1$ dominates both vertices in $X$. Let $e_1=w_1x_2$.
	
	First assume that there exists an arc from $Y$ to $w_2$, say $e_2$. Set $D'=D-\{e_1,e_2\}$, $X'=X\cup \{w_1\}$ and $Y'=Y\cup \{w_2\}$. There is only one initial strong component in $D'[X']$, say $X_1'$. Note that $d_{Y'}^-(X_1')\ge2$. Let $Y_j'$ be the terminal strong components in $D'[Y']$, $j\in[a]$. Note that $a\le3$. Then $d_{X'}^+(Y_j')\ge2$ for all $j\in[a]$ except at most one, say $Y_1'$, has $d_{X'}^+(Y_1')=1$. Note that $e_2\in (Y_1',w_2)_D$. By Proposition~\ref{prop2}, we get arc-disjoint $\mathcal{P}_{X'}$ and $\mathcal{P}_{Y'}$ and $T_{X'}$, $T_{Y'}$. Then $D$ has a good pair $(B^++e_2+\mathcal{P}_{X'}+T_{X'},B^-+e_1+\mathcal{P}_{Y'}+T_{Y'})$, a contradiction.
	
	Henceforth we may assume that $(Y,w_2)_D=\emptyset$. Now $d_{X\cup \{w_1\}}^-(w_2)\ge2$. Let $e_2$ be an arbitrary out-arc of $w_2$. Set $D'=D-\{e_1,e_2\}$ and $X'=X\cup W$. There is only one initial strong component in $D'[X']$, say $X_1'$. Note that $d_{Y}^-(X_1')\ge2$. For any terminal strong component $Y_j$ in $D'[Y]$, $d_{X'}^+(Y_j)\ge2$, where $j\in[a]$ and $a\le2.$ By Proposition~\ref{prop2}, we get arc-disjoint $\mathcal{P}_{X'}$ and $\mathcal{P}_{Y}$, and $T_{X'}$, $T_{Y}$. Then $D$ has a good pair $(B^++\mathcal{P}_{X'}+T_{X'},B^-+e_1+e_2+\mathcal{P}_Y+T_Y)$, a contradiction.
	\hfill $\lozenge$
	
	\begin{description}
		\item[Claim \ref{lem8}.3] {\it If $Y$ is not independent, then at least one of $(w_k,X)_D$ and $(Y,w_{3-k})_D$ is empty, for any $k\in[2]$.}
	\end{description}
	{\it Proof.}
	W.l.o.g., assume that $y_1y_2\in A$ and $k=1$.
	Suppose that neither $(w_1,X)_D$ nor $(Y,w_2)_D$ is empty.
	
	{\bf Case 1:} $y_2w_2\in A$ ($w_1x_1\in A$).
	
	By the digraph duality, it suffices to prove the case of $y_2w_2\in A$. We distinguish several subcases as follows.
	
	{\bf Subcase 1.1: $(Y,x_1)_D\neq \emptyset.$}
	
	Assume $y_jx_1\in (Y,x_1)_D$ and $w_1x_i\in (w_1,X)_D$ as $(w_1,X)_D\neq \emptyset$, where $i,j\in[2]$.
	
	First assume that $(w_2,X\cup \{w_1\})_D\neq \emptyset$.
	Set $w_2w_2^+\in (w_2,X\cup \{w_1\})_D$. Then $(B^++y_jx_1x_2+w_1^-w_1+w_2^-w_2,B^-+w_1x_i+y_1y_2w_2w_2^+)$ is a good pair of $D$, where $w_1^-\neq w_2$ and $w_2^-\neq y_2$ as $\lambda(D)\ge2$, a contradiction.
	
	Next assume that $(w_2,X\cup \{w_1\})_D=\emptyset$.
	Namely $w_2y_1,w_2y_2\in A$. Since $\lambda(D)\ge2$, there exists an arc $e\neq y_jx_1$ which is from $Y\cup \{w_2\}$ to $X\cup \{w_1\}$. Now we find an out-branching of $D$ as $O=B^++y_jx_1x_2+w_1^-w_1+w_2^-w_2$, where $w_1^-w_1\neq e$ and $w_2^-\neq y_2$ as $\lambda(D)\ge2$. Note that $B^-+w_1x_i+y_1y_2w_2y_1+e$ contains an in-branching $I$ of $D$. Then $(O,I)$ is a good pair of $D$, a contradiction.
	
	\vspace{2mm}
	Discussions of subcases when ``$(Y,x_1)_D=\emptyset$ but $(Y,x_2)_D\neq \emptyset$'' and ``$(Y,X)_D=\emptyset$'' and the case when ``$y_2w_2,w_1x_1\notin A$'' are given in Appendix.
	\hfill $\lozenge$
	
	\begin{description}
		\item[Claim \ref{lem8}.4] {\it If $Y$ is an independent set and $(w_k,X)_D$ is not empty, then $(Y,w_{3-k})_D= \emptyset$ for any $k\in[2]$.}
	\end{description}
	{\it Proof.}
	W.l.o.g., assume that $k=1$. Set $w_1x_i\in A$, where $i\in[2]$. Suppose $(Y,w_2)_D\neq \emptyset$. We distinguish the following two cases.
	
	{\bf Case 1: $|(Y,w_2)_D|=2$.}
	
	That is $y_1w_2,y_2w_2\in A$. Let $y_jy_j^+$ be an out-arc of $y_j$ which is different from $y_jw_2$, for any $j\in[2]$. Note that $y_j^+\in X\cup \{w_1\}$ as $Y$ is independent.
	
	{\bf Subcase 1.1:} $w_2x_1\in A$.
	
	Let $w_2w_2^+$ be an out-arc of $w_2$ which is different from $w_2x_1$.
	
	First assume that $w_2^+\in X\cup \{w_1\}$. Since $\lambda(D)\ge2$, $w_1$ has an in-neighbour $w_1^-\neq w_2$ and there is at least one vertex in $Y$ which is not $w_1^-$, w.l.o.g., say $y_1\neq w_1^-$. Then $(B^++y_1w_2x_1x_2+w_1^-w_1,B^-+w_1x_i+y_2w_2w_2^++y_1y_1^+)$ is a good pair of $D$, a contradiction.
	
	Next assume that $w_2^+\in Y$, w.l.o.g., say $w_2^+=y_1$, i.e., $w_2y_1\in A$. Then $(B^++y_1w_2x_1x_2+w_1^-w_1,B^-+w_1x_i+y_2w_2y_1y_1^+)$ is a good pair of $D$, where $w_1^-\neq y_1$ as $\lambda(D)\ge2$, a contradiction.
	
	\vspace{2mm}
	Discussions of subcases when ``$w_2x_1\notin A$ but $(Y,x_1)_D\neq \emptyset$'' and ``$(Y\cup \{w_2\},x_1)_D= \emptyset$'' and the case when ``$|(Y,w_2)_D|=1$'' are given in Appendix.
	\hfill $\lozenge$
	
	\begin{description}
		\item[Claim \ref{lem8}.5] {\it $(W,X)_D=\emptyset$. Moreover, if $Y$ is not an independent set, then $(Y,W)_D=\emptyset$.}
	\end{description}
	{\it Proof.}
	By the digraph duality, it suffices to prove that $(W,X)_D=\emptyset$.
	Suppose $(W,X)_D\neq\emptyset$. W.l.o.g., assume that $(w_1,X)_D\neq \emptyset$, i.e., $w_1x_i\in A$, for some $i\in[2]$. Note that $(Y,w_2)_D=\emptyset$ by Claims \ref{lem8}.3 and \ref{lem8}.4.
	
	{\bf Case 1:} $(w_2,Y)_D\neq \emptyset$.
	
	Set $w_2y_j\in A$, where $j\in[2]$. Since $N^-(X\cup W)=Y$ and $\lambda(D)\ge2$, any initial strong component of $D[X\cup W]$ has at least two in-arcs from $Y$. Set $D'=D-w_1x_i-w_2y_j$. Now any initial strong component of $D'[X\cup W]$ has at least two in-arcs from $Y$, except at most one initial strong component, say $X'_1$, has exactly one in-arc from $Y$. Note that $x_i\in X'_1$ but $w_1\notin X'_1$ and any terminal strong component of $D'[Y]$ has at least two out-arcs to $X\cup W$. By Proposition~\ref{prop2}, we get $\mathcal{P}_{X\cup W},T_{X\cup W}$ and $\mathcal{P}_{Y},T_{Y}$ in $D$. Since $(Y,w_2)_D=\emptyset$, $(B^++\mathcal{P}_{X\cup W}+T_{X\cup W},B^-+\mathcal{P}_{Y}+T_{Y}+w_1x_i+w_2y_j)$ is a good pair of $D$, a contradiction.
	
	{\bf Case 2:} $(w_2,Y)_D=\emptyset$.
	
	That is $|(w_2,X)_D|\ge1$. By Claim \ref{lem8}.2, $|(w_2,X)_D|\le1$, i.e., $w_2w_1\in A$ and $|(w_2,X)_D|=1$. Interchange $w_1$ and $w_2$, likewise $(Y,w_1)_D=\emptyset$ by Claims~\ref{lem8}.3 and \ref{lem8}.4. Now $(Y,W)_D=\emptyset$. By Cliam~\ref{lem8}.1(2), $|(y_j,X)_D|\le1$ for any $j\in[2]$, namely $y_1y_2,y_2y_1\in A$ and $|(y_1,X)_D|=|(y_2,X)_D|=1$. We also get that $|(Y,x_i)_D|\le1$, for any $i\in[2]$. This implies that $D$ has a good pair by Cliam~\ref{lem8}.1(3).
	\hfill $\lozenge$
	
	\vspace{2mm}
	Now we are ready to finish the proof of Lemma~\ref{lem8}.
	By Claim~\ref{lem8}.1(2), $|(y_j,X)_D|\le1$ for any $j\in[2]$, i.e. $|(Y,X)_D|\le2$. Since $(W,X)_D=\emptyset$, $D[X]=C_2$ and $y_jx_1,y_{3-j}x_2\in A$ for some $j\in[2]$. Note that $Y$ is an independent set by Cliam~\ref{lem8}.1(3). W.l.o.g., assume that $j=1$, i.e., $y_1x_1,y_2x_2\in A$, which implies that $y_2x_1,y_1x_2\notin A$.
	Since $(W,X)_D=\emptyset$, $|(w_2,Y)_D|\ge1$ as $\lambda(D)\ge2$. W.l.o.g., assume that $w_2y_2\in A$.
	
	First assume $y_1w_k,w_ky_2\in A$, for some $k\in[2]$. W.l.o.g., assume $k=1$. Then $(B^++y_1x_1x_2+w_2^-w_2+w_1^-w_1,B^-+y_1w_1y_2x_2+w_2w_2^+)$ is a good pair of $D$, where $w_2^-,w_2^+\neq w_1$ and $w_1^-\neq y_1$, a contradiction.
	
	Next assume $y_1w_2\notin A$. As $\lambda(D)\ge2$, $y_1w_1\in A$, likewise $w_1y_2\notin A$. It follows that $w_1w_2,w_1y_1\in A$. Then $(B^++y_1x_1x_2+w_2^-w_2+w_1^-w_1,B^-+y_1w_1w_2y_2x_2)$ is a good pair of $D$, where $w_1^-\neq y_1$ and  $w_2^-\neq w_1$ as $\lambda(D)\ge2$, a contradiction.
	
	This completes the proof of Lemma~\ref{lem8}.	
\end{proof}

\begin{lemma}\label{lem8-1}
	Let $D=(V,A)$ be a 2-arc-strong digraph on $n$ vertices that contains a subdigraph $Q$ on $n-6$ vertices with a good pair. Set $X=N_D^-(Q)=\{x_1,x_2\}$ and $Y=N_D^+(Q)=\{y_1,y_2\}$ with $X\cap Y=\emptyset$, and $W=V-X-Y-V(Q)=\{w_1,w_2\}$. If $X,Y$ are both independent sets, then $D$ has a good pair except for the case below:
	
	{\bf ($\ast$)} $(Y,X)_D=\{y_jx_i,y_{3-j}x_{3-i}\}$ for some $i,j\in[2]$, $D[W]=C_2$ and $N_W^+(y_j)\cap N_W^+(y_{3-j})=N_W^-(x_i)\cap N_W^-(x_{3-i})=\emptyset$ while $N_W^+(y_j)\cap N_W^-(x_i)\neq \emptyset$ and $N_W^+(y_{3-j})\cap N_W^-(x_{3-i})\neq \emptyset$.
\end{lemma}
\begin{proof}
	Suppose that $D$ has no good pair. Follow the definitions of $B_Q^+,B_Q^-,B^+$ and $B^-$ in the proof of Lemma~\ref{lem8}. Observe that Claim~\ref{lem8}.1 still holds here. We distinguish four cases below depending on $|(Y,X)_D|$. By Claim~\ref{lem8}.1, $|(Y,X)_D|\le3$.
	
	{\bf Case 1:} $|(Y,X)_D|=2$.
	
	{\bf Subcase 1.1:} $(Y,X)_D=\{y_jx_1,y_jx_2\}$, for some $j\in[2]$.
	
	W.l.o.g., assume that $j=1$. It implies that $y_2$ dominates both vertices in $W$. Note that there exists an arc from $W$ to $x_i$, say $e_{x_i}$, where $i\in[2]$.
	
	If $e_{x_1}$ is adjacent to $e_{x_2}$, say $w_1$ is the common vertex, then $(B^++y_1x_1+y_2w_1x_2+w_2^-w_2,B^-+y_1x_2+w_1x_1+y_2w_2w_2^+)$ is a good pair of $D$, where $w_2^-,w_2^+\neq y_2$ as $\lambda(D)\ge2$, a contradiction.
	
	Hence $e_{x_1}$ is non-adjacent to $e_{x_2}$, w.l.o.g., say $w_1x_2,w_2x_1\in A$.
	If $D[W]\neq C_2$, w.l.o.g., say $w_1w_2\notin A$, then $(B^++y_1x_1+y_2w_1x_2+w_2^-w_2,B^-+y_1x_2+y_2w_2x_1+w_1w_1^+)$ is a good pair of $D$, where $w_2^-\neq y_2$ and $w_1^+\neq x_2$ as $\lambda(D)\ge2$, a contradiction.
	If $D[W]=C_2$, then $(B^++y_1x_2+y_2w_1w_2x_1,B^-+y_1x_1+y_2w_2w_1x_2)$ is a good pair of $D$, a contradiction.
	
	By the digraph duality, we also get a contradiction when $(Y,X)_D=\{y_1x_i,y_2x_i\}$, for some $i\in[2]$.
	
	Discussions of the subcase when ``$(Y,X)_D=\{y_jx_i,y_{3-j}x_{3-i}\}$, for some $i,j\in[2]$'' and cases when ``$|(Y,X)_D|=3$'', ``$|(Y,X)_D|=1$'' and ``$|(Y,X)_D|=0$'' are given in Appendix.
\end{proof}

We use $D \supseteq E_3$ ($D\nsupseteq E_3$) to denote that $D$ contains an arbitrary orientation (no orientation) of $E_3$ as a subdigraph. ($E_3$ is a mixed graph and only the two edges are to be oriented.) $E_3$ is shown in Figure~\ref{fig4}.

\begin{figure}[!htpb]
	\centering\includegraphics[scale=0.5]{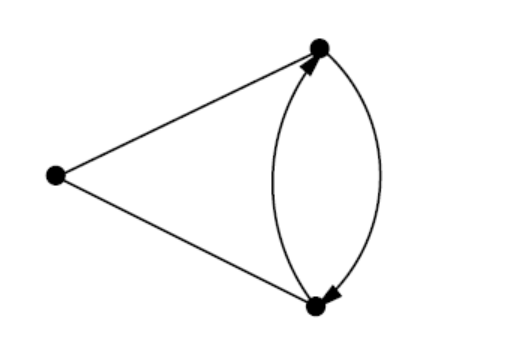}
	\caption{$E_3$.}
	\label{fig4}
\end{figure}

\begin{lemma}\label{lem8-2}
	Let $D=(V,A)$ be a 2-arc-strong digraph on $n$ vertices that contains a subdigraph $Q$ on $n-6$ vertices with a good pair. Set $X=N_D^-(Q)=\{x_1,x_2\}$ and $Y=N_D^+(Q)=\{y_1,y_2\}$ with $X\cap Y=\emptyset$, and $W=V-X-Y-V(Q)=\{w_1,w_2\}$. If $n=8$ or 9 and $X,Y$ are both independent, then $D$ has a good pair.
\end{lemma}
\begin{proof}
	By contradiction, suppose that $D$ has no good pair. Follow the definitions of $B_Q^+,B_Q^-,B^+$ and $B^-$ in the proof of Lemma~\ref{lem8}. Observe that Claim~\ref{lem8}.1 still holds here.
	
	From Lemma~\ref{lem8-1}, it suffices to consider the case {\bf ($\ast$)}. Recall that $(Y,X)_D=\{y_jx_i,y_{3-j}x_{3-i}\}$ for some $i,j\in[2]$, $D[W]=C_2$ and $N_W^+(y_j)\cap N_W^+(y_{3-j})=N_W^-(x_i)\cap N_W^-(x_{3-i})=\emptyset$ while $N_W^+(y_j)\cap N_W^-(x_i)\neq \emptyset$ and $N_W^+(y_{3-j})\cap N_W^-(x_{3-i})\neq \emptyset$. W.l.o.g., assume that $i=j=1$ and $w_k\in N_W^+(y_k)\cap N_W^-(x_k)$ for any $k\in[2]$. Note that $|Q|$ is 2 or 3.
	
	Suppose $|Q|=2$, then $Q=C_2$. Set $V(Q)=\{q_1,q_2\}$ and $x_1q_1,x_2q_2\in A$.
	If $q_1y_2,q_2y_1\in A$, then $D$ has a good pair $(B_{x_1}^+,B_{x_1}^-)$ as $B_{x_1}^+=x_1q_1q_2y_1w_1w_2x_2+y_2^-y_2$ and $B_{x_1}^-=x_2q_2q_1y_2w_2w_1x_1+y_1x_1$, where $y_2^-\neq q_1$ as $\lambda(D)\ge2$, a contradiction.
	If $q_1y_1,q_2y_2\in A$, then $D$ has a good pair $(B_{w_2}^+,B_{x_1}^-)$ as $B_{w_2}^+=w_2w_1x_1q_1q_2y_2x_2+y_1^-y_1$ and $B_{x_1}^-=w_1w_2x_2q_2q_1y_1x_1+y_2w_2$, where $y_1^-\neq q_1$ as $\lambda(D)\ge2$, a contradiction.
	
	Now $|Q|=3$. By Proposition~\ref{prop3}, $|E(Q)|\ge4$. Set $V(Q)=\{q_1,q_2,q_3\}$. Note that $Q\supseteq E_3$ or $Q$ contains a bidigon as a subdigraph, i.e., $C_2\subset Q$. W.l.o.g., assume $C_2=q_1q_2q_1$. Set $B^+=w_1w_2x_2+w_1x_1$ and $B^-=y_1w_1+y_2w_2w_1$. Since $X=N_D^-(Q)$ and $Y=N_D^+(Q)$, there exists an arc from $x_i$ to $Q$ and an arc from $Q$ to $y_j$, say $x_iq_{x_i}$ and $q_{y_j}y_j$, respectively, where $i,j\in[2]$.
	
	\begin{description}
		\item[Claim~\ref{lem8-2}.1] If $Q$ has a good pair $(B_{q_{x_i}}^+,B_{q_{y_j}}^-)$, then $N^+(x_i)=\{q_{x_i},y_j\}$ and $N^-(y_j)=\{q_{y_j},x_i\}$, where $i,j\in[2]$.
	\end{description}
	{\it Proof.}
	Suppose that $x_i$ has an out-neighbour $x_i^+\notin \{q_{x_i},y_j\}$ or $y_j$ has an in-neighbour $y_j^-\notin \{q_{y_j},x_i\}$, then $(B^++x_iq_{x_i}+B_{q_{x_i}}^++y_j^-y_j+q_{y_{3-j}}y_{3-j},B^-+q_{y_j}y_j+B_{q_{y_j}}^-+x_ix_i^++x_{3-i}q_{x_{3-i}})$ is a good pair of $D$, a contradiction.
	\hfill $\lozenge$
	
	\vspace{2mm}
	We distinguish several cases as follows.
	
	{\bf Case 1:} $Q$ contains a bidigon.
	
	Set $q_2q_3,q_3q_2\in A$.
	
	{\bf Subcase 1.1:} $x_iq_1\in A$, for some $i\in[2]$.
	
	W.l.o.g., assume $i=1$.
	
	{\bf A.} $(q_1,Y)_D\neq \emptyset$.
	
	Assume $q_1y_j\in A$, $j\in[2]$. Note that $Q$ has a good pair $B_{q_1}^+=q_1q_2q_3$ and $B_{q_1}^-=q_3q_2q_1$. This implies that $N^+(x_1)=\{q_1,y_j\}$ and $N^-(y_j)=\{q_1,x_1\}$ by Claim~\ref{lem8-2}.1.
	
	We first show $q_3y_{3-j},x_2q_3\in A$.
	If $q_3y_{3-j}\notin A$, i.e., $q_3q_1\in A$, then $Q$ has a good pair $(B_{q_i}^+,B_{q_1}^-)$ for any $i\in[3]$ (see Figure~\ref{fig5}). It follows that $(B^++x_1y_j+x_2q_{x_2}+B_{q_{x_2}}^++q_{y_{3-j}}y_{3-j},B^-+x_1q_1y_j+B_{q_1}^-+x_2x_2^+)$ is a good pair of $D$, where $x_2^+\neq q_{x_2}$ as $\lambda(D)\ge2$, a contradiction. Hence $q_3y_{3-j}\in A$.
	If $x_2q_3\notin A$, namely $q_1q_3\in A$, then $(B^++x_1q_1y_j+q_1q_3q_2+y_{3-j}^-y_{3-j},B^-+x_1y_j+q_1q_2q_3y_{3-j}+x_2q_{x_2})$ is a good pair of $D$, where $y_{3-j}^-\neq q_3$ as $\lambda(D)\ge2$, a contradiction.
	
	\begin{figure}[!htpb]
		\centering\includegraphics[scale=0.5]{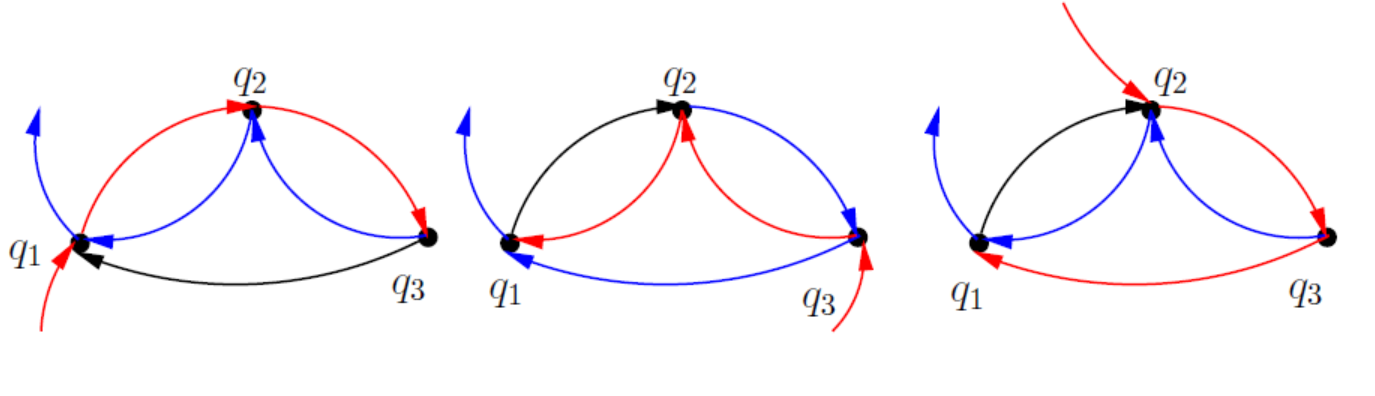}
		\caption{Good pairs $(B_{q_i}^+,B_{q_1}^-)$ of $Q$, for any $i\in[3]$.}
		\label{fig5}
	\end{figure}
	
	Now $q_3y_{3-j},x_2q_3\in A$. Note that $Q$ has a good pair $B_{q_3}^+=q_3q_2q_1$ and $B_{q_3}^-=q_1q_2q_3$. This implies that $N^+(x_2)=\{q_3,y_{3-j}\}$ and $N^-(y_{3-j})=\{q_3,x_2\}$ by Claim~\ref{lem8-2}.1.
	If $j=1$, then $D$ has a good pair $(B_{x_1}^+,B_{x_1}^-)$ with $B_{x_1}^+=x_1q_1q_2q_3+x_1y_1w_1w_2x_2y_2$ and $B_{x_1}^-=w_2w_1x_1+y_2x_2q_3q_2q_1y_1x_1$, a contradiction.
	If $j=2$, then $D$ has a good pair $(B_{x_1}^+,B_{x_1}^-)$ with $B_{x_1}^+=x_1q_1q_2q_3+x_1y_2x_2y_1w_1w_2$ and $B_{x_1}^-=x_1q_3y_1x_1+q_2q_1y_2w_2w_1x_1$, a contradiction.
	
	\vspace{2mm}
	Therefore, $q_1q_3\in A$.
	Discussions of ``$(q_1,Y)_D=\emptyset$ and $(q_3,Y)_D\neq \emptyset$'' and ``$N^+(q_1)\cup N^+(q_3)\subseteq Q$'' and the subcase when ``$N(q_i)\subset Q$, for any $i\in\{1,3\}$'' are given in Appendix.
	
	Case 1 implies that $D$ has a good pair when $Q$ contains a bidigon as a subdigraph.
	
	{\bf Case 2:} $Q\supseteq E_3$.
	
	First assume $q_1q_3,q_2q_3\in A$. Since $Q$ has no bidigon as subdigraph and $\lambda(D)\ge2$, $q_3y_1,q_3y_2\in A$ and $(X,q_1)_D\neq \emptyset$. Set $x_iq_1\in A$. Note that $Q$ has a good pair $B_{q_1}^+=q_1q_2q_3$ and $B_{q_3}^-=q_2q_1q_3$. By Claim~\ref{lem8-2}.1, $N^+(x_i)=\{q_1,y_1\}=\{q_1,y_2\}$, a contradiction.
	
	The case of $q_3q_1,q_3q_2\in A$ is analogous.
	Hence $q_2q_3,q_3q_1\in A$. We distinguish several subcases as follows.

	{\bf Subcase 2.1:} $x_iq_1\in A$, for some $i\in[2]$.
	
	W.l.o.g., assume $i=1$. Since $Q$ has no bidigon, $q_1q_3\notin A$, which implies that $(q_1,Y)_D\neq \emptyset$. Set $q_1y_j\in A,~j\in[2]$.
	Note that $Q$ has a good pair $B_{q_1}^+=q_1q_2q_3$ and $B_{q_1}^-=q_3q_1+q_2q_1$. This implies that $N^+(x_1)=\{q_1,y_j\}$ and $N^-(y_j)=\{q_1,x_1\}$ by Claim~\ref{lem8-2}.1. By Case 1, $x_2q_3,x_2q_2,q_3y_{3-j}\in A$. Then $(B^++x_1y_j+x_2q_3q_1q_2+y_{3-j}^-y_{3-j},B^-+x_1q_1y_j+x_2q_2q_1+q_3y_{3-j})$ is a good pair of $D$, where $y_{3-j}^-\neq q_3$ as $\lambda(D)\ge2$, a contradiction.
	
	\vspace{2mm}
	By the digraph duality, we also get a contradiction when $q_2y_j\in A$, where $j\in[2]$. This implies that $(X,q_1)_D=(q_2,Y)_D=\emptyset$.
	The proof of ``$x_iq_2\in A$, for some $i\in[2]$'' can be found in Appendix.
\end{proof}

\begin{prop}[\cite{BBY}]\label{prop5}
	A digraph $D$ has an out-branching (resp. in-branching) if and only if it has precisely one initial (resp. terminal) strong component. In that case every vertex of the initial (resp. terminal) strong component can be the root of an out-branching (resp. in-branching) in $D.$
\end{prop}

We use $T_x^+$ (resp. $T_x^-$) to denote an out-tree (resp. in-tree) rooted at $x$.

\begin{prop}\label{prop6}
	Let $D$ be an  oriented graph on $n$ vertices. Let $P_D=x_1x_2\ldots x_n$ be the Hamilton dipath of $D$ and $D'=D-A(P)$. Assume that there are exactly two non-adjacent strong components $I_1$ and $I_2$ in $D'$. Set $q\in \{2,3,n-1,n\}$. If for some $q$, $x_{q-1}$ and $x_q$ are respectively in $I_1$ and $I_2$, then $D$ has a good pair.
\end{prop}
\begin{proof}
	W.l.o.g., assume that $x_{q-1}\in I_1$ and $x_q\in I_2$. Since $I_i$ is strong, $\delta^0(I_i)\ge 1$, for any $i\in[2]$.
	
	First assume $q\in\{n-1,n\}$. Let $x$ be an in-neighbour of $x_q$ in $I_2$. We get an out-branching of $D$ as $B_{x_1}^+=P_D-x_{q-1}x_q+xx_q$. Then we will show that there is an in-branching $B_x^-$ in $D-A(B_{x_1}^+)$. Since $I_2$ is strong, $I_2-xx_q$ is connected and has only one terminal srong component which contains $x$. This implies that there is an in-branching $T_x^-$ in $I_2-xx_q$. Note that there exists an in-branching $T_{x_{q-1}}^-$ in $I_1$, as $I_1$ is strong. Then $B_x^-=T_x^-+x_{q-1}x_q+T_{x_{q-1}}^-$, which implies that $(B_{x_1}^+,B_x^-)$ is a good pair of $D$.
	
	Now we assume $q\in\{2,3\}$. Let $y$ be an out-neighbour of $x_{q-1}$ in $I_1$. We get an in-branching of $D$ as $B_{x_n}^-=P_D-x_{q-1}x_q+x_{q-1}y$. Then we will show that there is an out-branching $B_y^+$ in $D-A(B_{x_n}^-)$. Since $I_1$ is strong, $I_1-x_{q-1}y$ is connected and has only one initial srong component which contains $y$. This implies that there is an out-branching $T_y^+$ in $I_1-x_{q-1}y$. Note that there exists an out-branching $T_{x_q}^+$ in $I_2$, as $I_2$ is strong. Then $B_y^+=T_y^++x_{q-1}x_q+T_{x_q}^+$. So, $(B_y^+,B_{x_n}^-)$ is a good pair of $D$.
\end{proof}

\begin{prop}\label{prop7}
	Let $D$ be a 2-arc-strong oriented graph on at least seven vertices. Then $D$ has  a dipath $P_6.$
\end{prop}
\begin{proof}
	Suppose that there is no $P_6$ in $D$. Assume that $P_t$ is the longest dipath in $D$, then $t\ge4$, as there is no digon in $D$ and $\lambda(D)\ge2$. Observe that there is no $C_t$ in $D$, otherwise $D$ has a longer dipath $P_{t+1}$.
	
	First assume that $t=4$ and set $P_4=x_1x_2x_3x_4$. Since $d_D^+(x_4)\ge 2$ and $D$ has no digon, the out-neighbourhood of $x_4$ either contains $x_1$ or contains a vertex in $V-V(P_4)$. This implies that there is a $P_5$ in $D$, a contradiction.
	
	Henceforth we may assume that $t=5$ and set $P_5=x_1x_2x_3x_4x_5$. Since $\lambda(D)\ge2$, $d_D^+(x_5)\ge 2$ and $d_D^-(x_1)\ge 2$. Then we get $N_D^+(x_5)=\{x_2,x_3\}$ and $N_D^-(x_1)=\{x_3,x_4\}$, as $P_5$ is the longest dipath in $D$ and $D$ has no digon. Observe that there exsits a different 4-length dipath, $x_4x_5x_3x_1x_2$, in $D$. Likewise, $N_D^+(x_2)=\{x_3,x_5\}$, which implies that $D[\{x_2,x_5\}]$ is a digon, a contradiction.	
\end{proof}

\section{Good pairs in digraphs of order 7}\label{7}
\begin{prop}\label{h1}
	A 2-arc-strong oriented graph $D$ on $n$ vertices has a $P_7$, where $7\le n\le9$.
\end{prop}
\begin{proof}
	The proof is similar as that in Proposition~\ref{prop7}, which will be given in Appendix.
\end{proof}

Now we are ready to prove Theorem~\ref{thm2}. For convenience, we restate it here.
\begin{description}
	\item[Theorem~\ref{thm2}.] {\it Every 2-arc-strong digraph on 7 vertices has a good pair.}
\end{description}
\begin{proof}
	Suppose that $D$ has no good pair. Let $R$ be a largest clique in $D$. By Lemma~\ref{lem5} and Proposition~\ref{prop4}, $|R|=3$. Moreover, $R$ is a tournament by Lemma~\ref{lem6} and Proposition~\ref{prop3}.
	
	\begin{description}
		\item[Claim \ref{thm2}.1] {\it $D$ is an oriented graph.}
	\end{description}
	{\it Proof.}
	Suppose that there is a digon $Q$ in $D$ with $V(Q)=\{s,t\}$. Observe that $Q$ has a good pair. Since $R$ is a tournament with three vertices, both in- and out-neibourhoods of $Q$ in $D$ have at least two vertices. This implies that $D$ has a good pair by Lemma~\ref{lem7}, a contradiction.
	\hfill $\lozenge$
	
	\vspace{2mm}
	Assume that $P_D=x_1x_2\ldots x_7$ is a Hamilton dipath of $D$ by Proposition~\ref{h1}. Set $D'=D-A(P_D)$.
	Let $I_i$ and $T_j$ respectively be the initial and terminal strong component in $D'$, where $i\in[a]$ and $j\in[b]$. Note that $a,b\ge2$ by Proposition~\ref{prop5}. Since $D$ is an oriented graph and $\lambda(D)\ge2$, $|I_i|,|T_j|\ge3$, for any $i\in[a],j\in[b]$. Thus there are only two non-adjacent strong components in $D'$, say $I_1$ and $I_2$, with $|I_1|=3$ and $|I_2|=4$. Note that $|N_{D'}^-(x_1)|\ge2$ and $|N_{D'}^+(x_7)|\ge2$ as $\lambda(D)\ge2$, which implies that $x_1,x_7\in I_2$. Moreover, $x_2,x_6\in I_1$ by Claim~\ref{thm2}.1. Then $D$ has a good pair by Proposition~\ref{prop6}.
\end{proof}

\section{Good pairs in digraphs of order 8}\label{8}

The digraph $E_3$ used in the next proposition is shown in Figure~\ref{fig4}.
\begin{prop}[\cite{BBHY}]\label{prop9}
	Let $D$ be a 2-arc-strong digraph without any subdigraph on order 4 that has a good pair. If $D$ contains an orientation $Q$ of $E_3$ as a subdigraph, then $N_D^+(Q)\cap N_D^-(Q)=\emptyset$, $|N_D^+(Q)|\ge2$ and $|N_D^-(Q)|\ge2$.
\end{prop}

\begin{prop}\label{h2-1}
	Let $D$ be a 2-arc-strong oriented graph on $n$ vertices without $K_4$ as a subdigraph, where $8\le n\le9$. If $D$ has two disjoint cycles $C^1$ and $C^2$ which cover 7 vertices, then $D$ contains a $P_8.$
\end{prop}
\begin{proof}
	Suppose that $P_7$ is the longest dipath of $D$ by Proposition~\ref{h1}. In fact there exist arcs between $C^1$ and $C^2$ from both directions, otherwise $D$ has a $P_8$ as $\lambda(D)\ge2$. W.l.o.g., assume $|C^1|\ge |C^2|$. Then $|C^1|=4$ and $|C^2|=3$. Let $C^1=x_1x_2x_3x_4x_1$, $C^2=x_5x_6x_7x_5$, $P_7=x_1x_2\ldots x_7$ and $y_j$ be the vertex in $V-V(C^1\cup C^2)$, where $j=1$ when $n=8$ and $j\in[2]$ when $n=9$. From the maximality of $P_7$ in $D$, we have the following facts.
	\begin{description}
		\item[Fact~\ref{h2-1}.1.] For any $j$, at least one of $(C^i,y_j)_D$ and $(y_j,C^{3-i})_D$ is empty for any $i\in[2]$.
		\item[Fact~\ref{h2-1}.2.] For any $j$, at least one of arcs $x_iy_j$ and $y_jx_{i+1}$ is not in $A$ for any $i\in[6]$.
		\item[Fact~\ref{h2-1}.3.] For $n=9$, let $y_jy_{3-j}\in A$. If $x_iy_j\in A$, then $y_{3-j}x_{i+1},y_{3-j}x_{i+2}\notin A$, where $j\in[2]$ and $i\in[5]$.
	\end{description}
	Since $D$ is oriented, there are at least three arcs between $y_j$ and $C^i$, for some $i$, by Fact~\ref{h2-1}.1. W.l.o.g., assume $i=1$. Note that $d_{C^1}^+(y_j)\ge1$ and $d_{C^1}^-(y_j)\ge1$. Then $N(y_j)\subset \{y_{3-j}\}\cup C^1$ when $n=9$ and $N(y_j)\subset C^1$ when $n=8$.
	
	If $y_j$ is not adjacent to $y_{3-j}$ or $n=8$, then $N^+(y_j)=\{x_1,x_2\}$ and $N^-(y_{3-j})=\{x_3,x_4\}$ by Fact~\ref{h2-1}.2, which implies that $D$ has a $P_8$ as $y_jx_1\in A$, a contradiction.
	
	Hence $n=9$ and $y_1$ is adjacent to $y_2$. W.l.o.g., assume that $y_1y_2\in A$.
	If $x_1y_1\in A$, then $N^+(y_2)=\{x_1,x_4\}$ by Fact~\ref{h2-1}.3 and $\lambda(D)\ge2$, which implies that $D$ has a Hamilton dipath as $y_2x_1\in A$, a contradiciton. Hence $x_1$ is not adjacent to $y_1$. By Fact~\ref{h2-1}.2, $N^+(y_1)=\{x_2,x_9\}$ and $N^-(y_1)=\{x_3,x_4\}$. By Fact~\ref{h2-1}.3 and the longestness of $P_7$,  $N^+(y_2)=\{x_2,x_3\}$. It implies that $D[\{x_2,x_3,y_1,y_2\}]$ is a $K_4$, a contradiction.
\end{proof}

\begin{prop}\label{h2-2}
	Let $D=(V,A)$ be a 2-arc-strong digraph on $n$ vertices without good pair, where $8\le n\le9$. If $D$ is an oriented graph without $K_4$ as a subdigraph, then $D$ has a $P_8$.
\end{prop}
\begin{proof}
	Suppose that $P=x_1x_2 \ldots x_7$ is the longest dipath in $D$ by Proposition~\ref{h1}. Let $X=V(P)$ and $Y=V-X$. We have the following fact.
	\begin{description}
		\item[Fact~\ref{h2-2}.1.] $N^-(x_1),N^+(x_7)\subset X$ and $x_7x_1\notin A$.
	\end{description}	
	We will now show the following notes.
	\begin{description}
		\item[Note~\ref{h2-2}.1] Let $B_X^+$ be an out-branching of $D[X]$, and let $B^-$ contain two disjoint in-tree $T_{r_1}^-$ and $T_{r_2}^-$ in $D[X]-A(B_X^+)$. Let $H_i=D[T_{r_i}^-]$ for any $i\in[2]$. For some $y\in Y$, if there exists an arc $r_iy$ and an out-arc of $y$ to $H_{3-i}$, then $D$ has a good pair.
	\end{description}	
	{\it Proof.} Let $e_1=r_1y_1$ and $e_2$ be an out-arc of $y_1$ to $H_2$. It is trivial when $n=8$ or $y_2y_1\notin A$ as $\lambda(D)\ge2$. It suffices to consider the case when $n=9$ and $y_2y_1\in A$. Set $B^+=B_X^++y_2y_1+y_2^-y_2$ and $B^-=T_{r_1}^-+T_{r_2}^-+e_1+e_2+y_2y_2^+$ where $y_2^-,y_2^+\neq y_1$ as $\lambda(D)\ge2$. Therefore, $(B^+,B^-)$ is a good pair of $D$.
	\begin{description}
		\item[Note~\ref{h2-2}.2] Let $B_X^+$ and $T_X^-$ respectively be an out-branching and an in-tree of $D[X]$, with $V(T_X^-)=X-v$, for arbitrary $v\in X$. If $v$ has an out-arc which is not into $B_X^+\cup T_X^-$, then $D$ has a good pair.
	\end{description}	
	{\it Proof.}
	It can be seen as a special case of Note~\ref{h2-2}.1, in which let $T_{r_1}^-=v$ and $T_{r_2}^-=T_X^-$.
	
	\vspace{2mm}
	We will distinguish several cases depending on the in-neighbours of $x_1$ in Appendix.
\end{proof}

Now we are ready to show Theorem~\ref{thm3}. For convenience, we restate it here.
\begin{description}
	\item[Theorem~\ref{thm3}.] {\it Every 2-arc-strong digraph on 8 vertices has a good pair.}
\end{description}
\begin{proof}
	Suppose that $D$ has no good pair. Let $R$ be a largest clique in $D$. By Lemma~\ref{lem6} and Proposition~\ref{prop4}, $|R|=3$.
	
	\begin{description}
		\item[Claim \ref{thm3}.1] {\it No subdigraph of $D$ of order at least 3 has a good pair.}
	\end{description}
	{\it Proof.}
	By Lemma~\ref{lem6}, it suffices to show that there is no $Q\subset D$ on 3 vertices with a good pair.
	Suppose that $Q$ has a good pair. If $Q$ is an orientation of $E_3$, then we use Lemma~\ref{lem7} to find a good pair of $D$ by Proposition~\ref{prop9}, a contradiction. Now assume that $Q$ is a bidigon. Set $V(Q)=\{x,y,z\}$ with $Q[\{x,y\}]=C_2$ and $Q[\{y,z\}]=C_2$. If there exists a vertex $w$ in $N_D^+(Q)\cap N_D^-(Q)$, then $D[Q\cup\{w\}]$ has a good pair by Lemma~\ref{lem1}. Thus $N_D^+(Q)\cap N_D^-(Q)=\emptyset$. If $N_D^-(Q)=\{w\}$, then $D[Q\cup\{w\}]$ has a good pair as $B_w^+=wzyx$ and $B_z^-=wxyz$. By symmetry, this implies that $|N_D^+(Q)|\ge2$ and $|N_D^-(Q)|\ge2$. Thus by Lemma~\ref{lem7}, $D$ has a good pair, a contradiction.
	\hfill $\lozenge$
	
	\vspace{2mm}
	By the claim above, $R$ is a tournament.
	
	\begin{description}
		\item[Claim \ref{thm3}.2] {\it $D$ is an oriented graph.}
	\end{description}
	{\it Proof.}
	Suppose that there is a digon $Q$ in $D$ with $V(Q)=\{s,t\}$. Observe that $Q$ has a good pair. Since $R$ is a tournament with 3 vertices, both in- and out-neibourhoods of $Q$ in $D$ have at least two vertices with $N_D^+(Q)\cap N_D^-(Q)=\emptyset$. This implies that $D$ has a good pair by Lemmas~\ref{lem2}, \ref{lem8}, \ref{lem8-1} and \ref{lem8-2}, and Corollary~\ref{cor1}, a contradiction.
	\hfill $\lozenge$
	
	\vspace{2mm}
	By Proposition~\ref{h2-2}, assume that $P_D=x_1x_2\ldots x_8$ is a Hamilton dipath of $D$. Set $D'=D-A(P_D)$. Let $I_i$ and $T_j$ respectively be the initial and terminal strong component in $D'$, where $i\in[a]$ and $j\in[b]$. Note that $a,b\ge2$ by Proposition~\ref{prop5}. Since $D$ is an oriented graph and $\lambda(D)\ge2$, $|I_i|,|T_j|\ge3$ for any $i\in[a],j\in[b]$. Thus there are only two non-adjacent strong components in $D'$, say $I_1$ and $I_2$, as $n=8$. Since $\lambda(D)\ge2$, $x_1$ has at least two in-neighbours and one out-neighbour in $D'$, while $x_8$ has at least two out-neighbours and one in-neighbour in $D'$.
	If $|I_1|=3$ and $|I_2|=5$, then $x_1,x_8\in I_2$ and $|A(I_2)|\ge6$. Note that at least one of $x_2$ and $x_7$ is in $I_1$ as $|R|=3$. Then we use Proposition~\ref{prop6} to get a good pair of $D$.
	Now assume $|I_1|=|I_2|=4$. If $x_8\in I_1$ then $x_7\in I_2$ by Claim~\ref{thm3}.2. By Proposition~\ref{prop6}, $D$ has a good pair.
\end{proof}

\section{Good pairs in digraphs of order 9}\label{9}
We have several generalizations of Proposition~\ref{prop2} here, which are easy to check as they satisfy the conditions in Proposition~\ref{prop2}.

\begin{prop}\label{prop10}
	Let $D=(V,A)$ be a digraph and $Q$ be a subdigraph of $D$ with good pair $(O_Q,I_Q)$. Set $X=N_D^-(Q)$ and $Y=N_D^+(Q)$ with $X\cap Y=\emptyset$ and $X\cup Y=V-V(Q)-W$, where $W=\{w_1,w_2\}$. Let $e_1$ be an arc from $w_1$ to $X$ and $e_2$ be an arc from $Y$ to $w_2$. Set $X'=X\cup w_1$, $Y'=Y\cup w_2$ and $D'=(V,A')$ with $A'=A-\{e_1,e_2\}$.
	Let $\mathcal{X}$ be the set of initial strong components in $D'[X']$ and $\mathcal{Y}$ be the set of terminal strong components in $D'[Y']$. Assume that there exists $X_0$ and $Y_0$ in $\mathcal{X}$ and $\mathcal{Y}$ respectively such that $d_Y^-(X_0)=1$ and $d_X^+(Y_0)=1$. Let $e_x$ and $e_y$ be arcs from $Y$ to $X_0$ and from $Y_0$ to $X$ respectively.
	If one of the following holds, then $D$ has a good pair.
	\begin{enumerate}
		\item\label{prop2-1} $e_x\neq e_y$, but at least one of $\mathcal{X}$ or $\mathcal{Y}$ has only one element.
		\item\label{prop2-2} $e_x$ (or $e_y$) is adjacent to some $Y_x$ (or $X_y$) in $\mathcal{Y}$ (or $\mathcal{X}$), such that $d_X^+(Y_x)\ge 3$ (or $d_Y^-(X_y)\ge 3$).
		\item\label{prop2-3} $e_x$ (or $e_y$) is adjacent to $Y'-V(\mathcal{Y})$ (or $X'-V(\mathcal{X})$).
		\item\label{prop2-4} $e_x$ (or $e_y$) is adjacent to some $Y_x\neq Y_0$ (or $X_y\neq X_0$) in $\mathcal{Y}$ (or $\mathcal{X}$), such that there exists an arc from $Y_x$ (or $X_y$) to $X'-V(\mathcal{X})$ (or $Y'-V(\mathcal{Y})$).
	\end{enumerate}
\end{prop}

\begin{lemma}\label{lem11}
	Let $D$ be a 2-arc-strong digraph on $9$ vertices that contains a digon $Q$. Assume that $D$ has no subdigraph with a good pair on 3 or 4 vertices. Set $X=N_D^-(Q)$ and $Y=N_D^+(Q)$ with $X\cap Y=\emptyset$. If $|X|=3$ and $|Y|=2$, then $D$ has a good pair.
\end{lemma}
\begin{proof}
	Let $D=(V,A)$ and $W=V-X-Y-V(Q)=\{w_1,w_2\}$. Set $V(Q)=\{q_1,q_2\}$, $X=\{x_1,x_2,x_3\}$ and $Y=\{y_1,y_2\}$. By contradiction, suppose that $D$ has no good pair. Since $\lambda(D)\ge2$, w.l.o.g., assume $q_1y_1,q_2y_2\in A$. Let $B^+=q_1q_2y_2+q_1y_1$ and $B^-$ be an in-tree rooted at $q_1$ such that $A(B^-)\subseteq \{q_2q_1\}\cup (X,Q)_D$.
	By Proposition~\ref{prop4}, we get the following fact.
	\begin{description}
		\item[Fact~\ref{lem11}.1.] There is no digon in the induced subdigraph of $X$ or $Y$.
	\end{description}
	
	We will now show some notes below.
	\begin{description}
		\item[Note~\ref{lem11}.1.] There is no $C_3$ in the induced subdigraph of $X$.
	\end{description}
	{\it Proof.}
	W.l.o.g., assume that $x_1q_1$, $x_2q_1$, $x_3q_2\in A$ and $D[X]=C_3=x_1x_2x_3x_1$. Then $D[Q\cup X]$ has a good pair as $B_{x_1}^+=x_1x_2x_3q_2q_1$ and $B_{q_2}^-=x_3x_1q_1+x_2q_1q_2$. By Lemma~\ref{lem6}, $D$ has a good pair, a contradiction.
	\hfill $\lozenge$
	
	\begin{description}
		\item[Note~\ref{lem11}.2.] None of the following holds:
		\begin{enumerate}
			\item $|(Y,X_1)_D|=4$ for any $X_1\subset X$ and $|X_1|=2$.
			\item $Y$ is not independent and there exists a vertex in $X$ which is dominated by both vertices in $Y$.
			\item $X$ is not independent, w.l.o.g., set that $X_1$ contains two arbitrarily adjacent vertices in $X$, and there exists a vertex in $Y$ which dominates both vertices in $X_1$.
			\item Both $X$ and $Y$ are not independent, say $x_{i_1}x_{i_2},y_jy_{3-j}\in A$, and $y_jx_{i_1},y_{3-j}x_{i_2}\in A$, where $i_1,i_2\in[3]$ and $j\in[2]$.
			\item Both $X$ and $Y$ are not independent, say $x_{i_1}x_{i_2},y_jy_{3-j}\in A$, and $y_jx_{i_2},y_{3-j}x_{i_1}\in A$, where $i_1,i_2\in[3]$ and $j\in[2]$.
		\end{enumerate}
	\end{description}
	{\it Proof.}
	The proof is similar as that of Claim~\ref{lem8}.1. We give some details in Appendix.
	\hfill $\lozenge$
	
	\begin{description}
		\item[Note~\ref{lem11}.3.] There exists respectively an arc from $Y$ to $W$ and an arc from $W$ to $X$.
	\end{description}
	{\it Proof.}
	Let $X_i$ be the initial strong components in $D[X]$ and $Y_j$ be the terminal strong components in $D[Y]$, $i\in [a]$ and $j\in[b]$. Note that $1\le a\le3$ and $1\le b\le2$.
	
	{\bf Case 1:} $(W,X)_D=\emptyset$.
	
	This implies that there are at least two arcs from $W$ to $Y$ as $\lambda(D)\ge2$. For any $i,j$, $d_{X\cup W}^+(Y_j)\ge 2$ and $d_Y^-(X_i)\ge 2$ since $\lambda(D)\ge2$ and $(W,X)_D=\emptyset$. By Proposition~\ref{prop2}, we get $\mathcal{P}_X$ and $T_X$.
	
	If $d_X^+(Y_j)\ge 2$ for any $j$, then we get $\mathcal{P}_Y$ and $T_Y$ with $\mathcal{P}_X\cap \mathcal{P}_Y=\emptyset$ by Proposition~\ref{prop2}. It follows that $(B^++\mathcal{P}_X+T_X,B^-+\mathcal{P}_Y+T_Y)$ is a good pair of $D-W$, which implies that $D$ has a good pair by Lemma~\ref{lem1}, a contradiction.
	Hence there exists a terminal strong component in $D[Y]$, say $Y_1$, such that $d_X^+(Y_1)\le1$. This implies that $|Y_1|=1$, w.l.o.g., say $Y_1=\{y_1\}$. Note that there exists a dipath $P^1$ from $y_1$ to $y_2$ with $P^1-Y\subseteq W$ and $(y_2,X)_D\neq \emptyset$.
	
	If $y_2$ is not a terminal strong component in $D[Y]$ or $y_2$ is a terminal strong component in $D[Y]$ with $d_X^+(y_2)\ge2$, then we get $\mathcal{P}_Y$ and $T_Y$ with $\mathcal{P}_X\cap \mathcal{P}_Y=\emptyset$ by Proposition~\ref{prop2}. It follows that $D-W$ has a good pair $(B^++\mathcal{P}_X+T_X,B^-+\mathcal{P}_Y+T_Y)$, a contradiction.
	
	Hence $y_2$ is a terminal strong component in $D[Y]$ with $d_X^+(y_2)\le1$. This implies that $d_X^+(y_1)=d_X^+(y_2)=1$. Set $(y_i,X)_D=\{e_i\}$ for any $i\in[2]$. Note that $(Y,X)_D=\{e_1,e_2\}$, then there is only one initial strong component in $D[X]$ as $\lambda(D)\ge2$, say $X_1$. Now $(Y,X_1)_D=\{e_1,e_2\}$. By Proposition~\ref{prop2}, we get $\mathcal{P}_X$ and $T_X$, such that $\mathcal{P}_X=\{e_1\}$. Let $P_+=e_1+T_X$ and $P_-=e_2+P^1$.
	If $P^1=y_1w_iw_{3-i}y_2$, then $(B^++P_++w_i^-w_i+w_{3-i}^-w_{3-i},B^-+P_-)$ is a good pair of $D$, where $w_i^-\neq y_1$ and $w_{3-i}^-\neq w_i$ as $\lambda(D)\ge2$, a contradiction.
	If $P^1=y_1w_iy_2$, then $(B^++P_++w_i^-w_i+w_{3-i}^-w_{3-i},B^-+P_-+w_{3-i}w_{3-i}^+)$ is a good pair of $D$, where $w_i^-\neq y_1$ and $w_{3-i}^-,w_{3-i}^+\neq w_i$ as $\lambda(D)\ge2$, a contradiction.
	
	{\bf Case 2:} $(Y,W)_D=\emptyset$.
	
	Note that $|(Y,X)_D|\ge3$ by Fact~\ref{lem11}.1.
	
	First assume that $Y$ is not an independent set, w.l.o.g., say $y_1y_2\in A$. Now $d_X^+(y_1)\ge1$ and $d_X^+(y_2)\ge2$. By Note~\ref{lem11}.2, $y_1$ and $y_2$ can not dominate the same vertex in $X$, which implies that $d_X^+(y_1)=1$ and $d_X^+(y_2)=2$. W.l.o.g., assume $y_1x_1,y_2x_2,y_2x_3\in A$. Then $X$ is independent by Note~\ref{lem11}.2. Since $(Y,W)_D=\emptyset$ and $\lambda(D)\ge2$, there exists a dipath $P^1$ from some vertex $x$ in $\{x_1,x_2\}$ to $x_3$ with $V(P^1)-\{x,x_3\}\subseteq W$. Let $P_+=y_1x_1+y_2x_2+P^1$ and $P_-=y_1y_2x_3$.
	
	Next assume that $Y$ is an independent set, i.e., $d_X^+(y_i)\ge2$ for any $i\in[2]$. By Note~\ref{lem11}.2, at most one vertex in $X$ can be dominated by both vertices in $Y$, namely $|(Y,X)_D|=4$. W.l.o.g., say $x_1$ is dominated by both vertices in $Y$ and $y_1x_2,y_2x_3\in A$. Now $x_1$ is non-adjacent to $x_2$ or $x_3$ by Note~\ref{lem11}.2 and there exists at most one arc between $x_2$ and $x_3$ by Fact~\ref{lem11}.1, w.l.o.g., say $x_2x_3\notin A$. Likewise, there exists a dipath $P^1$ from some vertex $x$ in $\{x_1,x_2\}$ to $x_3$ with $V(P^1)-\{x,x_3\}\subseteq W$ as $(Y,W)_D=\emptyset$ and $\lambda(D)\ge2$. Let $P_+=y_1x_2+y_2x_1+P^1$ and $P_-=y_1x_1+y_2x_3$.
	
	If $P^1=xw_1w_2x_3$, then $(B^++P_+,B^-+P_-+w_1w_1^++w_2w_2^+)$ is a good pair of $D$, where $w_1^+\neq w_2$ and $w_2^+\neq x_3$, a contradiction.
	If $P^1=xw_1x_3$, then $(B^++P_++w_2^-w_2,B^-+P_-+w_1w_1^++w_2w_2^+)$ is a good pair of $D$, where $w_1^+\neq x_3$ and $w_2^-,w_2^+\neq w_1$, a contradiction.
	\hfill $\lozenge$
	
	\vspace{2mm}
	W.l.o.g., assume that $e_1=w_1x_1$ is an arc from $w_1$ to $X$. First assume that $e_2=y_1w_2$ is an arc from $Y$ to $w_2$.
	Set $D'=(V,A')$ with $A'=A-\{e_1,e_2\}$ and $X'=X\cup \{w_1\}$, $Y'=Y\cup \{w_2\}$. Let $\mathcal{X}$ be the set of initial strong components in $D'[X']$ and $\mathcal{Y}$ be the set of terminal strong components in $D'[Y']$.
	
	By Proposition~\ref{prop2}, if there exists at most one strong component $S$ in $\mathcal{X}\cup \mathcal{Y}$ which satisfies that $S\in \mathcal{X}$ (or $S\in \mathcal{Y}$) with $d_Y^-(S)=1$ (or $d_X^+(S)=1$), then we can find a good pair of $D$, a contradiction. We get the fact below.
	\begin{description}
		\item[Fact~\ref{lem11}.2.] There exist two strong components $X_0=x_1$ and $Y_0=y_1$ respectively in $\mathcal{X}$ and $\mathcal{Y}$ such that $d_Y^-(X_0)=1$ and $d_X^+(Y_0)=1$.
	\end{description}
	That is $y_1y_2,x_2x_1,x_3x_1\notin A$.
	Let $e_x$ and $e_y$ be the arc from $Y$ to $X_0$ and from $Y_0$ to $X$, respectively.
	One can easily check the note below.

	\begin{description}
		\item[Note \ref{lem11}.4] {\it Assume that $e_x=e_y=y_1x_1$ and there are at least two arcs from $w_2$ to $X'$ in $D$. Let $e$ be an arc from $w_2$ to $X$ and $D''=(V,A'')$ with $A''=A-\{e_x,e\}$. Then there exists at least one arc from $Y'$ to each initial strong component of $D''[X']$ respectively.}
	\end{description}
		
	\begin{description}
		\item[Claim \ref{lem11}.1] {\it $D[\{w_2,y_2\}]$ is not a digon.}
	\end{description}
	{\it Proof.}
	Suppose $D[\{w_2,y_2\}]=C_2$. Since $D\nsupseteqq E_3$, $A(D[Y'])=\{y_1w_1,w_2y_2,y_2w_2\}$. Set $D''=(V,A'')$ with $A''=A-\{y_2w_2,w_1x_1\}$. Note that $D''[Y']$ has only one terminal strong component $\{y_2\}$. By Fact~\ref{lem11}.2, $d_X^+(y_2)=1$, say $e_{y_2}=(y_2,X)_D$. Then $e_{y_2}=e_x$ by Proposition~\ref{prop10}.\ref{prop2-1}.
	
	Observe that $I_D=B^-+y_1w_2y_1x_1+w_1w_1^+$ is an in-branching of $D$, where $w_1^+\neq x_1$ as $\lambda(D)\ge2$.
	If $w_1^+\in Y'$, then any initial strong component in $D[X']$ has at least one in-arc from $Y'$ which is different from $y_1x_1$. This implies that $D$ has an out-branching $O_D=B^++y_2w_2+T_{X'}+\mathcal{P}_{X'}$ by Proposition~\ref{prop2}. Thus $(O_D,I_D)$ is a good pair in $D$, a contradiction.
	
	Otherwise $w_1^+\in \{x_2,x_3\}$. Set $D^*=(V,A^*)$ with $A^*=A-\{y_2w_2,w_1w_1^+\}$. By Fact~\ref{lem11}.2, $y_2$ and $w_1^+$ is respectively a terminal and an initial strong component in $D^*[Y']$ and $D^*[X']$ with $d_X^+(y_2)=d_Y^-(w_1^+)=1$. Since $e_{y_2}=y_2x_1\neq y_2w_1^+$, $D$ has a good pair by Proposition~\ref{prop10}.\ref{prop2-1}, a contradiction.
	\hfill $\lozenge$
	
	\begin{description}
		\item[Claim \ref{lem11}.2] $w_2y_2 \notin A.$
	\end{description}
	{\it Proof.}
	Suppose to the contrary that $w_2y_2\in A$.
	
	{\bf Case 1:} $y_1$ is adjacent to $y_2$ in $D$.
	
	That is $y_2y_1\in A$. Now $D'[Y']$ has only one terminal strong component $Y_0$, which implies that $e_y=e_x=y_1x_1$. Then $I_D=B^-+w_2y_2y_1x_1+w_1w_1^+$ is an in-branching of $D$, where $w_1^+\neq x_1$ as $\lambda(D)\ge2$.
	If $w_1^+\in Y'$, then any initial strong component in $D[X']$ has at least one in-arc from $Y'$ which is different from $y_1x_1$. This implies that $D$ has an out-branching $O_D=B^++y_1w_2+T_{X'}+\mathcal{P}_{X'}$ by Proposition~\ref{prop2}. Thus $(O_D,I_D)$ is a good pair in $D$, a contradiction.
	
	Otherwise $w_1^+\in \{x_2,x_3\}$. Set $D^*=(V,A^*)$ with $A^*=A-\{y_1w_2,w_1w_1^+\}$. By Fact~\ref{lem11}.2, $y_1$ and $w_1^+$, respectively, is a terminal and an initial strong component in $D^*[Y']$ and $D^*[X']$ with $d_X^+(y_1)=d_Y^-(w_1^+)=1$. Since $e_y=y_1x_1\neq y_1w_1^+$, $D$ has a good pair by Proposition~\ref{prop10}.\ref{prop2-1}, a contradiction.
	
	{\bf Case 2:} $y_1$ is not adjacent to $y_2$ in $D$.
	
	That is $|(y_2,X')_D|\ge2$.
	
	{\bf Subcase 2.1:} $|\mathcal{X}|=1$.
	
	Namely $\mathcal{X}=\{X_0\}$. By Proposition~\ref{prop10}.\ref{prop2-1}, $e_y=e_x=y_1x_1$. By Proposition~\ref{prop2}, we get $T_{X'}$ of $D'[X']$. Then $(B^++y_1x_1+T_{X'}+w_2^-w_2,B^-+y_1w_2y_2y_2^++w_1x_1)$ is a good pair of $D$, where $w_2^-\neq y_1$ and $y_2^+\neq w_2$, a contradiction.
	
	\vspace{2mm}
	Discussions of subcases when ``$|\mathcal{X}|=2$'', ``$|\mathcal{X}|=3$'' and ``$|\mathcal{X}|=4$'' are given in Appendix.
	\hfill $\lozenge$
	
	\begin{description}
		\item[Claim \ref{lem11}.3] $y_2w_2 \notin A.$
	\end{description}
	{\it Proof.}
	Suppose $y_2w_2\in A$. Set $D''=(V,A'')$ with $A''=A-\{w_1x_1,y_2w_2\}$. Note that $y_2$ is a terminal strong component in $D''[Y']$ by Fact~\ref{lem11}.2. Moreover $y_2y_1\notin A$, i.e., $y_1$ is not adjacent to $y_2$ in $D$. Let $e_{y_2}$ be the arc from $y_2$ to $X'$.
	
%
	
	First assume $e_y=e_x=y_1x_1$. Now $y_1$ is not in some terminal strong component of $D''[Y']$, which implies that $D$ has a good pair by Proposition~\ref{prop10}.\ref{prop2-3}, a contradiction.
	
	Hence $e_y\neq e_x$. Analogously $e_{y_2}\neq e_x$.
	Then $e_y$ (resp. $e_{y_2}$) is adjacent to some initial strong component in $D'[X']$ (resp. $D''[X']$) other than $X_0$ (resp. $x_1$). This implies that $e_x=w_2x_1$ and $|\mathcal{X}|\ge2$.
	If $|\mathcal{X}|\ge3$, then $|(Y',X')_D|\ge5$, i.e., $d_{X'}^+(w_2)\ge3$. By Proposition~\ref{prop10}.\ref{prop2-2}, $D$ has a good pair, a contradiction.
	If $|\mathcal{X}|=2$, then we get $T_{X'}$ of $D'[X']$ by Proposition~\ref{prop2}. It follows that $(B^++y_1w_2x_1+e_{y_2}+T_{X'},B^-+w_1x_1+e_y+y_2w_2w_2^+)$ is a good pair of $D$, where $w_2^+\neq x_1,y_2$ by $\lambda(D)\ge2$ and Claim~\ref{lem11}.2, a contradiction.
	\hfill $\lozenge$
	
	\vspace{2mm}
	Therefore $y_2$ is not adjacent to $w_2$ in $D$.
	
	\begin{description}
		\item[Claim \ref{lem11}.4] $y_1$ is not adjacent to $y_2$ in $D$.
	\end{description}
	{\it Proof.}
	Suppose to the contrary that $y_1$ is adjacent to $y_2$, i.e., $y_2y_1\in A$.
	
	{\bf Case 1:} $w_2y_1\in A$.
	
	Now $D'[Y']$ has only one terminal strong component $Y_0$, that is $e_y=e_x=y_1x_1$ by Proposition~\ref{prop10}.\ref{prop2-1}. Then $I_D=B^-+w_2y_1+y_2y_1x_1+w_1w_1^+$ is an in-branching of $D$, where $w_1^+\neq x_1$ as $\lambda(D)\ge2$.
	
	If $w_1^+\in Y'$, then any initial strong component in $D[X']$ has at least one in-arc from $Y'$ which is different from $y_1x_1$. This implies that $D$ has an out-branching $O_D=B^++y_1w_2+T_{X'}+\mathcal{P}_{X'}$ by Proposition~\ref{prop2}. Thus $(O_D,I_D)$ is a good pair in $D$, a contradiction.
	
	Otherwise $w_1^+\in \{x_2,x_3\}$. Set $D^*=(V,A^*)$ with $A^*=A-\{y_1w_2,w_1w_1^+\}$. By Fact~\ref{lem11}.2, $y_1$ is the only terminal strong component of $D^*[Y']$ and $w_1^+$ is an initial strong component in $D^*[X']$ with $d_X^+(y_2)=d_Y^-(w_1^+)=1$. Since $e_y=y_1x_1\neq y_1w_1^+$, $D$ has a good pair by Proposition~\ref{prop10}.\ref{prop2-1}, a contradiction.
	
	Thus $d_{X'}^+(w_2)\ge2$.
	
	{\bf Case 2:} $|\mathcal{X}|=1$.
	
	By Proposition~\ref{prop10}.\ref{prop2-1}, $e_y=e_x=y_1x_1$. By Proposition~\ref{prop2}, we get $T_{X'}$. It follows that $(B^++y_1x_1+T_{X'}+w_2^-w_2,B^-+y_2y_1w_2w_2^++w_1x_1)$ is a good pair of $D$, where $w_2^-\neq y_1$ and $w_2^+\in X'$ as $\lambda(D)\ge2$ and $d_{X'}^+(w_2)\ge2$, a contradiction.
	
	{\bf Case 3:} $|\mathcal{X}|=2$.
	
	Set $\mathcal{X}=\{X_0,X_1\}$. Recall that $X_0=x_1$. By Proposition~\ref{prop2}, we get $T_{X'}$.
	
	{\bf Subcase 3.1:} $e_x=e_y=y_1x_1$.
	
	If $(y_2,X_1)_D\neq \emptyset$, say $e'\in (y_2,X_1)_D$, then $D$ has a good pair $(B^++y_1x_1+e'+T_{X'}+w_2^-w_2,B^-+y_2y_1w_2w_2^++w_1x_1)$, where $w_2^-\neq y_1$ and $w_2^+\in X'$, a contradiction.
	
	Otherwise $d_{X_1}^+(w_2)\ge2$, namley $|X_1|\ge2$. As $(y_2,X'-V(X_0\cup X_1))_D\neq \emptyset$, $|X_1|=2$. This implies that $D[\{w_2\}\cup X_1]\supseteq E_3$, a contradiction to that $D\nsupseteqq E_3$.
	
	\vspace{2mm}
	Discussions of the subcase when ``$e_y\neq e_x$'' and cases when ``$|\mathcal{X}|=3$'' and ``$|\mathcal{X}|=4$'', see Appendix.
	\hfill $\lozenge$
	
	\vspace{2mm}
	Now $|(y_2,X')_D|\ge2$. Similar to Note~\ref{lem11}.4 we get the note below.
	\begin{description}
		\item[Note \ref{lem11}.5] {\it Assume that $e_x=e_y=y_1x_1$, $(w_1,X)_D=\{w_1x_1\}$ and there are at least two arcs from $y_2$ to $X'$ in $D$. Let $e$ be an arc from $y_2$ to $X$ and $D''=(V,A'')$ with $A''=A-\{e_x,e\}$. Then there exists at least one arc from $Y'$ to each initial strong component of $D''[X']$ respectively.}
	\end{description}
	
	Then we distinguish several cases as follows.
	
	{\bf Case 1} $e_x=e_y$.
	
	That is $e_x=e_y=y_1x_1$.
	
	{\bf Subcase 1.1} $w_1x_2\in A$.
	
	Now change $e_1$ from $w_1x_1$ to $w_1x_2$, then $X_0=x_2$, or $D$ has a good pair. If $x_1$ is not in some component in $\mathcal{X}$, then $D$ has a good pair by Proposition~\ref{prop10}.\ref{prop2-3}, a contradiction. Thus we assume that $x_1\in X_1\in \mathcal{X}$. Note that $w_1\in X_1$ as $w_1x_1\in A$.
	If $X_1=C_3$, then $V(X_1)=\{x_1,x_3,w_1\}$. As $\lambda(D)\ge2$, $|(Y',X_1)_D|\ge3$, which implies that $D$ has a good pair by Proposition~\ref{prop10}.\ref{prop2-2}, a contradiction.
	
	Hence $X_1=C_2=x_1w_1x_1$. Note that $|(Y',X)_D|\ge3$ as $|\mathcal{X}|\ge 2$. Since now $X_0=x_2$ and $y_1x_1\in A$, $(Y',X_1)=\{y_1x_1,w_1^-w_1\}$ and $(Y',X_0)=\{x_2^-x_2\}$, where $w_1^-,x_2^-\in Y'$.
	If $w_2y_1\in A$, then $x_2^-\neq w_2$ by Proposition~\ref{prop10}.\ref{prop2-3}. Since $D$ has no subdigraph with a good pair on 4 vertices, $w_2x_3,y_2w_1,y_2x_2\in A$. Let $P_+=y_1x_1w_1+y_2x_2+w_2^-w_2+x_3^-x_3$ and $P_-=y_1w_2x_3+y_2w_1x_2$, where $w_2^-\neq y_1$ and $x_3^-\neq w_2$ as $D\nsupseteq E_3$. Then $(B^++P_+,B^-+P_-)$ is a good pair of $D$, a contradiction.
	Henceforth $|(w_2,X')_D|\ge2$. It follows that $w_2x_3,y_2x_3\in A$ and $x_3\in X_2\in \mathcal{X}$ by Proposition~\ref{prop10}.\ref{prop2-4}. Let $P_+=y_1x_1w_1x_2+y_2x_3+w_2^-w_2$ and $P_-=y_1w_2x_3+w_1x_1+y_2y_2^+$ where $w_2^-\neq y_1$ and $y_2^+\neq x_3$ as $y_2$ is not adjacent to $w_2$. Then $(B^++P_+,B^-+P_-)$ is a good pair of $D$, a contradiction.
	
	The case of $w_1x_3\in A$ can be proved analogously.
	For the discussion of the subcase when ``$(w_1,X)_D=\{w_1x_1\}$'', see Appendix.
	
	{\bf Case 2} $e_x\neq e_y$.
	
	For the discussion of Case 2, see Appendix.
	
	\vspace{2mm}
	The discussion above implies that $D$ has a good pair when there exists an arc from $Y$ to $w_i$ and an arc from $w_{3-i}$ to $X$, where $i\in [2]$.
	
	Henceforth assume that $e_2=y_1w_1$ and there is no arc from $Y$ to $w_2$ or from $w_2$ to $X$. This implies that there exists an in-arc $e^-$ of $w_2$ from $X$ and an out-arc $e^+$ of $w_2$ to $Y$. Set $D'=(V',A')$ with $V'=V-w_2, A'=A-e_1$ and $X'=X\cup \{w_1\}$. Let $\mathcal{X}$ be the set of initial strong components in $D'[X']$ and $\mathcal{Y}$ be the set of terminal strong components in $D'[Y]$. Observe that there is at most one strong component $S_x$ in $\mathcal{X}$ with $d_Y^-(S_x)=1$, and meanwhile for an arbitrary strong component $S$ in $\mathcal{X}\cup \mathcal{Y}-S_x$, we have $d_Y^-(S)\ge 2$ when $S\in \mathcal{X}$ and $d_{X'}^+(S)\ge 2$ when $S\in \mathcal{Y}$. Now by Proposition~\ref{prop2}, we get two arc sets $\mathcal{P}_{X'}$ and $\mathcal{P}_Y$ with $\mathcal{P}_{X'}\cap \mathcal{P}_Y=\emptyset$, and $T_{X'},T_Y$. It follows that $D$ has a good pair $(B^++\mathcal{P}_{X'}+T_{X'}+e^-,B^-+\mathcal{P}_Y+T_Y+e_1+e^+)$, a contradiction.
	This completes the proof.
\end{proof}

\begin{lemma}\label{lem12}
	Let $D=(V,A)$ be a 2-arc-strong digraph on $9$ vertices that contains a digon $Q$. Assume that $D$ has no subdigraph with a good pair on at least 3 vertices. Set $X=N_D^-(Q)$ and $Y=N_D^+(Q)$ with $X\cap Y=\emptyset$ and $W=V-V(Q)-X-Y$. Assume that $|X|=|Y|=2$ and there is an arc $e=st\in A$ such that $s\in Y$ and $t\in W$ (resp. $s\in W$ and $t\in X$). If there are at least three arcs in $D[Y\cup\{t\}]$ (resp. $D[X\cup\{s\}]$), then $D$ has a good pair.
\end{lemma}
\begin{proof} Suppose $D$ has no good pair.
	Set $V(Q)=\{q_1,q_2\}$, $X=\{x_1,x_2\}$, $Y=\{y_1,y_2\}$ and $W=\{w_1,w_2,w_3\}$.
	By the digraph duality, it suffices to prove the case when $s\in Y$ and $t\in W$. W.l.o.g., let $s=y_1$ and $t=w_1$.
	
	Set $D[V(Q)\cup Y\cup\{w_1\}]=H$. Since both $D[\{q_1,q_2,y_1\}]$ and $D[\{q_1,q_2,y_2\}]$ has at most three arcs by Proposition~\ref{prop3}, there are six possible cases of $H$ by symmetry, which are depicted in Figure \ref{fig2}.
	We partition each $H_i$ into two parts as follows:
	
	~\\
	\begin{tabular}{l|c|c|c}\hline
		$H_i$ & $B^+(H_i)$ & $F^-(H_i)$ & $a_i$\\\hline
		$H_1$ & $q_2q_1y_1w_1+q_2y_2$ & $q_1q_2+y_1y_2+w_1y_2$ & an out-arc of $y_2$\\\hline
		$H_2$ & $B^+(H_1)$ & $q_1q_2+y_1y_2w_1$ & an out-arc of $w_1$\\\hline
		$H_3$ & $B^+(H_1)$ & $q_1q_2+w_1y_1y_2$ & an out-arc of $y_2$\\\hline
		$H_4$ & $B^+_1(H_4)=B^+(H_1)$ & $F^-_1(H_4)=q_1q_2+w_1y_2y_1$ & $a^1_4$: an out-arc of $y_1$\\\hline
		& $B^+_2(H_4)=q_2q_1y_1w_1y_2$ & $F^-_2(H_4)=q_1q_2y_2y_1+w_1$ & $a^2_4$: an out-arc of $w_1$\\\hline
	\end{tabular}\\
	\vspace{2mm}
	\begin{tabular}{l|c|c|c}\hline
		$H_5$ & $B^+_2(H_4)$ & $q_1q_2y_2+w_1y_1$ & an out-arc of $y_1$\\\hline
		$H_6$ & $B^+_1(H_6)=B^+(H_1)$ & $F^-_1(H_6)=q_1q_2+y_2w_1y_1$ & $a^1_6$: an out-arc of $y_1$\\\hline
		& $B^+_2(H_6)=q_1q_2y_2w_1y_1$ & $F^-_2(H_6)=q_2q_1y_1w_1+y_2$ & $a^2_6$: an out-arcs of $y_2$\\\hline
	\end{tabular}

	Notice that, for each case, $B^+(H_i)$ is always an out-branching of $H_i$, whose arcs are in blue, $F^-(H_i)$ is always an in-forest of $H_i$, whose arcs are in red (see Figure \ref{fig2}), and $a_i$ is some arc from $\{y_1,y_2,w_1\}$ to $\{x_1,x_2,w_1,w_2\}$, where $1\leq i\leq 6$. And for $i=4$ and $6$,
	$B^+(H_i)$ (resp. $F^-(H_i)$, $a_i$) denotes $B^+_1(H_i)$ (resp. $F^-_1(H_i)$, $a^1_i$) or $B^+_2(H_i)$ (resp. $F^-_2(H_i)$, $a^2_i$).
	
	\begin{figure}[!htpb]
		\centering\includegraphics[scale=0.8]{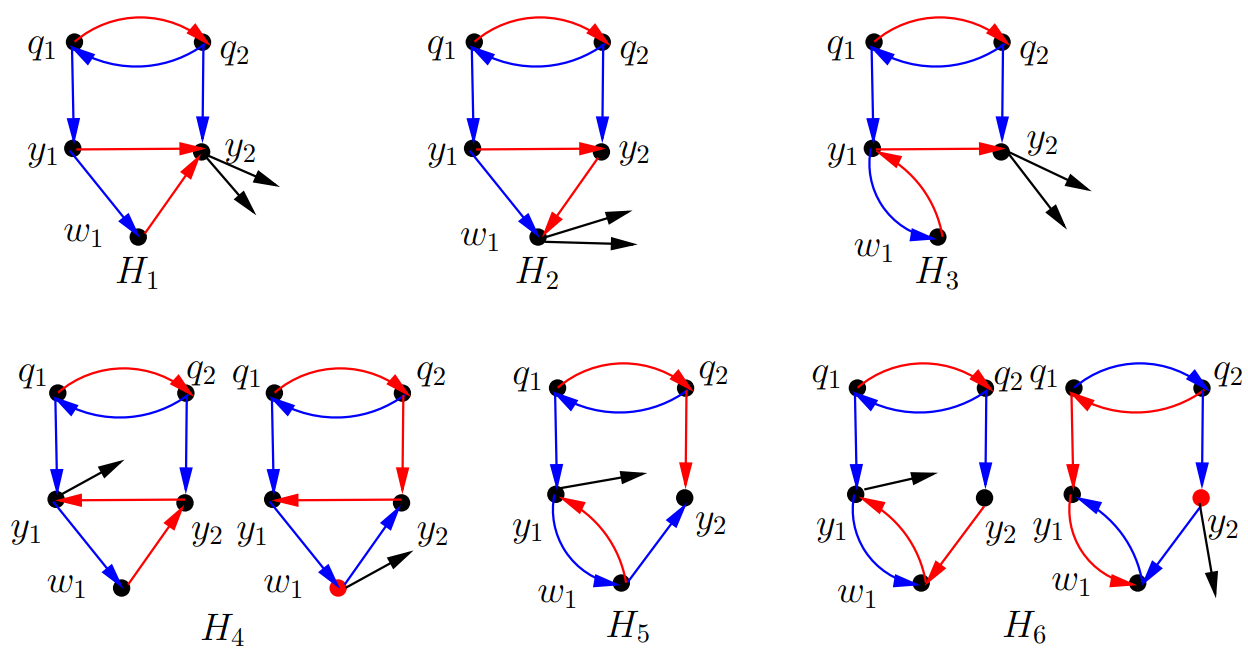}
		\caption{The six possible cases of $H$. The arcs of $B^+(H_i)$ are in blue and the arcs of $F^-(H_i)$ are in red for $1\leq i\leq 6$.}
		\label{fig2}
	\end{figure}
	
	The digraph $H_2-y_1y_2+y_2y_1$ is isomorphic to $H_2$. And, if $H=H_4-w_1y_2+w_1y_1$, then for the digon $D[\{y_1,w_1\}]=\hat{Q}$, $N^-(\hat{Q})$ contains the vertices $q_1,y_2$ and an in-neighbour of $w_1$ from $\{x_1,x_2,w_2\}$. Thus, $|N^-(\hat{Q})|\geq 3$. Since $|N^+(\hat{Q})|\geq 2$, $D$ has a good pair by Corollary \ref{cor1}, and Lemmas \ref{lem2} and \ref{lem11}.
	
	Let $e_{x_j}$ be an arc from $x_j$ to $Q$, for any $j\in[2]$.
	
	\begin{description}
		\item[Note~\ref{lem12}.1] If there exists an arc $a_i$ whose head is $x_1$ or $x_2$, then $F^-(H_i)+e_{x_1}+e_{x_2}+a_i$ is an in-branching of $D\setminus \{w_2,w_3\}$, for each $1\leq i\leq 6$.
	\end{description}
	
	Set $X'=X\cup\{w_2,w_3\}$ and $Y'=Y\cup w_1$.
	
	\begin{description}
		\item[Note~\ref{lem12}.2] Set $D'=(V,A')$ with $A'=A-\{a_i,e_{w_2},e_{w_3}\}$ and each initial strong component in $D'[X']$ has an in-arc from $Y'$. If $F^-(H_i)+e_{x_1}+e_{x_2}+a_i+e_{w_2}+e_{w_3}$ is an in-branching of $D$, then $D$ has a good pair.
	\end{description}
	{\it Proof.}
	Obviously $B^+(H_i)$ can be extended to an out-branching of $D$.
	\hfill $\lozenge$
	
	\vspace{2mm}
	Next, we distinguish three cases:
	
	{\bf Case 1:} $d^+_{Y'}(w_j)\geq 1$ for both $j=2$ and $3$. 
	
	Now there exists an out-arc of $w_j$ with head in $Y'$.	
	Since $\lambda(D)\geq 2$ and $N_D^+(Q)=Y$, each initial strong component in $D[X']$ has at least two in-arcs from $Y'$. Moreover, since $w_j$ (for both $j=2$ and $3$) has an out-arc with head in $Y'$, denoted by $e_{w_j}$, for $D'=(V,A')$ with $A'=A-\{a_i,e_{w_2},e_{w_3}\}$, each initial strong component in $D'[X']$ has at least one in-arc from $Y'$. Thus, if there exists an arc $a_i$ whose head is $x_1$ or $x_2$, then $D$ has a good pair $(O_D,I_D)$ by Notes~\ref{lem12}.1 and \ref{lem12}.2, where $I_D=F^-(H_i)+e_{x_1}+e_{x_2}+a_i+e_{w_2}+e_{w_3}$ and $O_D$ is obtained by extending $B^+(H_i)$, for $1\leq i\leq 6$. Henceforth, we may assume that the head of $a_i$ can only be $w_2$ or $w_3$.
	
	\begin{description}
		\item[Claim \ref{lem12}.1] {\it If there exist an arc $a_i$ with head $w_j$ and an out-arc of $w_j$ with head $x_k$, where $1\leq i\leq 6$, $j\in\{2,3\}$ and $k\in[2]$, then $D$ has a good pair.}
	\end{description}
	{\it Proof.}
	W.l.o.g., suppose that there exists an arc $a_i$ with head $w_2$ and $w_2x_1\in A$. Let $e_{w_2}=w_2x_1$ and $e_{w_3}$ be an out-arc of $w_3$ with head in $Y'$. For $D'=(V,A')$ with $A'=A-e_{w_2}$, if there exists an initial strong component $X_0$ in $D'[X']$ such that $d^-_{Y'}(X_0)=1$, then $x_1\in X_0$ and $w_2\notin X_0$. Hence, for $D''=(V,A'')$ with $A''=A'-\{a_i,e_{w_3}\}$, each initial strong component in $D''[X']$ has at least one in-arc from $Y'$, since the head of $a_i$ is not in $X_0$ and the head of $e_{w_3}$ is not in $X'$. Then $D$ has a good pair $(O_D,I_D)$ by Note~\ref{lem12}.2, where $I_D=F^-(H_i)+e_{x_1}+e_{x_2}+a_i+w_2x_1+e_{w_3}$ and $O_D$ is obtained by extending $B^+(H_i)$, for $1\leq i\leq 6$. The proof is complete.
	\hfill $\lozenge$
	
	\vspace{2mm}
	For $1\leq i\leq 3$, $a_i$ has two choices $e_1$ and $e_2$, both of which are from $y_2$ (for $i=1$ and $3$) or $w_1$ (for $i=2$). Since the head of $a_i$ can only be $w_2$ or $w_3$ and there is no multiple arc in $D$, w.l.o.g., let $e_1=y_2w_2$ (resp. $w_1w_2$) and $e_2=y_2w_3$ (resp. $w_1w_3$) for $i=1$ and $3$ (resp. $i=2$).
	
	We first consider $H_1$ and $H_2$.	
	Now, $D[Y']$ is a tournament of order $3$. If $d^+_{Y'}(w_j)\geq 2$ ($j=2$ or $3$), then $D[Y'\cup\{w_j\}]$ is a tournament of order $4$ or contains a subdigraph on $3$ vertices with $4$ arcs, a contradiction. Hence, $d^+_{Y'}(w_j)=1$ for both $j=2$ and $3$. Moreover, if $D[\{w_2,w_3\}]$ is a digon, then $D[\{y_2,w_2,w_3\}]$ or $D[\{w_1,w_2,w_3\}]$ is a subdigraph on $3$ vertices with $4$ arcs, a contradiction. Thus, at least one of the vertices $w_2$ and $w_3$, say $w_2$, has an out-neighbour in $X$. Let $a_i=e_1$. Now, by Claim~\ref{lem12}.1, $D$ has a good pair.
	
	Discussions of $H_i$, where $i\in\{3,4,5,6\}$, see Appendix.
	
	{\bf Case 2:} $d^+_{Y'}(w_j)=0$ for both $j=2$ and $3$.
	
	Now each out-arc of $w_j$ has head in $X'$. 	
	We will distinguish three subcases in Appendix.
	
	{\bf Case 3:} $d^+_{Y'}(w_j)=0$ and $d^+_{Y'}(w_{5-j})\geq 1$ for $j=2$ or $3$.
	
	W.l.o.g., assume that $d^+_{Y'}(w_2)=0$ (i.e., $d^+_{X'}(w_2)\geq 2$) and $d^+_{Y'}(w_3)\geq 1$. Let $e_{w_3}$ be an out-arc of $w_3$ with head in $Y'$. There are two situations: $(a)$ $w_2x_1,w_2x_2\in A$ and $(b)$ w.l.o.g., $w_2x_1,w_2w_3\in A$ and $w_2x_2\notin A$.
	
	\begin{description}
		\item[Claim \ref{lem12}.2] {\it There exists $a_i$ with head in $\{x_1,x_2,w_2\}$, for each $1\leq i\leq 6$.}
	\end{description}
	{\it Proof.}
	The proof is given in Appendix.
	\hfill $\lozenge$
	
	\vspace{2mm}
	Next, by Claim~\ref{lem12}.2, we distinguish three subcases. Note that, we always let $D'=(V,A')$ with $A'=A-\{e_{w_2},e_{w_3}\}$ and $D''=(V,A'')$ with $A''=A'-\{a_i\}$.
	
	{\bf Subcase 3.1:} $a_i$ with head $x_1$.
	
	For $w_2$ and Situation $(a)$, let $e_{w_2}=w_2x_2$, and for Situation $(b)$, let $e_{w_2}=w_2w_3$. Clearly, for any $1\leq i\leq 6$, $F^-(H_i)+e_{x_1}+e_{x_2}+a_i+e_{w_2}+e_{w_3}+$ is always an in-branching of $D$. If there exists an initial strong component $X_0$ in $D'[X']$ such that $d^-_{Y'}(X_0)=1$, then for Situation $(a)$ (resp. Situation $(b)$), $x_2\in X_0$ (resp. $w_3\in X_0$) and $w_2\notin X_0$. Moreover, $x_1\notin X_0$, since $w_2x_1\in A(D')$ for both situations. Since the head of $a_i$ is $x_1$, in $D''$, each initial strong component of $D''[X']$ has always at least one in-arc from $Y'$. By Note~\ref{lem12}.2, $D$ has a good pair.
	
	For discussions of subcases when ``$a_i$ with head $x_2$'' and ``$a_i$ with head $w_2$'', see Appendix.
	
	The proof is complete.	
\end{proof}

\begin{lemma}\label{lem13}
	Let $D$ be a 2-arc-strong digraph on $9$ vertices that contains a digon $Q$. Assume that $D$ has no subdigraph with a good pair on 3 or 4 vertices. Set $X=N_D^-(Q)$ and $Y=N_D^+(Q)$ with $X\cap Y=\emptyset$. If $|X|=2$ and $|Y|=2$, then $D$ has a good pair.
\end{lemma}
\begin{proof}
	Let $D=(V,A)$, $Q=q_1q_2q_1$ and $W=V-X-Y-V(Q)=\{w_1,w_2,w_3\}$. Set $X=\{x_1,x_2\}$ and $Y=\{y_1,y_2\}$. By contradiction, suppose that $D$ has no good pair when $|X|=|Y|=2$. Let $B^+$ be an out-tree rooted at $q_1$ such that $A(B^+)\subseteq \{q_1q_2\}\cup (Q,Y)_D$ and $B^-$ be an in-tree rooted at $q_1$ such that $A(B^-)\subseteq \{q_2q_1\}\cup (X,Q)_D$.
	
	By Propositions~\ref{prop3} and \ref{prop4}, and Lemma~\ref{lem11}, we get the following facts.
	\begin{description}
		\item[Fact~\ref{lem13}.1.] There is no digon in the induced subdigraph of $X$ or $Y$.
		\item[Fact~\ref{lem13}.2.] There are exactly two vertices respectively in the disjoint in- and out-neibourhood of any $C_2$ in $D$.
		\item[Fact~\ref{lem13}.3.] $|E(W)|\le3$.
		\item[Fact~\ref{lem13}.4.] $|(W,D-W)_D|\ge 3$ and $|(D-W,W)_D|\ge 3$. That is at least 2 vertices in $W$ have out-arcs to $X\cup Y$, analogously at least 2 vertices in $W$ have in-arcs from $X\cup Y$.
	\end{description}
	Let $v$ (resp. $u$) be the vertex in $W$ which does not have any out-neighbour (resp. in-neighbour) in $X\cup Y$ if it exists, and arbitrarily otherwise.
	From Claim~\ref{lem8}.1 and Note~\ref{lem11}.2, we have the fact below by Lemmas~\ref{lem5} and~\ref{lem6}.
	\begin{description}
		\item[Fact~\ref{lem13}.5] None of the following holds:
		\begin{enumerate}
			\item $|(Y,X)_D|=4$;
			\item $X$ is not independent and there exists a vertex in $Y$ which dominates each vertex in $X$;
			\item Analogously $Y$ is not independent and there exists a vertex in $X$ which is dominated by each vertex in $Y$;
			\item Both $X$ and $Y$ are not independent, say $x_ix_{3-i},y_jy_{3-j}\in A$, and $y_jx_i,y_{3-j}x_{3-i}\in A$, where $i,j\in[2].$
			\item Both $X$ and $Y$ are not independent, say $x_ix_{3-i},y_jy_{3-j}\in A$, and $y_jx_{3-i},y_{3-j}x_i\in A$, where $i,j\in[2]$.
		\end{enumerate}
	\end{description}
	
	Now assume that there is no good pair in $D-W$, i.e., $|(Y,X)_D|\le 3$.
	We will now show the note below.
	\begin{description}
		\item[Note~\ref{lem13}.1.] There exists an arc from $Y$ to $W$ and an arc from $W$ to $X$.
	\end{description}
	{\it Proof.}
	Assume that $(Y,W)_D=\emptyset$, i.e., $|(Y,X)_D|\ge3$ by Fact~\ref{lem13}.1. By Fact~\ref{lem13}.5, $|(Y,X)_D|\neq 4$, then $|(Y,X)_D|=3$, which implies that $Y$ is not an independent set. W.l.o.g., assume $y_1y_2\in A$. Now $d_X^+(y_1)\ge1$ and $d_X^+(y_2)\ge2$, namely $y_2$ dominates both vertices in $X$ and $y_1$ dominates at least one vertex in $X$. This is impossible by Fact~\ref{lem13}.5, a contradiction. Thus $(Y,W)_D\neq \emptyset$. Likewise $(W,X)_D\neq \emptyset$.
	\hfill $\lozenge$
	
	\begin{description}
		\item[Claim \ref{lem13}.1] {\it If $|(Y,X)_D|=3$ or $|(Y,X)_D|=2$ with a $P_3$ in $D[X\cup Y]$, then $D$ has a good pair.}
	\end{description}
	{\it Proof.}
	Suppose that $D$ has no good pair. 	
	If $|(Y,X)_D|=3$, then w.l.o.g., assume $(Y,X)_D=\{y_1x_1,y_1x_2,y_2x_2\}$. By Fact~\ref{lem13}.5, both $X$ and $Y$ are independent sets. Let $P^1=y_1x_1+y_2x_2$ and $P^2=y_1x_2$. Note that $V(B^++P^1)=X\cup Y$ and $V(B^-+P^2)=X\cup Y-y_2$.
	
	We do the similar process to the case when $|(Y,X)_D|=2$ with a $P_3$ in $D[X\cup Y]$. Since there exists a $P_3$ in $D[X\cup Y]$ when $|(Y,X)_D|=2$, exactly one of $X$ and $Y$ is not an independent set by Fact~\ref{lem13}.5. W.l.o.g., assume $x_2x_1\in A$. If $(Y,X)_D=\{y_1x_1,y_2x_2\}$, then let $P^1=y_2x_2x_1$ and $P^2=y_1x_1$. If $(Y,X)_D=\{y_1x_2,y_2x_2\}$, then let $P^1=y_2x_2x_1$ and $P^2=y_1x_2$. Since $\lambda(D)\ge2$, $y_2$ has an out-arc to $W$, say $y_2w_1$. We want to add $y_2$ to $B^-$ by $y_2w_1$.
	
	If there exists an arc from $w_1$ to $X\cup Y-y_2$, say $e_1$, then let $P_-=y_2w_1+e_1$. If there exists an arc from $X\cup Y-y_2$ to $w_1$, say $e_2$, then let $P_+=e_2$. Note that $(B^++P^1+P_+,B^-+P^2+P_-)$ is a good pair of $D\setminus \{w_2,w_3\}$, which implies that $D$ has a good pair by Lemma~\ref{lem1} as $\lambda(D)\ge2$, a contradiction.
	
	Hence assume that at least one of $e_1$ and $e_2$ does not exist. We distinguish several cases below.
	
	{\bf Case 1:} $N^+(w_1)\subseteq W$.
	
	That is $w_1w_2,w_1w_3\in A$. Since $|E(W)|\le3$, there is at least one more arc in $E(W)$ other than $w_1$ and $w_2$. This implies that at least one of $w_2$ and $w_3$ has both in- and out-neighbours in $X\cup Y$, say $w_2$. Let $e_1$ be an arc from $X\cup Y$ to $w_2$.
	
	First assume that $(w_2,X\cup Y-y_2)_D\neq \emptyset$ and $(D-w_3,w_1)_D\neq \emptyset$, say $e_2\in(w_2,X\cup Y-y_2)_D$ and $e_3\in(D-w_3,w_1)_D$. Let $P_+=e_1+e_3$ and $P_-=y_2w_1w_2+e_2$. Note that $(B^++P^1+P_+,B^-+P^2+P_-)$ is a good pair of $D-w_3$, which implies that $D$ has a good pair as $\lambda(D)\ge2$, a contradiction.
	
	Next assume that $w_2y_2\in A$ and $N_{D-y_2}^+(w_2)\subseteq W$.
	If $w_2w_1\in A$, then $w_3$ has at least one in-neighbour and two out-neighbours in $X\cup Y$ as $|E(W)|\le3$, which implies that $D$ has a good pair by the discussion above, a contradiction.
	That is $w_2w_3\in A$. Now $w_3$ has at least one out-neighbour $w_3^+\neq y_2$ in $X\cup Y$ and $w_1$ has at least one in-neighbour $w_1^-$ in $X\cup Y$. Let $P_+=w_1^-w_1w_3+e_1$ and $P_-=y_2w_1w_2w_3$. It follows that $(B^++P^1+P_+,B^-+P^2+P_-)$ is a good pair of $D$, a contradiction.
	
	Henceforth, assume $w_3w_1\in A$. Since $|E(W)|\le3$, $w_2$ has at least one out-neighbour $w_2^+\neq y_2$ in $X\cup Y$ and $w_3$ respectively has an in-neighbour $w_3^-$ and an out-neighbour $w_3^+$ in $X\cup Y$. Let $P_+=e_1+w_3^-w_3w_1$ and $P_-=y_2w_1w_2w_2^++w_3w_3^+$. It follows that $(B^++P^1+P_+,B^-+P^2+P_-)$ is a good pair of $D$, a contradiction.
	
	\vspace{2mm}
	Discussions of cases when ``$w_1y_2\in A$ and $N_{D-y_2}^+(w_1)\subseteq W$'' and ``$e_1$ exists but $N_{D-y_2}^-(w_1)\subset W$'', which are analogous to Case 1, are given in Appendix.
	\hfill $\lozenge$
	
	\begin{description}
		\item[Claim \ref{lem13}.2] {\it At least one of $X$ and $Y$ is an independent set.}
	\end{description}
	{\it Proof.}
	W.l.o.g., assume $q_1y_1,q_2y_2\in A$.
	Suppose that both $X$ and $Y$ are not independent sets, i.e., $|E(X)|=|E(Y)|=1$ by Fact~\ref{lem13}.1. W.l.o.g., assume $y_1y_2,x_1x_2\in A$. By Lemma~\ref{lem12}, $|[Y,w_i]_D|\le1$ and $|[X,w_i]_D|\le1$, for arbitrary $w_i\in W$.
	
	{\bf Case 1:} $|(Y,X)_D|=0$.
	
	Since $\lambda(D)\ge2$, $|(Y,W)_D|=3$ with $|N_W^+(y_1)|=1$ and $|N_W^+(y_2)|=2$ by the fact that $|[Y,w_i]_D|\le1$. W.l.o.g., assume $y_1w_1,y_2w_2,y_2w_3\in A$. Note that $(W,Y)_D=\emptyset$. Likewise, $|(W,X)_D|=3$ and $(X,W)_D=\emptyset$ with $|N_W^-(x_1)|=2$ and $|N_W^-(x_2)|=1$. This implies that $D[W]=C_3$ by Fact~\ref{lem13}.3. Set $C_3=w_1w_{j+1}w_{4-j}w_1$, where $j\in\{1,2\}$.
	
	If $w_1x_1\notin A$, then $w_2x_1,w_3x_1,w_1x_2\in A$. It follows that $(B^++y_1w_1w_{j+1}+y_2w_{4-j}x_1x_2,B^-+y_1y_2w_{j+1}w_{4-j}w_1x_2)$ is a good pair of $D$, a contradiction.
	If $w_1x_1\in A$, then set $w_ix_1\in A$, where $i\in \{2,3\}$. It follows that $(B^++y_1w_1w_{j+1}+y_2w_{4-j}+w_ix_1x_2,B^-+y_1y_2w_{j+1}w_{4-j}w_1x_1)$ is a good pair of $D$, a contradiction.
	
	{\bf Case 2:} $|(Y,X)_D|=1$.
	
	Discussion of Case 2 is given in Appendix.
	\hfill $\lozenge$
	
	\begin{description}
		\item[Claim \ref{lem13}.3] {\it If $|E(X)|+|E(Y)|=1$, then $|(Y,X)_D|\le 1$.}
	\end{description}
	{\it Proof.}
	Suppose to the contrary that $|(Y,X)_D|=2$ by Claim~\ref{lem13}.1. W.l.o.g., assume $E(X)=\{x_2x_1\}$. By Claim~\ref{lem13}.1, assume $(Y,X)_D=\{y_1x_1,y_2x_1\}$. Let $P^1=x_2x_1$ and $P^2=y_1x_1+y_2x_1$. Then $B^-+P^2$ is an in-tree containing $Y$ but $B^++P^1$ is not an out-tree.
	Since $|[W,X\cup Y]_D|\ge4$ and $|W|=3$, at least one vertex in $W$ is adjacent to two vertices in $X\cup Y$, say $w_1$.
	
	{\bf Case 1:} $y_1w_1,y_2w_1\in A$.
	
	{\bf Subcase 1.1:} $w_1x_2\in A$.
	
	If $w_1$ has an out-neighbour $w_1^+$ in $X\cup Y$ such that $w_1^+\neq x_2$, then $(B^++P^1+y_2w_1x_2,B^-+P^2+w_1w_1^+)$ is a good pair of $D-\{w_2,w_3\}$. This implies that $D$ has a good pair by Lemma~\ref{lem1}, a contradiction.
	Otherwise, $N_{D-x_2}^+(w_1)\subset W$. W.l.o.g., assume $w_1w_2\in A$.
	If $(w_2,X\cup Y)_D\neq \emptyset$ and $(X\cup Y,w_2)_D\neq \emptyset$, then set $e_1\in (w_2,X\cup Y)_D$ and $e_2\in (X\cup Y,w_2)_D$. Let $P_+=y_2w_1x_2+e_2$ and $P_-=w_1w_2+e_1$. Then $(B^++P^1+P_+,B^-+P^2+P_-)$ is a good pair of $D-w_3$, a contradiction. Now we discuss situations when $e_1$ or $e_2$ does not exist.
	
	First assume that $e_1$ does not exist, namely $N^+(w_2) \subset W$. Then $E(W)=\{w_1w_2,w_2w_3,w_2w_1\}$. Let $P_+=y_2w_1x_2+w_2^-w_2+w_3^-w_3$ and $P_-=w_1w_2w_3w_3^+$, where $w_2^-,w_3^-,w_3^+\in X\cup Y$ by Fact~\ref{lem13}.3. It follows that $(B^++P^1+P_+,B^-+P^2+P_-)$ is a good pair of $D$, a contradiction.
	
	Next assume that $e_1$ exists, but $e_2$ does not, that is $w_3w_2\in A$. Let $P_+=y_2w_1x_2+w_3^-w_3w_2$ and $P_-=w_1w_2+e_1+w_3w_3^+$, where $w_3^-,w_3^+\neq w_2$ as $\lambda(D)\ge2$. It follows that $(B^++P^1+P_+,B^-+P^2+P_-)$ is a good pair of $D$, a contradiction.
	
	\vspace{2mm}
	Discussions of the subcase when ``$w_2x_2,w_3x_2\in A$'' and the case when ``$y_2w_1,w_1x_2\in A$'' are given in Appendix.
	\hfill $\lozenge$
	
	\begin{description}
		\item[Claim \ref{lem13}.4] {\it Both $X$ and $Y$ are independent sets.}
	\end{description}
	{\it Proof.}
	W.l.o.g., assume $q_1y_1,q_2y_2\in A$.
	Suppose that one of $X$ and $Y$ is not an independent set by Claim~\ref{lem13}.2. W.l.o.g., assume that $|E(Y)|=1$ by Fact~\ref{lem13}.1, say $y_1y_2\in A$. Now $|(Y,X)_D|\le 1$ by Claim~\ref{lem13}.3.
	
	{\bf Case 1:} $|(Y,X)_D|=1$.
	
	{\bf Subcase 1.1:} $(Y,X)_D=\{y_1x_i\},~i\in \{1,2\}$.
	
	W.l.o.g., assume $i=2$. Then $N^-(x_1)\subseteq W$, $d_W^-(x_2)\ge1$ and $d_W^+(y_2)\ge2$. Assume $w_1x_1,w_2x_1\in A$, then $|(y_2,\{w_1,w_2\})_D|\ge1$, w.l.o.g., say $y_2w_1\in A$.
	
	First assume $w_1x_2\in A$. Let $B_{q_2}^+=q_2q_1y_1y_2w_1x_2$ and $B_{y_2}^-$ be an in-tree rooted at $y_2$ such that $A(B_{y_2}^-)\subseteq \{q_1q_2y_2,y_1x_2,w_1x_1\}\cup (X,Q)_D$.
	If $N_{D-\{x_1,w_2\}}^-(w_3)\neq \emptyset$, say $w_3^-\in N_{D-\{x_1,w_2\}}^-(w_3)$, then $D$ has a good pair $(B_{q_2}^++w_3^-w_3+w_2^-w_2x_1,B_{y_2}^-+w_3w_3^++w_2w_2^+)$, where $w_2^-\neq w_3$, $w_3^+\neq w_2$ and $w_2^+\neq x_1$ as $\lambda(D)\ge2$, a contradiction.
	Hence $x_1w_3,w_2w_3\in A$. It follows that $(B_{q_2}^++w_2^-w_2x_1w_3,B_{y_2}^-+w_2w_3w_3^+)$ is a good pair of $D$, where $w_2^-\neq x_1,w_3$ and $w_3^+\neq x_1,w_2$ as $D\nsupseteqq E_3$, a contradiction.
	
	Next assume that $w_1x_2\notin A$ but $w_2x_2\in A$. Now there exists at least one arc from $y_2$ to $\{w_2,w_3\}$. By the discussion above, $y_2w_3\in A$. Let $B_{q_2}^+=q_2q_1y_1y_2w_1x_1+y_2w_3$ and $B_{y_2}^-$ be an in-tree rooted at $y_2$ such that $A(B_{y_2}^-)\subseteq \{q_1q_2y_2,y_1x_2,w_2x_1\}\cup (X,Q)_D$. It follows that $(B_{q_2}^++w_2^-w_2x_2,B_{y_2}^-+w_1w_1^++w_3w_3^+)$ is a good pair of $D$, where $w_2^-\neq w_1,x_2$, $w_1^+\neq x_1$ and $w_3^+\neq w_2$ by $\lambda(D)\ge2$ and Lemma~\ref{lem12}, a contradiction.
	
	Henceforth assume that $w_1x_2,w_2x_2\notin A$ but $w_3x_2\in A$.
	
	{\bf A.} $N^-(w_3)\cap \{y_1,y_2,x_2\}\neq \emptyset$.
	
	Set $w_3^-\in\{y_1,y_2,x_2\}$.
	If $w_2w_1\in A$, then $D$ has a good pair $(B^++y_1x_2+w_3^-w_3+w_2^-w_2w_1x_1,B^-+w_3x_2+w_2x_1+y_1y_2w_1w_1^+)$, where $w_2^-\neq w_1,x_1$ and $w_1^+\notin Y$ as $D\nsupseteqq E_3$ and Lemma~\ref{lem12}, a contradiction.
	Thus $w_2w_1\notin A$. Let $w_1^-$ and $w_2^-$ respectively be an in-neighbour of $w_1$ and an in-neighbour of $w_2$ such that $w_1^-\neq y_2$ and $w_2^-\neq x_1$ as $\lambda(D)\ge2$. Since $D\nsupseteqq E_3$, at least one of $w_1^-$ and $w_2^-$ is not in $\{x_1,w_1,w_2\}$. It follows that $(B^++y_1x_2+w_3^-w_3+w_1^-w_1+w_2^-w_2x_1,B^-+w_3x_2+y_1y_2w_1x_1+w_2w_2^+)$ is a good pair of $D$, where $w_2^+\neq x_1$, a contradiction.
	
	{\bf B.} $y_2w_2\in A$ and $d_{\{x_1,w_1,w_2\}}^-(w_3)\ge2$.
	
	Since $d_{\{x_1,w_1,w_2\}}^-(w_3)\ge2$, assume that $w_iw_3\in A$, where $i\in[2]$. We first show that $w_3w_i\notin A$. If $x_1w_3\in A$, then $w_3w_i\notin A$ as $D\nsupseteqq E_3$. If $w_1w_3,w_2w_3\in A$, then $w_3w_i\notin A$ by Fact~\ref{lem13}.3. Let $B_{q_2}^+=q_2q_1y_1y_2w_ix_1+y_2w_{3-i}$ and $B_{y_2}^-$ be an in-tree rooted at $y_2$ such that $A(B_{y_2}^-)\subseteq \{q_1q_2y_2,y_1x_2,w_{3-i}x_1\}\cup (X,Q)_D$.
	It follows that $(B_{q_2}^++w_3^-w_3x_2,B_{y_2}^-+w_iw_3w_3^+)$ is a good pair of $D$, where $w_3^-\neq w_i$ and $w_3^+\neq x_2$ as $\lambda(D)\ge2$, a contradiction.
	
	\vspace{2mm}
	The proofs of the subcase when ``$(Y,X)_D=\{y_2x_i\},~i\in \{1,2\}$'' and the case when ``$|(Y,X)_D|=0$'', see Appendix.
	\hfill $\lozenge$
	
	\vspace{2mm}
	Now both $X$ and $Y$ are independent sets.
	\begin{description}
		\item[Claim \ref{lem13}.5] {\it If both $X$ and $Y$ are independent sets and $(Y,X)_D=\{y_1x_1,y_2x_2\}$, then $D$ has a good pair.}
	\end{description}
	{\it Proof.}
	Suppose that $D$ has no good pair. Now $|[X\cup Y,W]_D|\ge4$. Since $|W|=3$, at least one vertex in $W$ is adjacent to at least two vertices in $X\cup Y$, w.l.o.g., say $w_1$.
	
	{\bf Case 1:} $y_1w_1,y_2w_1\in A$ ($w_1x_1,w_1x_2\in A$).
	
	By the digraph duality, it suffices to prove the case of $y_1w_1,y_2w_1\in A$. By Lemma~\ref{lem12}, there is no arc from $w_1$ to $Y$.
	
	{\bf Subcase 1.1:} $w_1x_1,w_1x_2\in A$.
	
	Let $P_+=y_1x_1+y_2w_1x_2$ and $P_-=y_1w_1x_1+y_2x_2$. Then $(B^++P_+,B^-+P_-)$ is a good pair of $D-\{w_2,w_3\}$, which implies that $D$ has a good pair by Lemma~\ref{lem1}, a contradiction.
	
	{\bf Subcase 1.2:} $w_1x_2\in A$ but $w_1x_1\notin A$ ($w_1x_1\in A$ but $w_1x_2\notin A$).
	
	By the digraph duality, it suffices to prove the case when $w_1x_2\in A$ but $w_1x_1\notin A$.
	Then $N_{D-x_2}^+(w_1)\subset W$. Since $(W,x_1)_D\neq \emptyset$, w.l.o.g., assume $w_2x_1\in A$. Let $P^1=y_1x_1+y_2w_1x_2$ and $P^2=y_2x_2+y_1w_1+w_2x_1$.
	
	First assume $w_1w_2\in A$.
	If $(w_2,X\cup Y)_D\neq \emptyset$, say $w_2^-w_2\in (w_2,X\cup Y)_D$, then $(B^++P^1+w_2^-w_2,B^-+P^2+w_1w_2)$ is a good pair of $D-w_3$, a contradiction.
	Hence $w_3w_2\in A$. Let $P_+=w_3^-w_3w_2$ and $P_-=w_1w_2+w_3w_3^+$, where $w_3^-\in X\cup Y$ and $w_3^+\neq w_2$. It follows that $(B^++P_++P^1,B^-+P_-+P^2)$ is a good pair of $D$, a contradiction.
		Thus $w_1w_3\in A$ but $w_1w_2\notin A$. Then $(B^++P^2+w_2^-w_2+w_1w_3,B^-+P^1+w_2w_2^++w_3w_3^+)$ is a good pair of $D$, where $w_2^-,w_2^+\neq x_1$ and $w_3^+\neq w_2$ as $\lambda(D)\ge2$, a contradiction.
	
	\vspace{2mm}
	Discussions of subcases when ``$(w_1,X)_D=\emptyset$ and $w_2x_1,w_2x_2\in A$'' and ``$w_2x_1,w_3x_2\in A$'' and cases when ``$y_2w_1,w_1x_2\in A$'' and ``$y_2w_1,w_1x_1\in A$'' are given in Appendix.
	\hfill $\lozenge$
	
	\begin{description}
		\item[Claim \ref{lem13}.6] {\it If both $X$ and $Y$ are independent sets and $(Y,X)_D=\{y_1x_1,y_2x_1\}$, then $D$ has a good pair.}
	\end{description}
	{\it Proof.}
	Suppose that $D$ has no good pair. Now $d_W^-(x_2)\ge2$ and $d_W^+(y_j)\ge1$ for any $j\in[2]$. W.l.o.g., assume $w_1x_2,w_2x_2\in A$.
	
	{\bf Case 1:} $d_Y^-(w_i)\ge1$ for any $i\in[2]$.
	
	W.l.o.g., assume $y_1w_1,y_2w_2\in A$. Since $D\nsupseteqq E_3$, $D[\{w_1,w_2\}]\neq C_2$, w.l.o.g., say $w_1w_2\notin A$.
	If $w_1w_3,w_3w_2\in A$, then $(B^++y_1w_1w_3w_2x_2+y_2x_1,B^-+y_1x_1+y_2w_2w_2^++w_1x_1+w_3w_3^+)$ is a good pair of $D$, where $w_2^+\neq x_2$ and $w_3^+\neq w_2$ as $\lambda(D)\ge2$, a contradiction.
	If $w_1w_3,w_3w_2\notin A$, then $(B^++y_1w_1x_2+y_2x_1+w_2^-w_2,B^-+y_1x_1+y_2w_2x_2+w_1w_1^+)$ is a good pair of $D$, where $w_2^-\neq y_2$ and $w_1^+\neq x_2$, a contradiction.
	Otherwise $D$ has a good pair $(B^++y_1w_1x_2+y_2x_1+w_3^-w_3+w_2^-w_2,B^-+y_1x_1+y_2w_2x_2+w_1w_1^++w_3w_3^+)$, where $w_2^-\neq y_2,w_1^+\neq x_2,w_3^-\neq w_2$ and $w_3^+\neq w_1$ as $\lambda(D)\ge2$, a contradiction.
	
	Then at least one of $w_1$ and $w_2$ has no in-neighbours in $Y$, w.l.o.g., say $d_Y^-(w_1)=1$, that is $N^-(w_1)\cap W\neq \emptyset$ by Lemma~\ref{lem12}.
	
	\vspace{2mm}
	Discussions of cases when ``$d_Y^-(w_2)=2$'', ``$d_Y^-(w_2)=1$'' and ``$(N^-(w_1)\cup N^-(w_2))\cap Y=\emptyset$'' are given in Appendix.
	\hfill $\lozenge$
	
	\vspace{2mm}
	By the digraph duality, we also prove the case of $(Y,X)_D=\{y_1x_1,y_1x_2\}$.
	Now $|(Y,X)_D|\le1$.
	
	\begin{description}
		\item[Claim \ref{lem13}.7] {\it If both $X$ and $Y$ are independent sets, then $(Y,X)_D=\emptyset$.}
	\end{description}
	{\it Proof.}
	Suppose $|(Y,X)_D|=1$. W.l.o.g., assume $y_1x_1\in (Y,X)_D$. Now $|N_W^+(Y)|\ge2$ and $|N_W^-(X)|\ge2$ as $\lambda(D)\ge2$.
	
	{\bf Case 1:} $|N_W^+(Y)|=2$ ($|N_W^-(X)|=2$).
	
	By the digraph duality, it suffices to prove the case of $|N_W^+(Y)|=2$.
	W.l.o.g., assume $y_1w_1,y_2w_1,y_2w_2\in A$.
	
	{\bf Subcase 1.1:} $w_1x_2\in A$.
	
	{\bf A.} $w_2x_2,w_3x_1\in A$.
	
	We find a good pair $(B^++P_+,B^-+P_-)$ of $D$ as follows, a contradiction.
	
	~\\
	\vspace{2mm}
	\begin{tabular}{l|c|c}\hline
		Case & $P_+,P_-$ & Notation \\\hline
		$w_2w_3\notin A$ & $y_1w_1y_2w_2x_2+w_3^-w_3x_1,y_1x_1+y_2w_1x_2+w_2w_2^++w_3w_3^+$ & $w_3^-,w_3^+\neq x_1;w_2^+\neq x_2$\\\hline
		$w_2w_3\in A,w_1y_2\notin A$& $y_1w_1x_2+y_2w_2w_3x_1,y_1x_1+y_2w_1w_1^++w_2x_2+w_3w_3^+$ & $w_1^+\neq x_2,w_3^+\neq x_1$\\\hline
		$w_2w_3,w_1y_2\in A$& $y_1x_1+y_2w_1x_2+w_2^-w_2+w_3^-w_3,y_1w_1y_2w_2w_3x_1$& $w_2^-\neq y_2,w_3^-\neq w_2$\\\hline
	\end{tabular}
	
	Discussions of ``$w_2x_1,w_2x_2\in A$'', ``$w_2x_2,w_1x_1\in A$'', ``$w_1x_1,w_1x_2\in A$'', ``$w_2x_1,w_3x_2\in A$'' and ``$w_1x_1,w_3x_2\in A$'' are shown in Appendix.
	
	The proof of the subcase when ``$w_2x_2,w_3x_2\in A$'' and the case when ``$|N_W^+(Y)|=3$'', see Appendix.
	\hfill $\lozenge$
	
	\vspace{2mm}
	Now we are ready to finish the proof of Lemma~\ref{lem13}. Since $|(Y,X)_D|=0$ and both $X$ and $Y$ are independent sets, $|(Y,W)_D|\ge4$ and $|(W,X)_D|\ge4$. By Lemma~\ref{lem12}, for any $w\in W$, $d_Y^-(w)\le2$ and $d_X^+(w)\le2$. This implies that there exist vertices $w_i$ and $w_j$ in $W$ such that $y_1w_i,y_2w_i,w_jx_1,w_jx_2\in A$. Note that it is possible that $w_i=w_j$.
	
	{\bf Case 1:} At least two vertices in $W$ have two in-neighbours in $Y$.
	
	W.l.o.g., assume $y_1w_1,y_2w_1,y_1w_2,y_2w_2\in A$. Note that $|(Y,W)_D|\ge5$.
	
	{\bf Subcase 1.1:} $j\in \{1,2\}$.
	
	W.l.o.g., assume $j=2$.
	
	First assume $d_X^+(w_1)\ge1$, w.l.o.g., say $w_1x_1\in A$. Then $w_1$ has an out-neighbour $w_1^+\neq x_1,y_2$ as $\lambda(D)\ge2$ and Lemma~\ref{lem12}. Set $P_+=y_1w_1x_1+y_2w_2x_2$ and $P_-=y_1w_2x_1+y_2w_1w_1^+$.
	If $w_1^+\neq w_3$, then $(B^++P_+,B^-+P_-)$ is a good pair of $D-w_3$, a contradiction.
	Hence $w_1^+=w_3$. This implies that $d_X^+(w_3)\ge1$. It follows that $(B^++P_++w_3^-w_3,B^-+P_-+w_3w_3^+)$ is a good pair of $D$, where $w_3^-\neq w_1$ as $\lambda(D)\ge2$ and $w_3^+\in X$, a contradiction.
	
	Next assume $d_X^+(w_1)=0$. This implies that $N^+(w_1)\subset W$ and $d_X^+(w_3)=2$ by Lemma~\ref{lem12}, i.e., $w_1w_2,w_1w_3,w_3x_1,w_3x_2\in A$. It follows that $(B^++y_1w_1w_3x_1+y_2w_2x_2,B^-+y_1w_2x_1+y_2w_1w_2+w_3x_2)$ is a good pair of $D$, a contradiction.
	
	{\bf Subcase 1.2:} $d_X^+(w_1)=d_X^+(w_2)=1$.
	
	By Lemma~\ref{lem12}, $d_X^+(w_3)=2$, i.e., $w_3x_1,w_3x_2\in A$. W.l.o.g., assume $w_1x_1,w_2x_2\in A$. It follows that $(B^++y_1w_1+y_2w_2x_2+w_3^-w_3x_1,B^-+w_3x_2+y_2w_1x_1+y_1w_2w_2^+)$ is a good pair of $D$, where $w_3^-\notin \{w_2\}\cup X$ and $w_2^+\neq x_2,y_1$ as $\lambda(D)\ge2$ and Lemma~\ref{lem12}, a contradiction.
	
	{\bf Case 2:} Only $w_i$ in $W$ has $d_Y^-(w_i)=2$.
	
	W.l.o.g., assume $w_i=w_1$, i.e., $y_1w_1,y_2w_1\in A$. Note that $|(Y,W)_D|=4$ and $d_Y^-(w_2)=d_Y^-(w_3)=1$, w.l.o.g., say $y_1w_2,y_2w_3\in A$.
	
	{\bf Subcase 2.1:} $d_X^+(w_2)=d_X^+(w_3)=2$.
	
	That is $w_2x_1,w_2x_2,w_3x_1,w_3x_2\in A$. It follows that $(B^++y_2w_3x_1+y_1w_1+w_2^-w_2x_2,B^-+y_1w_2x_1+w_3x_2+y_2w_1w_1^+)$ is a good pair of $D$, where $w_2^-\neq y_1,x_2$ and $w_1^+\neq w_2,y_2$ as $\lambda(D)\ge2$ and Lemma~\ref{lem12}, a contradiction.
	
	Discussions of subcases when ``Exactly one of $w_2$ and $w_3$ has two out-neighbours in $X$'' and ``$d_X^+(w_2)=d_X^+(w_3)=1$'' are shown in Appendix.
	
	This compeltes the proof of Lemma~\ref{lem13}.
\end{proof}

\begin{prop}\label{h3-1}
	Let $D=(V,A)$ be a 2-arc-strong oriented graph on 9 vertices without $K_4$ as a subdigraph. If $D$ have two cycles $C^1$ and $C^2$ with $C^1\cap C^2=\emptyset$ which cover 8 vertices, then $D$ contains a Hamilton dipath.
\end{prop}
\begin{proof}
	The proof is similar as that of Proposition~\ref{h2-1}. We will give some details of it in Appendix.
\end{proof}

\begin{lemma}\label{h3-2}
	Let $D=(V,A)$ be a 2-arc-strong digraph on 9 vertices without good pair. If $D$ is an oriented graph without $K_4$ as a subdigraph, then $D$ has a Hamilton dipath.
\end{lemma}
\begin{proof}
	The proof is similar as that of Proposition~\ref{h2-2}. We will give some details of it in Appendix.	
\end{proof}

Now we are ready to show Theorem~\ref{thm4}. For convenience, we restate it here.
\begin{description}
	\item[Theorem~\ref{thm4}.] {\it Every 2-arc-strong digraph on 9 vertices has a good pair.}
\end{description}
\begin{proof}
	By contradiction, suppose that $D$ has no good pair.
	
	\begin{description}
	\item[Claim \ref{thm4}.1] {\it No subdigraph of $D$ of order at least 4 has a good pair.}
\end{description}
{\it Proof.} Suppose that a subdigraph $Q$ of $D$ of order at least 4 has a good pair. If $|Q|\geq 5=n-4$, then $D$ has a good pair by Lemmas \ref{lem1}, \ref{lem5} and \ref{lem6}. Thus, assume $|Q|=4=n-5$. Set $X=N_D^-(Q)$ and $Y=N_D^+(Q)$.

If $X\cap Y\neq \emptyset$, then there is a vertex $v$ in $D-Q$ which has both an in-neighbour and an out-neighbour in $Q$. By Lemma \ref{lem1}, $D[V(Q)\cup\{v\}]$ has a good pair. Since $|Q\cup\{v\}|=5=n-4$, $D$ has a good pair by Lemma \ref{lem6}. Henceforth, assume $X\cap Y= \emptyset$.

If $|X|\geq 2$ or $|Y|\geq 2$, then $D$ has a good pair by Lemma \ref{lem7}. Therefore, $|X|=|Y|=1$. Set $Q=\{q_1,q_2,q_3,q_4\}$, $X=\{x\}$, $Y=\{y\}$ and $V-V(Q)-X-Y=W=\{w_1,w_2,w_3\}$. Moreover, let $(B^+_s,B^-_t)$ be a good pair of $Q$, where $s,t\in V(Q)$.
Since $\lambda(D)\ge2$, $N_D^-(Q)=\{x\}$ and $N_D^+(Q)=\{y\}$, there are at least two out-neighbours of $x$ and at least two
in-neighbours of $y$ in $Q$.

\begin{description}
	\item[Note~\ref{thm4}.1] $xs,ty\notin A$.
\end{description}
{\it Proof.}
By the digraph duality, it suffices to prove the case of $xs\notin A$. Suppose $xs\in A$. Let $e_x(\neq xs)$ be another arc from $x$ to $Q$. Then $(xs+B^+_s,e_x+B^-_t)$ is a good pair of $D[V(Q)\cup\{x\}]$, which implies that $D$ has a good pair since $|Q\cup\{x\}|=5$, a contradiciton.
\hfill $\lozenge$

\vspace{2mm}
Next, we distinguish two cases:

{\bf Case 1:} $s=t$. 

W.l.o.g., suppose $s=t=q_1$. By Note~\ref{thm4}.1, neither $x$ nor $y$ is adjacent to $q_1$. Thus, $N_D(q_1)\subseteq \{q_2,q_3,q_4\}$.
Since $\lambda(D)\ge2$, $d^+(q_1)\geq 2$ and $d^-(q_1)\geq 2$, which implies that there is a vertex $q_i\in \{q_2,q_3,q_4\}$ such that
$D[\{q_1,q_i\}]$ is a digon. W.l.o.g., let $i=2$. By Proposition \ref{prop1}, $Q$ has a good pair
$(B^+_{q_2},B^-_{q_2})$. Again by Note~\ref{thm4}.1, neither $x$ nor $y$ is adjacent to $q_2$. Thus, $N_D(\{q_1,q_2\})=\{q_3,q_4\}$
and $N^+(x)\cap Q=N^-(y)\cap Q=\{q_3,q_4\}$. Note that, it is impossible that $D[\{q_i,q_j\}]$ is a digon, where
$i=1$ or $2$ and $j=3$ or $4$. Otherwise, by Proposition \ref{prop1}, $Q$ has a good pair $(B^+_{q_j},B^-_{q_j})$, a contradiction.
But $xq_j\in A$ and $q_jy\in A$ ($j=3$ and $4$), contradicting Note~\ref{thm4}.1. Hence, one of the vertices $q_3$ and $q_4$ is an
out-neighbour of $q_1$ (resp. $q_2$) and the other an in-neighbour of $q_1$ (resp. $q_2$). By symmetry, assume $q_1q_3,q_4q_1\in A$.

Now, if $q_2q_3$, $q_4q_2\in A$, then $D[V(Q)\cup\{x\}]$ has a good pair $(B_x^+,B_{q_3}^-)$ with $B^+_x=xq_4q_2q_1q_3$ and $B^-_{q_3}=q_4q_1q_2q_3+xq_3$. Since $|Q\cup\{x\}|=5=n-4$, by Lemma \ref{lem6}, $D$ has a good pair, a contradiction. Hence, $q_3q_2$, $q_2q_4\in A$.
Figure \ref{fig1} shows the subgraph $H$ of $D$, which contains an out-branching $\hat{B}^+_x=xq_4q_1q_2+xq_3y$ of $H$
and an in-branching $\hat{B}^-_y=q_1q_3q_2q_4y$ of $H\setminus \{x\}$. If $yx\in A$, then $(\hat{B}^+_x,\hat{B}^-_y+yx)$
is a good pair of $H$, which implies that $D$ has a good pair by Lemma \ref{lem5}, a contradiction. Thus, assume $yx\notin A$.

\begin{figure}[!htpb]
	\centering\includegraphics[scale=0.8]{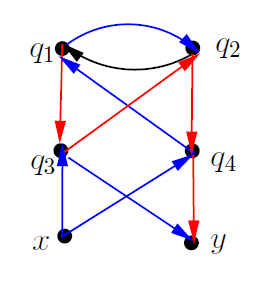}
	\caption{The subgraph $H$ of $D$. The arcs of an out-branching of $H$ are in blue and the arcs of an in-branching of $H\setminus \{x\}$ are in red.}
	\label{fig1}
\end{figure}

Since $N_D^-(Q)=\{x\}$ and $N_D^+(Q)=\{y\}$, $Q$ contains no out-neighbours of $y$ or in-neighbours of $x$. Moreover, since $yx\notin A$,
$N^+_D(y)\subseteq W=\{w_1,w_2,w_3\}$ and $N^-_D(x)\subseteq W$. Because $\lambda(D)\ge2$, $d^+(y)\geq 2$ and $d^-(x)\geq 2$.
Thus, there is a vertex $w_i\in W$ such that $yw_i,w_ix\in A$. W.l.o.g., let $i=1$, i.e., $yw_1,w_1x\in A$.

If $xw_1\in A$, then $(\hat{B}^+_x+xw_1,\hat{B}^-_y+yw_1x)$ is a good pair of $D[V(H)\cup\{w_1\}]$,
which implies that $D$ has a good pair by Lemma \ref{lem1}, a contradiction. By the digraph duality, we also get a contradiction when $w_1y\in A$. Therefore, assume $xw_1,w_1y\notin A$.

Now, suppose that there is another vertex $w_i(\neq w_1)\in W$ such that $yw_i,w_ix\in A$. W.l.o.g., let $i=2$.
If $w_jw_{3-j}\in A$ ($j=1$ or $2$), then $(w_jx+\hat{B}^+_x+yw_{3-j},\hat{B}^-_y+yw_jw_{3-j}x)$ is a good pair of
$D[V(H)\cup\{w_1,w_2\}]$, which implies that $D$ has a good pair by Lemma \ref{lem1}, a contradiction. Hence, $w_1$ and $w_2$ are
not adjacent. Note that, $xw_1,w_1y\notin A$ and for a similar reason, $xw_2,w_2y\notin A$. For both $i=1$ and $2$, since
$d^+(w_i)\geq 2$ and $d^-(w_i)\geq 2$, $D[\{w_i,w_3\}]$ is a digon. Then,
$(w_1w_3w_2x+\hat{B}^+_x,\hat{B}^-_y+yw_1x+w_2w_3w_1)$ is a good pair of $D$, a contradiction. Thus, $w_1$ is the unique vertex which is both an out-neighbour of $y$ and an in-neighbour of $x$.
Now, w.l.o.g., suppose that $yw_2\in A$ and $w_3x\in A$, and moreover, $w_2x\notin A, yw_3\notin A$.

Next, consider the following two situations.

{\bf (A):} $w_2y\in A$.

Since $d^+(w_2)\geq 2$ and $w_2x\notin A$, at least one of the vertices $w_1$ and $w_3$ is an out-neighbour of $w_2$.
If $w_2w_1\in A$, then $(\hat{B}^+_x+yw_2w_1,\hat{B}^-_y+yw_1x+w_2y)$ is a good pair of $D-\{w_3\}$,
which implies that $D$ has a good pair by Lemma \ref{lem1}, a contradiction. If $w_2w_1\notin A$ but $w_2w_3\in A$,  $w_3w_1\in A$ since $d^-(w_1)\geq 2$ and $xw_1\notin A$. Then, $(\hat{B}^+_x+yw_2w_3w_1,\hat{B}^-_y+yw_1x+w_2y+w_3x)$ is a good pair of $D$,
a contradiction.

{\bf (B):} $w_2y\notin A$.

Since $d^+(w_2)\geq 2$ and $w_2y,w_2x\notin A$,  $w_2w_1,w_2w_3\in A$. Moreover, since $d^+(w_1)\geq 2$ and $w_1y\notin A$,
at least one of the vertices $w_2$ and $w_3$ is an out-neighbour of $w_1$. If $w_1w_2\in A$, then $(\hat{B}^+_x+yw_1w_2w_3,\hat{B}^-_y+yw_2w_1x+w_3x)$ is a good pair of $D$, a contradiction. If $w_1w_3\in A$, then $(\hat{B}^+_x+yw_2w_1w_3,\hat{B}^-_y+yw_1x+w_2w_3x)$ is a good pair of $D$, a contradiction.

Therefore, in this case, we can always obtain a contradiction.

{\bf Case 2:} $s\neq t$. 

W.l.o.g., suppose that $s=q_1$ and $t=q_2$, i.e., $Q$ has a good pair $(B^+_{q_1},B^-_{q_2})$. By Note~\ref{thm4}.1, $xq_1\notin A$.
But $N^-(Q)=\{x\}$ and so $N^-_D(q_1)\subseteq \{q_2,q_3,q_4\}$, namely $d^-_Q(q_1)\geq 2$. Moreover, both $B^+_{q_1}$
and $B^-_{q_2}$ contain an out-arc of $q_1$. So, $d^+_Q(q_1)\geq 2$. Similarly, $d^+_Q(q_2)\geq 2$ and $d^-_Q(q_2)\geq 2$.

Since the root of any out-branching has in-degree zero and the root of any in-branching has out-degree zero, if $q_2q_1\in A$,
then $q_2q_1\notin B^+_{q_1}\cup B^-_{q_2}$. Set $\tilde{B}^-_{q_1}=B^-_{q_2}+q_2q_1-e$, where $e$ is the unique out-arc of $q_1$ in $B^-_{q_2}$.
Obviously, $\tilde{B}^-_{q_1}$ is an in-branching of $Q$ with root $q_1$. So, $Q$ has a good pair $(B^+_{q_1},\tilde{B}^-_{q_1})$,
whose out-branching and in-branching have the same root $q_1$, which is Case 1. Thus, assume $q_2q_1\notin A$.

Since $d^-_Q(q_1)\geq 2$ and $q_2q_1\notin A$, $N^-_D(q_1)=\{q_3,q_4\}$. Moreover, since $d^+_Q(q_1)\geq 2$, $D$ contains
at least one of the arcs $q_1q_3$ and $q_1q_4$. W.l.o.g., let $q_1q_3\in A$, that is $D[\{q_1,q_3\}]$ is a digon.
Similarly, since $d^+_Q(q_2)\geq 2$ and $q_2q_1\notin A$, $N^+_D(q_2)=\{q_3,q_4\}$. Moreover, since $d^-_Q(q_2)\geq 2$, there is an arc
$e$ from $\{q_1,q_3\}$ to $q_2$. Set $\tilde{B}^+_{q_1}=q_1q_3+e+q_2q_4$ and $\tilde{B}^-_{q_1}=q_2q_3q_1+q_4q_1$.
Then, $(\tilde{B}^+_{q_1},\tilde{B}^-_{q_1})$ is a good pair of $Q$, whose out-branching and in-branching have the same root $q_1$, which is Case 1.
The proof is complete.
\hfill $\lozenge$

	\vspace{2mm}
	Let $R$ be a largest clique in $D$. Then $R$ has three vertices by Claim \ref{thm4}.1 and Proposition~\ref{prop4}.	
	
	\begin{description}
		\item[Claim \ref{thm4}.2] {\it No subdigraph of $D$ of order at least 3 has a good pair.}
	\end{description}
	{\it Proof.}
	By Lemma~\ref{lem7}, it suffices to show that there is no $Q\subset D$ on 3 vertices with good pair.
	Suppose to the contrary that $Q$ has a good pair. Analogous to Claim \ref{thm3}.1, $|N_D^+(Q)|\ge2$ and $|N_D^-(Q)|\ge2$ with $N_D^+(Q)\cap N_D^-(Q)=\emptyset$. Thus by Lemma~\ref{lem8}, $D$ has a good pair, a contradiction.
	\hfill $\lozenge$
	
	\vspace{2mm}
	By the claim above, $R$ is a tournament.
	
	\begin{description}
		\item[Claim \ref{thm4}.3] {\it $D$ is an oriented graph.}
	\end{description}
	{\it Proof.}
	Suppose that $D$ has a digon $Q$. Set $X=N_D^-(Q)$ and $Y=N_D^+(Q)$. By Claim~\ref{thm4}.2, $X\cap Y=\emptyset$. Since $\lambda(D)\ge2$, both $X$ and $Y$ have at least two vertices.
	If $|X|+|Y|=4$, then $D$ has a good pair by Lemma~\ref{lem13}, a contradiction. If $|X|+|Y|=5$, then $D$ has a good pair by Lemma~\ref{lem11} and the digraph duality, a contradiction. If $|X|+|Y|=6$, then $D$ has a good pair by Lemma~\ref{lem2}, a contradiction. If $|X|+|Y|=7$, then $D$ has a good pair by Corollary~\ref{cor1}, a contradiction.
	\hfill $\lozenge$
	
	\vspace{2mm}
	Now we are ready to finish the proof of Theorem~\ref{thm3}. By Lemma~\ref{h3-2}, assume that $P_D=x_1x_2\ldots x_9$ is a Hamilton dipath of $D$. Set $D'=D-A(P_D)$. Let $I_i,~i\in[a]$, be the initial strong components in $D'$ and let $T_j,~j\in[b]$, be the terminal strong components in $D'$. Note that $a,b\ge2$ by Proposition~\ref{prop5}. Since $D$ is an oriented graph and $\lambda(D)\ge2$, $|I_i|,|T_j|\ge3$, for any $i\in[a],j\in[b]$.  Since $\lambda(D)\ge2$, $x_1$ has at least two in-neighbours and one out-neighbour in $D'$ and $x_9$ has at least two out-neighbours and one in-neighbour in $D'$. Thus there are only two non-adjacent strong components in $D'$, say $I_1$ and $I_2$, as $n=9$ and $D$ is an oriented graph.
	We distinguish two cases below.
	
	{\bf Case 1:} $|I_1|=4$ and $|I_2|=5$.
	
	If $x_9\in I_1$, then $x_8\in I_2$ as $|R|=3$. Analogously, if $x_1\in I_1$, then $x_2\in I_2$.
	By Proposition~\ref{prop6}, $D$ has a good pair for each cases.
	Henceforth, both $x_1$ and $x_9$ are in $I_2$. Note that at least one of $x_2$ and $x_8$ is in $I_1$ as $|R|=3$. By Proposition~\ref{prop6}, $D$ has a good pair, a contradiction.
	
	{\bf Case 2:} $|I_1|=3$ and $|I_2|=6$.
	
	In this case, $x_1,x_9\in I_2$ and $|A(I_2)|\ge7$. If one of $x_2$ and $x_8$ is in $I_1$, then $D$ has a good pair by Proposition~\ref{prop6}. Thus both $x_2$ and $x_8$ are in $I_2$. Then $V(I_1)=\{x_3,x_5,x_7\}$, which implies that $D$ has a good pair by Proposition~\ref{prop6}, a contradiction.
	
	This completes the proof of Theorem~\ref{thm4}.
\end{proof}

\paragraph{Acknowledgments.} Gu was supported by Natural Science Foundation
of Jiangsu Province (No. BK20170860), National Natural Science Foundation of
China (No. 11701143), and Fundamental Research Funds for the Central Universities. Li was supported by National Natural Science Foundation of China (No. 11301480),
Zhejiang Provincial Natural Science Foundation of China (No. LY18A010002),
and the Natural Science Foundation of Ningbo, China. Shi and Taoqiu are supported by the National Natural Science Foundation of China (No. 11922112),
the Natural Science
Foundation of Tianjin (Nos. 20JCJQJC00090 and 20JCZDJC00840) and the Fundamental Research Funds for the Central Universities, Nankai University.

\appendix
\section{Appendix}\label{app}
In this appendix, we give partial proofs of some lemmas and propositions which are omitted above. We use tables to shorten our partial proofs. It will not be hard for the reader to construct full proofs from our partial proofs.

\subsection{Lemma~\ref{lem6}}
{\bf Case 2:} $P^2_{(y,x)}\cap \{w_1,w_2\}\neq \emptyset$.

~\\
\begin{tabular}{l|c|c}\hline
	Case & $O_D$ & $I_D$\\\hline
	$P^1_{(y,x)}\supseteq \{w_1,w_2\},P^2_{(y,x)}\supseteq \{w_1,w_2\}$ & $O_Q+e_y+P^1_{(y,x)}$ & $P^2_{(y,x)}+e_x+I_Q$\\\hline
	$P^1_{(y,x)}\supseteq \{w_1,w_2\}$, w.l.o.g., $P^2_{(y,x)}\cap \{w_1,w_2\}=w_1$ & $O_Q+e_y+P^1_{(y,x)}$ & $P^2_{(y,x)}+a+e_x+I_Q$\\\hline
	\multicolumn{3}{l}{$P^1_{(y,x)}=yw_ix,P^2_{(y,x)}=yw_jx$, $i\neq j\in[2]$}\\\hline
	$w_iw_j,w_jw_i\in A$ & $O_Q+e_y+yw_iw_jx$ & $yw_jw_ix+e_x+I_Q$\\\hline
	$w_iw_j\notin A$ & $O_Q+e_y+P^1_{(y,x)}+xw_j$ & $P^2_{(y,x)}+a+e_x+I_Q+w_iy$\\\hline
\end{tabular}

\subsection{Lemma~\ref{lem7}}
{\bf Subcase 2.2:} $f_1(w_i)=f_2(w_{3-i})=1$ and $f_1(w_{3-i})=f_2(w_i)=0$, where $i\in[2]$.

W.l.o.g., assume $i=1$. Similarly, since $f_1(w_1)=1$ and $f_1(w_2)=0$,
$w_2$ is a leaf of $\hat{B}^-_1$ and $P_{(y_k,x)}$ containing $w_1$ in $\hat{B}^-_1$, where $y_kw_1\in \hat{B}^-_1$ ($k=1$ or $2$).
Since $f_2(w_1)=0$ and $f_2(w_2)=1$, $w_1$ is a leaf of $\hat{B}^-_2$ and $P_{(y_t,x)}$ containing $w_2$ in $\hat{B}^-_2$, where $y_tw_2\in \hat{B}^-_2$ ($t=1$ or $2$). Next, we consider the following two situations.

{\bf Subcase 2.2.1:} $d_{\hat{B}^-_j- w_{3-j}}^-(w_j)=1$ for some $j\in[2]$.

W.l.o.g., suppose that $j=1$ and $y_kw_1\in A$.
Note that, in any in-branching, any vertex except the root has a unique out-neighbour. In $\hat{B}^-_1$, let $v^*$ be the unique out-neighbour of $w_2$.

~\\
\vspace{1mm}
\begin{tabular}{l|c|c}\hline
	Case & Good Pair of $D$ & Notation\\\hline
	$v^*\neq w_1$ & $B^++P_{(y_t,x)}+a,\hat{B}^-_1+B^-$ & $P_{(y_t,x)}\subseteq \hat{B}^-_2$ and $a\notin \hat{B}^-_1\cup \hat{B}^-_2$ with head $w_1$\\\hline
\end{tabular}

Henceforth, assume $v^*=w_1$, that is $w_2w_1\in \hat{B}^-_1$.

{\bf (A):} $N_{\hat{B}^-_2}^+(w_1)\neq \{w_2\}$.

If $d_{\hat{B}^-_2- w_1}^-(w_2)=1$, then $D$ has a good pair by symmetry and the above discussion. Hence, assume that in $\hat{B}^-_2- w_1$, $w_2$ has two in-neighbours, which must be $y_1$ and $y_2$.
So $y_1w_2x+y_2w_2\subseteq \hat{B}^-_2$.

~\\
\vspace{1mm}
\begin{tabular}{l|c|c}\hline
	Case & Good Pair of $D$ & Notation\\\hline
	$x\in N^+(w_1)$ & $B^++y_2w_2w_1x,
	\hat{B}^-_1-w_2w_1+w_2x+B^-$ & $w_2w_1\in \hat{B}^-_1,w_2x\in \hat{B}^-_2$\\\hline
	$\{y_1,y_2\}\cap N^+(w_1)\neq \emptyset$\\ (w.l.o.g., let $w_1y_1\in \hat{B}^-_2$) & $B^++y_{3-k}w_2x+y_kw_1,\hat{B}^-_1-y_kw_1+y_kw_2+B^-$ & $y_{3-k}w_2x,y_kw_2\in \hat{B}^-_2;y_kw_1\in \hat{B}^-_1$\\\hline
\end{tabular}

{\bf (B):} $N_{\hat{B}^-_2}^+(w_1)=\{w_2\}$.

Since $N_{\hat{B}^-_1- w_2}^-(w_1)=\{y_k\}$, $y_{3-k}w_1\notin \hat{B}^-_1$.
Moreover, since $w_2$ is a leaf of $\hat{B}^-_1$, the out-neighbour of $y_{3-k}$ is $y_k$ or $x$ in $\hat{B}^-_1$.

If $y_{3-k}y_k\in \hat{B}^-_1$, then $y_{3-k}y_kw_1+w_1w_2+P_{(y_t,x)}-y_tw_2+B^-$ contains an in-branching $I_D$ of $D$, where $y_{3-k}y_kw_1\in \hat{B}^-_1$ and $P_{(y_t,x)}\in \hat{B}^-_2$. It follows that $(O_D,I_D)$ is a good pair of $D$ with $O_D=B^++y_tw_2+w_2w_1x$, where $y_tw_2\in \hat{B}^-_2$ and $w_2w_1x\in \hat{B}^-_1$.

If $y_{3-k}x\in \hat{B}^-_1$, then $D$ has a good pair $(O_D,I_D)$ with $O_D=B^++y_kw_1w_2+y_{3-k}x$ and $I_D=\hat{B}^-_2-w_1w_2+w_1v^*+B^-$, where $y_kw_1w_1^+,y_{3-k}x\in \hat{B}^-_1$, $w_1w_2\in \hat{B}^-_2$ and $N_{\hat{B}^-_1}^+(w_1)=\{w_1^+\}$.

{\bf Subcase 2.2.2:} $d_{\hat{B}^-_j- w_{3-j}}^-(w_j)=2$ for any $j\in[2]$.

Hence, both $y_1$ and $y_2$ are in-neighbours of $w_j$ in $\hat{B}^-_j$, for any $j=1$ and $2$. Then $D$ has a good pair $(O_D,I_D)$ with $O_D=B^++y_1w_1+y_2w_2x$ and $I_D=\hat{B}^-_1-y_1w_1+y_1w_2+B^-$, where $y_1w_1\in \hat{B}^-_1$ and $y_2w_2x,y_1w_2\in \hat{B}^-_2$.

~\\
\begin{tabular}{l|c|c}\hline
	Case & $O_D$ & $I_D$\\\hline
	$f_i(w_1)+f_i(w_2)=1,f_{3-i}(w_1)+f_{3-i}(w_2)=0,i\in[2]$ & $B^++P_{(y_j,x)}+a$ & $\hat{B}^-_2+B^-$\\\hline
	& \multicolumn{2}{l}{w.l.o.g., $i=1$, where $a\notin \hat{B}_1^-\cup \hat{B}_2^-$ with head $w_2$}\\\hline
	$f_i(w_1)+f_i(w_2)=0,i\in[2]$ & $B^++y_jx+a_1+a_2$ & $\hat{B}^-_1+B^-$\\\hline
	\multicolumn{3}{l}{w.l.o.g., $i=1$, where $a_1\notin \hat{B}_1^-$ such that the head of $a_1$ is $w_1$ and the tail of $a_1$ is not $w_2$, $a_2\notin \hat{B}_1^-$ with head $w_2$}\\\hline
\end{tabular}

\subsection{Lemma~\ref{lem8}}
\subsubsection{Claim~\ref{lem8}.3}
{\bf Subcase 1.2:} $(Y,x_1)_D=\emptyset$ but $(Y,x_2)_D\neq \emptyset.$

~\\
\vspace{1mm}
\begin{tabular}{l|c|c|c}\hline
	Case & $O_D$ & $I_D$ & Notation\\\hline
	$w_iw_{3-i}\notin A$ & $B^++y_2w_ix_1x_2+w_{3-i}^-w_{3-i}$ & $B^-+y_1y_2w_{3-i}x_1+w_iw_i^+$ & \bb{$w_{3-i}^-\neq y_2,w_i^+\neq x_1$}\\\cline{1-3}
	$D[W]=C_2$ & $B^++y_2w_2w_1x_1x_2$ & $B^-+y_1y_2w_1w_2x_1$ &\\\hline
\end{tabular}

{\bf Subcase 1.3:} $(Y,X)_D=\emptyset.$

{\bf A.} $w_1x_1\in A$.

~\\
\vspace{1mm}
\begin{tabular}{l|c|c|c}\hline
	Case & $O_D$ & $I_D$ & Notation\\\hline
	$w_iw_{3-i}\notin A$ & $B^++y_1w_1x_1x_2+w_2^-w_2$ & $B^-+y_1y_2w_2x_2+w_1w_1^+$ & \bb{$i=1,w_2^-\neq y_2,w_1^+\neq x_1$} \\\cline{1-3}
	$D[W]=C_2$ & $B^++y_2w_2w_1x_1x_2$ & $B^-+y_2y_1w_1w_2x_2$ &\\\hline
\end{tabular}

{\bf B.} $w_1x_1\notin A$ but $w_1x_2\in A$.

Analogous to {\bf A} after interchanging $x_1$ and $x_2$.

{\bf Case 2:} $y_2w_2,w_1x_1\notin A$.

If $x_2x_1\in A$ ($x_2x_1\notin A$), then we get Case 1 after interchanging $x_1$ and $x_2$ ($w_1$ and $w_2$).

\subsubsection{Claim~\ref{lem8}.4}
{\bf Subcase 1.2:} $w_2x_1\notin A$ but $(Y,x_1)_D\neq \emptyset$.

We find a good pair of $D$ as $(B^++P_+,B^-+P_-)$ as follows, a contradiction.

~\\
\vspace{1mm}
\begin{tabular}{l|c|c}\hline
	Case & $P_+,P_-$ & Notation\\\hline
	$w_2^+\notin \{y_1,w_1\}$ & $y_1x_1x_2+y_2w_2+w_1^-w_1,w_1x_i+y_2y_2^++y_1w_2w_2^+$ & $w_1^-\neq y_2;y_2^+\neq w_2$ \\\hline
	\multicolumn{3}{l}{$N^+(w_2)=\{y_1,w_1\}$ ($P'_+=y_1w_2w_1,P'_-=y_2w_2y_1x_1$)}\\\hline
	$w_1x_1\in A$ & $P'_++w_1x_1x_2,P'_-+w_1w_1^+$ & $w_1^+\neq x_1$\\\hline
	$w_1x_2\in A$ & $P'_++w_1x_2+x_1^-x_1,P'_-+w_1w_1^+$ & $x_1^-\neq y_1;w_1^+\neq x_2$\\\hline
\end{tabular}

{\bf Subcase 1.3:} $(Y\cup \{w_2\},x_1)_D= \emptyset$.

Interchange $x_1$ and $x_2$, we get Subcases 1.1 and 1.2.

{\bf Case 2: $|(Y,w_2)_D|=1$.}

{\bf Subcase 2.1:} $(Y\cup \{w_2\},x_1)_D\neq \emptyset$.

We find a good pair of $D$ as $(B^++P_+,B^-+P_-)$ as follows, a contradiction.

~\\
\vspace{1mm}
\begin{tabular}{l|c|c}\hline
	Case & $P_+,P_-$ & Notation\\\hline
	$w_2x_1\in A$ & $y_1w_2x_1x_2+y_2w_1,w_1x_i+y_2y_2^++y_1y_1^++w_2w_2^+$ & $w_2^+\neq x_1;y_1^+\neq w_2;y_2^+\neq w_1$\\\hline
	$y_1x_1\in A$ & $y_1x_1x_2+y_2w_1+w_2^-w_2,w_1x_i+y_2y_2^++y_1w_2w_2^+$ & $w_2^+,w_2^-\neq y_1;y_2^+\neq w_1$\\\hline
	$y_2x_1\in A;\exists y_1^+\neq w_1$ & $y_2x_1x_2+y_1w_2+w_1^-w_1,y_2w_1x_i+y_1y_1^++w_2w_2^+$ & $w_1^-\neq y_2;w_2^+\neq w_1;y_1^+\neq w_2$\\\hline
	$y_2x_1\in A;N^+(y_1)=W$ & $y_2x_1x_2+y_1w_1+w_2^-w_2,y_2w_1x_i+y_1w_2w_2^+$ & $w_2^-,w_2^+\neq y_1$\\\hline
\end{tabular}

{\bf Subcase 2.2:} $(Y\cup \{w_2\},x_1)_D= \emptyset$.

Interchange $x_1$ and $x_2$, we get Subcase 2.1.

\subsection{Lemma~\ref{lem8-1}}
{\bf Subcase 1.2:} $(Y,X)_D=\{y_jx_i,y_{3-j}x_{3-i}\}$, for some $i,j\in[2]$.

W.l.o.g., assume $j=i=1$. There exists an arc from $y_j$ to $W$ and an arc from $W$ to $x_i$, say $e_{y_j}$ and $e_{x_i}$ respectively, where $i,j\in[2]$.

~\\
\vspace{1mm}
\begin{tabular}{l|c|c|c}\hline
	Case & $O_D$ & $I_D$ & Notation\\\hline
	\multicolumn{4}{l}{$N_W^+(y_1)\cap N_W^+(y_2)\neq \emptyset$ (w.l.o.g., say $y_1w_1,y_2w_1\in A$, $P_+=y_1x_1+y_2w_1$ and $P_-=y_2x_2+y_1w_1$)}\\\hline
	$e_{x_2}=w_1x_2$ & $B^++P_-+x_1^-x_1+w_2^-w_2$ & $B^-+P_++e_{x_2}+w_2w_2^+$ & $x_1^-\neq y_1;w_2^-,w_2^+\neq x_1$\\\hline
	$w_1^+\notin \{y_1,w_2\}$ & $B^++P_++w_2^-w_2+e_{x_2}$ & $B^-+P_-+w_1w_1^++w_2w_2^+$ & $w_2^-,w_2^+\neq x_2$\\\hline
	$w_2^+\notin X$ & $B^++y_1w_1w_2x_2+x_1^-x_1$ & $B^-+y_2w_1y_1x_1+w_2w_2^+$ & $x_1^-\neq y_1$\\\hline
	$N^+(w_2)=X$ & $B^++P_-+w_1w_2x_1$ & $B^-+P_++w_1y_1+w_2x_2$ \\\hline
	\multicolumn{4}{l}{$N_W^+(y_1)\cap N_W^+(y_2)=\emptyset,N_W^-(x_1)\cap N_W^-(x_2)=\emptyset$ (w.l.o.g., say $y_1w_1,y_2w_2\in A$)}\\\hline
	$D[W]\neq C_2$ & $B^++y_1w_1x_i+y_2w_2x_{3-i}$ & $B^-+y_1x_1+y_2x_2+w_1w_1^++w_2w_2^+$ & $w_1^+\neq x_i,w_2^+\neq x_{3-i},\forall i$\\\hline
	$D[W]=C_2$ & $B^++y_1w_1w_2x_1+y_2x_2$ & $B^-+y_1x_1+y_2w_2w_1x_2$\\\hline
\end{tabular}

{\bf Cases 2 to 4:}

~\\
\begin{tabular}{l|c|c}\hline
	Case & Good Pair of $D$ & Notation\\\hline
	\multicolumn{3}{l}{$|(Y,X)_D|=3$: $(Y,X)_D=\{y_1x_1,y_1x_2,y_2x_2\}$}\\\hline
	\multicolumn{3}{l}{$e_y=y_2w_1,e_x\in (W,x_1)_D$, $P_+=y_2x_2+y_1x_1$ and $P_-=y_1x_2+e_y+e_x$}\\\hline
	$e_x=w_1x_1$ & $B^++P_++w_2^-w_2+w_1^-w_1,B^-+P_-+w_2w_2^+$ & $w_2^-,w_2^+\neq w_1;w_1^-\neq y_2$\\\hline
	$e_x=w_2x_1$ & $B^++P_++w_2^-w_2+w_1^-w_1,B^-+P_-+w_1w_1^+$ & $w_1^-,w_1^+\neq y_2;w_2^-\neq w_1$\\\hline
	\multicolumn{3}{l}{$|(Y,X)_D|=1$: $(Y,X)_D=\{y_1x_1\}$; $e_y=y_1w_1,e_x\in (W,x_1)_D$}\\\hline
\end{tabular}\\
\begin{tabular}{l|c|c}\hline
	$e_x\neq w_2x_1$ & $B^++y_1w_1+y_2w_2x_2+e_x,B^-+y_1x_1+y_2w_1x_2+w_2w_2^+$ & $w_2^+\neq x_2$\\\hline
	$w_1^+\notin \{y_2,x_2\}$ & $B^++y_1w_1x_2+y_2w_2x_1,B^-+y_1x_1+y_2w_1w_1^++w_2x_2$\\\hline
	$w_2^-\notin \{x_2,y_2\}$ & $B^++y_1x_1+y_2w_1+w_2^-w_2x_2,B^-+y_1w_1x_2+y_2w_2x_1$\\\hline
	$N^+(w_1)=N^-(w_2)=\{x_2,y_2\}$ & $B^++y_2w_1x_2w_2x_1,B^-+y_1w_1y_2w_2x_2$\\\hline
	$|(Y,X)_D|=0$ & $B^++y_1w_1x_1+y_2w_2x_2,B^-+y_1w_2x_1+y_2w_1x_2$\\\hline
\end{tabular}

\subsection{Lemma~\ref{lem8-2}}
\subsubsection{Case 1}
\begin{tabular}{l|c|c|c}\hline
	Case & $O_D$ & $I_D$ & Notation\\\hline
	$(q_1,Y)_D=\emptyset;(q_3,Y)_D\neq \emptyset;x_2q_1\in A$ & $B^++x_1y_j+x_2q_1q_3q_2+q_{y_{3-j}}y_{3-j}$ & $B^-+x_1q_1q_2q_3y_j+x_2x_2^+$ & $x_2^+\neq q_1$\\\hline
	$(q_1,Y)_D=\emptyset;(q_3,Y)_D\neq \emptyset;x_2q_2\in A$ & $B^++x_1y_j+x_2q_2q_1q_3+q_{y_{3-j}}y_{3-j}$ & $B^-+x_1q_1q_2q_3y_j+x_2x_2^+$ & $x_2^+\neq q_2$\\\hline
	$N^+(q_1)\cup N^+(q_3)\subseteq Q$ & $B^++x_1q_1q_3q_2y_1+y_2^-y_2$ & $B^-+x_1y_1+q_3q_1q_2y_2+x_2q_{x_2}$ & $y_2^-\neq q_2$\\\hline
\end{tabular}

\vspace{1mm}
By symmetry, we get the case when $x_iq_3\in A$ analogously. By the digraph duality, we also get a contradiction when $q_1y_j$ or $q_3y_j$ is in $A$, where $j\in[2]$.

~\\
\begin{tabular}{l|c|c|c}\hline
	{\bf Subcase 1.2} & $O_D$ & $I_D$ & Notation\\\hline
	$N(q_i)\subset Q,\forall i\in\{1,3\}$ & $B^++x_1y_1+x_2q_2q_1+q_2q_3+q_2y_2$ & $B^-+x_1q_2y_1+q_1q_2+q_3q_2+x_2x_2^+$ & $x_2^+\neq q_2$\\\hline
\end{tabular}

\subsubsection{Case 2}
{\bf Subcase 2.2:} $x_iq_2\in A$, for some $i\in[2]$.

W.l.o.g., assume $i=1$. By Case 1, $(q_1,Y)_D$, $(q_3,Y)_D$ and $(X,q_3)_D$ are not empty. Set $q_1y_j\in A,~j\in[2]$.

~\\
\vspace{1mm}
\begin{tabular}{l|c|c}\hline
	Case & Good Pair of $D$ & Notation\\\hline
	$x_1q_3,q_3y_{3-j}\in A$ & $B^++x_1q_2q_3q_1+y_j^-y_j+y_{3-j}^-y_{3-j},B^-+x_1q_3y_{3-j}+q_2q_1y_j+x_2q_{x_2}$ & $y_j^-\neq q_1,y_{3-j}^-\neq q_3$\\\hline
	$q_3y_j,x_2q_3\in A$ & $B^++x_1q_2q_1y_j+x_2q_3+q_{y_{3-j}}y_{3-j},B^-+x_1q_3y_j+q_1q_2q_3+x_2x_2^+$ & $x_2^+\neq q_3$\\\hline
	$q_3y_j,x_2q_2\in A$ & $B^++x_1q_3+x_2q_2q_1y_j+q_{y_{3-j}}y_{3-j},B^-+x_1q_2q_3y_j+q_1q_2+x_2x_2^+$ & $x_2^+\neq q_2$\\\hline
	\multicolumn{3}{l}{Therefore, $x_2q_3\in A$. Note that $q_3y_{3-j}\in A$.}\\\hline
	$\exists x_1^+\notin \{q_2\}\cup Y$ & $B^++x_1q_2q_3q_1+y_j^-y_j+y_{3-j}^-y_{3-j},B^-+x_1x_1^++x_2q_3y_{3-j}+q_2q_1y_j$ & \bb{$y_j^-\neq q_1,y_{3-j}^-\neq q_3$}\\\cline{1-2}
	$\exists x_2^+\notin \{q_3\}\cup Y$ & $B^++x_2q_3q_1q_2+y_j^-y_j+y_{3-j}^-y_{3-j},B^-+x_1q_2q_1y_j+x_2x_2^++q_3y_{3-j}$\\\hline
\end{tabular}

Hence $N^+(x_1)\subseteq \{q_2\}\cup Y$ and $N^+(x_2)\subseteq \{q_3\}\cup Y$. By the digraph duality, $N^-(y_j)\subseteq \{q_1\}\cup X$ and $N^-(y_{3-j})\subseteq \{q_3\}\cup X$.

If some vertex in $X$, say $x_i$, dominates both vertices in $Y$, then set $P_+=x_1q_2q_1+x_2q_3+x_iy_j$ and $P_-=x_{3-i}x_{3-i}^++x_iy_{3-j}$, where $x_{3-i}\neq q_2$ when $i=2$ and $x_{3-i}\neq q_3$ when $i=1$.
It follows that $(B^++P_++q_3y_{3-j},B^-+P_-+q_2q_3q_1y_j)$ is a good pair of $D$, a contradiction.
Hence $|N^+(x_1)|=|N^+(x_2)|=2$. Analogously, $|N^-(y_1)|=|N^-(y_2)|=2$.

If $x_1y_1,x_2y_2\in A$, then let $B_{x_1}^+=x_1y_1w_1w_2x_2y_2+x_1q_2q_3q_1$ and $B_{x_1}^-=y_2w_2w_1x_1+y_1x_1$.
If $x_1y_2,x_2y_1\in A$, then let $B_{x_1}^+=x_1y_2x_2y_1w_1w_2+x_1q_2q_3q_1$ and $B_{x_1}^-=y_2w_2w_1x_1+y_1x_1$.
W.l.o.g., assume $q_1y_j\in A$, then $q_3y_{3-j}\in A$, where $j\in[2]$. We find a good pair of $D$ as $(B_{x_1}^+,B_{x_1}^-+x_2q_3y_{3-j}+q_2q_1y_j)$, a contradiction.

\subsection{Proposition~\ref{h1}}
Suppose to the contraty that $P$ is the longest dipath in $D$, where $|P|=6$. Obviously $D$ has no $C_6$ by Proposition~\ref{prop7}. Set $P=x_1x_2x_3x_4x_5x_6$. Since $\lambda(D)\ge2$, $d_D^+(x_6)\ge 2$ and $d_D^-(x_1)\ge 2$. Note that $N_D^+(x_6)\subseteq \{x_2,x_3,x_4\}$ and $N_D^-(x_1)\subseteq \{x_3,x_4,x_5\}$.

~\\
\begin{tabular}{l|c}\hline
	Case & Contradiction \\\hline
	\multicolumn{2}{l}{$N_D^+(x_6)\cap N_D^-(x_1)=\emptyset$}\\\hline
	$N_D^+(x_6)=\{x_2,x_3\},N_D^-(x_1)=\{x_4,x_5\}$ & $O_{D[P]}^+=x_4x_1x_2+x_4x_5+x_4x_6x_3$ and $I_{D[P]}^-=x_5x_6x_2x_3x_4+x_1x_3$\\\hline
	$N_D^+(x_6)=\{x_2,x_4\},N_D^-(x_1)=\{x_3,x_5\}$ & $C_6=x_6x_2x_3x_4x_5x_1x_6$\\\hline
\end{tabular}\\
\begin{tabular}{l|c}\hline
	$N_D^+(x_6)\cap N_D^-(x_1)\neq\emptyset,x_2\in N_D^+(x_6)$ & $C_6=x_6x_2x_3x_4x_5x_1x_6$ (note that $x_3\in N_D^+(x_6)\cap N_D^-(x_1)$)\\\hline
	$N_D^+(x_6)=\{x_3,x_4\},x_4\in N_D^-(x_1)$ & $O_{D[P]}^+=x_2x_6x_4+x_6x_3x_1x_5$ and $I_{D[P]}^-=P_D$\\\hline
	$N_D^+(x_6)=\{x_3,x_4\},x_5\in N_D^-(x_1)$ & $O_{D[P]}^+=x_2x_5x_1x_6x_3+x_6x_4$ and $I_{D[P]}^-=P_D$\\\hline
\end{tabular}

\subsection{Proposition~\ref{h2-2}}
\begin{tabular}{l|c|c}\hline
	Case & Contradiction & Analogous Case \\\hline
	$\{x_3,x_4\} \subset N^-(x_1)$ & $O_{D[X]}^+=x_2x_6x_7x_3x_1+x_3x_4x_5,~I_{D[X]}^-=x_5x_6x_4x_1x_2x_3+x_7x_2$ & $\{x_4,x_5\}\subset N^+(x_7)$\\\hline
	$\{x_3,x_5\}\subset N^-(x_1)$ & $O_{D[X]}^+=P,~T_{x_1}^-=x_2x_5x_1+x_7x_3x_1+x_4x_4^+$ and Note~\ref{h2-2}.2 & $\{x_3,x_5\}\subset N^+(x_7)$\\\hline
	& \multicolumn{2}{l}{where $x_4^+\in D[X]-x_5$}\\\hline
	$\{x_3,x_6\}\subset N^-(x_1)$ & $O_{D[X]}^+=x_3x_1x_5x_6x_7x_2+x_4^-x_4,~I_{D[X]}^-=x_6x_1x_2x_3x_4x_5+x_7x_7^+$ & $\{x_2,x_5\}\subset N^+(x_7)$\\\hline
	& \multicolumn{2}{l}{where $x_4^-\in X-x_3;~x_7^+\in \{x_3,x_5\}$}\\\hline
	$\{x_4,x_5\}\subset N^-(x_1)$ & $O_{D[X]}^+=x_3x_4x_5x_1x_2+x_6^-x_6x_7,~I_{D[X]}^-=x_4x_1x_7x_2x_3+x_5x_6x_6^+$ & $\{x_3,x_4\}\subset N^+(x_7)$\\\hline
	& \multicolumn{2}{l}{where $x_6^-\in\{x_3,x_4\}; x_6^+\in\{x_2,x_3\}$}\\\hline
	$\{x_4,x_6\}\subset N^-(x_1)$ & $N^+(x_1)=x_2$ & $\{x_2,x_4\}\subset N^+(x_7)$\\\hline
\end{tabular}

\vspace{1mm}
Therefore, $N^-(x_1)=\{x_5,x_6\}$ and $N^+(x_7)=\{x_2,x_3\}$.
We get another two $P_7$: $x_7x_2x_3x_4x_5x_6x_1$ and $x_6x_7x_2x_3x_4x_5x_1$, which implies that $N^+(x_1)=x_2$ by Fact~\ref{h2-2}.2 and \ref{h2-2}.3, a contradiction to $\lambda(D)\ge2$.

\subsection{Lemma~\ref{lem11}}
\subsubsection{Note~\ref{lem11}.2}
\begin{tabular}{l|c|c}\hline
	Item & Assumption & Good Pair\\\hline
	1 & $X_1=\{x_1,x_2\}$ & $O_{D\setminus W\cup \{x_3\}}^+=B^++y_1x_1+y_2x_2,~I_{D\setminus W\cup \{x_3\}}^-=B^-+y_1x_2+y_2x_1$\\\hline
	2 & $y_1y_2\in A,~d_Y^-(x_1)=2$ & $O_{D\setminus W\cup \{x_2,x_3\}}^+=B^++y_1x_1,~I_{D\setminus W\cup \{x_2,x_3\}}^-=B^-+y_1y_2x_1$\\\hline
	3 & $x_1x_2\in D[X_1],~d_X^+(y_1)=2$ & $O_{D\setminus W\cup \{y_2,x_3\}}^+=B^++y_1x_1x_2,~I_{D\setminus W\cup \{y_2,x_3\}}^-=B^-+y_1x_2$\\\hline
	4 & $i_1=1,i_2=2$ and $j=1$ & $O_{D\setminus W\cup \{x_3\}}^+=B^++y_1x_1x_2,~I_{D\setminus W\cup \{x_3\}}^-=B-+y_1y_2x_2$\\\hline
	\multicolumn{3}{l}{5: assume that $i_1=1,i_2=2$ and $j=1$}\\\hline
	& $x_1q_i,x_2q_i\in A$ & $O_{D\setminus W\cup \{x_3\}}^+=x_1q_iq_{3-i}y_2+q_iy_1x_2,~I_{D\setminus W\cup \{x_3\}}^-=q_{3-i}q_i+y_1y_2x_1x_2q_i$\\\hline
	\multicolumn{3}{l}{next assume $x_1q_1,x_2q_2\in A$}\\\hline
	& $q_1y_2,q_2y_1\in A$ & $O_{D\setminus W\cup \{x_3\}}^+=x_1x_2+x_1q_1q_2y_1y_2,~I_{D\setminus W\cup \{x_3\}}^-=y_1x_2q_2q_1y_2x_1$\\\hline
	& $q_1y_1,q_2y_2\in A$ & $O_{D\setminus W\cup \{x_3\}}^+=x_1q_1q_2y_2+q_1y_1x_2,~I_{D\setminus W\cup \{x_3\}}^-=y_1y_2x_1x_2q_2q_1$\\\hline
\end{tabular}

\subsubsection{Claim~\ref{lem11}.2}
Assume that $\mathcal{X}=\{X_0,X_1,\ldots,X_k\}$ when $|\mathcal{X}|=k$.

~\\
\begin{tabular}{l|c|c}\hline
	Case & Good pair of $D$ & Notation \\\hline
	\multicolumn{3}{l}{{\bf Subcase 2.2:} $|\mathcal{X}|=2$ (by Proposition~\ref{prop2}, we find $T_{X'}$ of $D'[X']$)}\\\hline
	$e_y=e_x=y_1x_1$ & $B^++y_1x_1+e_{y_2}+T_{X'}+w_2^-w_2,B^-+y_1w_2y_2y_2^++w_1x_1$\\\hline
	& \multicolumn{2}{l}{where $y_2^+\notin e_{y_2}\in (y_2,X_1)_D;w_2^-\neq y_1$}\\\hline
	$e_y\neq e_x$ & $B^++e_y+y_2x_1+T_{X'}+w_2^-w_2,B^-+y_1w_2y_2y_2^++w_1x_1$& $w_2^-\neq y_1;y_2^+\notin e_y$\\\hline
	\multicolumn{3}{l}{{\bf Subcase 2.3:} $|\mathcal{X}|=3$ (assume that $X'=\{x_1,a,b,c\}$ with $X_1=a,X_2=b$)}\\\hline
	$e_x=y_2x_1$ & $B^++y_1w_2a+y_2x_1+w_2b+c^-c,B^-+y_1a+w_2y_2b+w_1x_1$ & $\forall c^-\in N_D^-(c)$\\\hline
	$e_x=e_y=y_1x_1,c=w_1$ & $B^++y_1w_2a+w_2b+w_1^-w_1x_1,B^-+y_1x_1+w_2y_2a+w_1w_1^+$ & $w_1^-,w_1^+\neq x_1$\\\hline
	$e_x=e_y=y_1x_1,c\neq w_1$ & $B^++y_1w_2w_1x_1+w_2b+c^-c,B^-+y_1x_1+w_2y_2b+w_1w_1^+$ & $c^-\neq w_1;w_1^+\neq x_1$\\\hline
	{\bf Subcase 2.4:} $|\mathcal{X}|=4$ & $B^++y_1w_2w_1x_1+w_2x_2+w_2x_3,B^-+y_1x_1+w_2y_2x_2+w_1w_1^+$ & $w_1^+\neq x_1$\\\hline
\end{tabular}

\subsubsection{Claim~\ref{lem11}.4}
{\bf Subcase 3.2:} $e_y\neq e_x$.

By Proposition~\ref{prop10}.\ref{prop2-3}, $e_x=w_2x_1$ and $(y_1,X_1)_D=\{e_y\}$. By Proposition~\ref{prop10}.\ref{prop2-4}, $(y_2,X_1)_D= \emptyset$ since $y_2$ is not in any terminal strong component in $D'[Y']$. That is $(w_2,X_1)_D\neq \emptyset$ as $\lambda(D)\ge2$. By Proposition~\ref{prop10}.\ref{prop2-2}, $d_{X'}^+(w_2)=2$. Note that $|X_1|\le2$ as $(y_2,X'-V(X_0\cup X_1))_D\neq \emptyset$.
If $|X_1|=2$, then $X_1=w_1x_iw_1$ where $i\in\{2,3\}$ by Fact~\ref{lem11}.1. Let $D''=D'+w_1x_1-w_1x_i$. Since $D''[X']$ has only one initial strong component, $D$ has a good pair by Case 2, a contradiction.

Hence $|X_1|=1$. Set $X_1=\{a\}$. We first show that $(w_1,X-x_1)_D=\emptyset$. Suppose that $(w_1,X-x_1)_D\neq\emptyset$, w.l.o.g., say $w_1x_2\in A$. Then $a\neq x_2$. Set $D''=D'+w_1x_1-w_1x_2$. By Fact~\ref{lem11}.2, $x_2$ is an initial strong component of $D''[X']$ with $d_{Y'}^-(x_2)=1$. Let $e_{x_2}$ be the arc from $Y'$ to $x_2$. Note that $e_{x_2}=y_2x_2$ as $a\neq x_2$ and $d_{X'}^+(w_2)=2$. Then $D$ has a good pair by Proposition~\ref{prop10}.\ref{prop2-3}.

Thus $N_{D-x_1}^+(w_1)\subset Y'$.
If $a=w_1$, then $w_1y_2\in A$, as $D\nsupseteqq E_3$. Then change $Y'$ from $Y\cup \{w_2\}$ to $Y\cup \{w_1\}$ and change $X'$ from $X\cup \{w_1\}$ to $X\cup \{w_2\}$. By Claim~\ref{lem11}.2, $D$ has a good pair, a contradiction.
If $a\in \{x_2,x_3\}$, then $D[X']$ only has one initial strong component. We get $T_{X'}$ of $D[X']$ by Proposition~\ref{prop2}. It follows that $(B^++y_1w_2a+T_{X'},B^-+y_2y_1a+w_2x_1+w_1w_1^+)$ is a good pair of $D$, where $w_1^+\in Y'$, a contradiction.

~\\
\begin{tabular}{l|c|c}\hline
	Case & Good pair of $D$ & Notation \\\hline
	\multicolumn{3}{l}{{\bf Case 4:} $|\mathcal{X}|=3$ (assume that $X'=\{x_1,a,b,c\}$ with $X_1=a,X_2=b$)}\\\hline
	$e_x=e_y=y_1x_1$ & $B^++y_1x_1+y_2a+y_2b+c^-c+w_2^-w_2,B^-+y_2y_1w_2a+w_1x_1$ & $w_2^-\neq y_1;\forall c^-\in N_D^-(c)$\\\hline
	$e_x=w_2x_1$ & $B^++y_1w_2x_1+y_2a+y_2b+c^-c,B^-+y_2y_1a+w_2b+w_1x_1$ & $\forall c^-\in N_D^-(c)$\\\hline
	{\bf Case 5:} $|\mathcal{X}|=4$ & $B^++y_1w_2+y_2w_1x_1+y_2x_2+y_2x_3,B^-+y_2y_1x_1+w_2x_2+w_1w_1^+$ & $w_1^+\neq x_1$\\\hline
\end{tabular}

\subsubsection{Case 1}
{\bf Subcase 1.2} $(w_1,X)_D=\{w_1x_1\}$.

That is $|(w_1,Y')_D|\ge1$. Let $D''=D-y_1w_2$. By Proposition~\ref{prop2}, we get arc-disjoint $\mathcal{P}_{X'},~\mathcal{P}_{Y'}$ and $T_{X'},~T_{Y'}$ of $D''$. Let $P_+=\mathcal{P}_{X'}$ and $P_-=\mathcal{P}_{Y'}$. Then $O_D=B^++P_++T_{X'}+y_1w_2$ is an out-branching of $D$. Let $I_D=B^-+P_-+T_{Y'}+w_1w_1^+$, then $I_D$ is arc-disjoint with $O_D$ and $V(I_D)=V$.\textbf{ We will show that $I_D$ is an in-branching of $D$}, i.e., there is no digon in $P_-+w_1w_1^+$.

If $w_2y_1\in A$, then let $P_-=w_2y_1x_1+y_1y_1^+$ where $y_1^+\neq w_1$, which is possible by Note~\ref{lem11}.5. Henceforth assume $N^+(w_2)\cap Y=\emptyset$, i.e., $|(w_2,X')_D|\ge2$. Now $|(Y',X')_D|\ge5$.

If $w_1y_1\in A$, then $w_1w_1^+=w_1y_1$. Since $\lambda(D)\ge2$, there are at least two arcs from $Y'$ to any initial strong component in $D''[X']$. Then let $P_-=y_1x_1+w_2w_2^++y_2y_2^+$ where $w_2^+,y_2^+\in X'$, such that $P_-\cap P_+=\emptyset$ as $(y_1,X')_D=\{y_1x_1\}$.

Thus assume $w_1^+\in \{w_2,y_2\}$, w.l.o.g., say $w_1^+=w_2$. If $D''[X']$ is strong, then let $P_+=e$, where $e$ is an arbitrary out-arc of $y_2$ and
$P_-=w_2w_2^++y_2y_2^+$ where $w_2^+\in X$ and $y_2y_2^+$ is an out-arc of $y_2$ which is different from $e$ as $\lambda(D)\ge2$.

If $D''[\{x_1,w_1,x_2\}]$ is strong, then there are at most two initial strong components in $D''[X']$. Set $V_0=\{x_1,w_1,x_2\}$. It follows that $y_2$ and $w_2$ to any initial strong component has at least one arc respectively by the fact $|(Y',X')_D|\ge5$ and $\lambda(D)\ge2$.
If $D''[X']$ has only one initial strong component, then assume that $e\neq y_1x_1$ is an in-arc of the initial strong component from $y_2$. Let $P_+=e$ and $P_-=w_2w_2^++y_2y_2^+$, where $w_2^+\neq w_1$ and $y_2y_2^+\neq e$.
Otherwise $D''[X']$ has exactly two initial strong components, i.e., $V_0$ and $x_3$. Then let $P_+=y_2x_3+w_2w_2^+$ and $P_-=w_2x_3+y_2y_2^+$ where $w_2^+,y_2^+\in V_0$.
The case when $D''[\{x_1,w_1,x_3\}]$ is strong can be proved analogously.

If $D''[\{x_1,w_1\}]$ is strong, i.e., $D''[\{x_1,w_1\}]=x_1w_1x_1$, then $w_2w_1\notin A$ as $D$ has no subdigraph on 3 vertices with a good pair. This implies that $P_-+w_1w_2$ has no digon, as required.

Henceforth $x_1$ is a strong component in $D''[X']$. If $w_2w_1\notin A$, then $P_-+w_1w_2$ has no digon as required. Thus it is suffices to consider the case of $w_2w_1\in A$. Note that $x_1$ is not an initial strong component in $D''[X']$ as $w_1x_1\in A$. Since $|(w_2,X')_D|\ge 2$, there exists at least one out-arc of $w_2$ to $\{x_2,x_3\}$, w.l.o.g., say $w_2x_2$. Let $P_-=w_2x_2+y_2y_2^+$ and $P_+=x_2^-x_2+x_3^-x_3+w_2w_1$ where $x_2^-\neq w_2$, $y_2^+\neq x_2$ and $x_3^-\neq y_2$ as $\lambda(D)\ge2$.

\subsubsection{Case 2}
Assume that $\mathcal{X}=\{X_0,X_1,\ldots,X_k\}$ when $|\mathcal{X}|=k+1$.

{\bf Subcase 2.1} $e_x=w_2x_1$.

By Proposition~\ref{prop10}.\ref{prop2-3}, $w_2y_1\notin A$, i.e., $|(w_2,X')_D|\ge2$. Note that there are at least two initial strong components in $D'[X']$ by Proposition~\ref{prop10}.\ref{prop2-2}, i.e., $|\mathcal{X}|\ge2$.

~\\
\vspace{1mm}
\begin{tabular}{l|c|c}\hline
	\multicolumn{3}{l}{$|\mathcal{X}|=2$}\\\hline
	Case & Good pair of $D$ & Notation \\\hline
	$|X_1|=3$ & by Proposition~\ref{prop10}.\ref{prop2-2}\\\hline
	$|X_1|=2$ & $B^++y_1w_2+x_2^-x_2w_1x_1+x_3^-x_3,B^-+w_2x_1+y_1w_1x_3+y_2y_2^+$ & $x_2^-\in Y';x_3^-\in X';y_2^+\neq x_2$\\\hline
	\multicolumn{3}{l}{$|X_1|=1$, set $X_1=v$, where $v\in X'-x_1$}\\\hline
	$v=w_1$ & $B^++w_2^-w_2+y_1w_1x_1,B_-+w_1w_1^++y_2y_2^++y_1w_2x_1$ & $w_2^-\in X';w_1^+\neq x_1;y_2^+\neq w_1$ \\\hline
	\multicolumn{3}{l}{$v\in \{x_2,x_3\}$, w.l.o.g., say $v=x_2$}\\\hline
	$x_1w_1\notin A$ & $B_++y_1w_2x_2+w_1^-w_1x_1,B_-+y_1x_2+w_2x_1+w_1w_1^++y_2y_2^+$ & $w_1^-,w_1^+\neq x_1;y_2^+\neq w_1$\\\hline
	$x_1w_1\in A$ & $B_++y_1x_2+w_2^-w_2x_1w_1,B_-+y_1w_2x_2+w_1x_1+y_2y_2^+$ & $w_2^-\notin \{y_1,x_1,w_1\};y_2^+\neq w_2$\\\hline
\end{tabular}

\textbf{Henceforth, $|\mathcal{X}|\ge3$. }
If $|(w_1,X)_D|\ge2$, w.l.o.g., assume $w_1x_2\in A$, then $y_1x_2\notin A$, or else we find a good pair of $D$ by changing $e_1$ from $w_1x_1$ to $w_1x_2$. Change $e_1$ from $w_1x_1$ to $w_1x_2$ and keep $D'=D-\{e_1,e_2\}$, now $X_0=\{x_2\}$ in $D'[X']$, or else $D$ has a good pair. This implies that there is no arc from $X$ to $x_1$ and $x_2$. Since $w_2x_2\notin A$ by Proposition~\ref{prop10}.\ref{prop2-4}, $y_2x_2\in A$ as $\lambda(D)\ge2$. Note that $N^-(x_1)=W$ and $N^-(x_2)=\{w_1,y_2\}$, which implies that $N_{D-w_2}^+(y_1)\cup N_{D-x_1}^+(w_2)\cup N_{D-x_2}^+(y_2)\subseteq \{w_1,x_3\}$. By Propositions~\ref{prop10}.\ref{prop2-3} and \ref{prop10}.\ref{prop2-4}, $N_{D-w_2}^+(y_1),N_{D-x_1}^+(w_2),N_{D-x_2}^+(y_2)$ are respectively in some initial strong components in $\mathcal{X}$. Note that $w_1$ and $x_3$ are both initial strong components in $\mathcal{X}$ by the fact that $|\mathcal{X}|\ge3$, which implies that $d_{Y'}^-(w_1)\ge2$ and $d_{Y'}^-(x_3)\ge2$. Since $|(y_1,\{w_1,x_3\})_D|=1$, at least one of $y_2$ and $w_2$ has three out-neighbours in $X'$, which implies that $D$ has a good pair by Proposition~\ref{prop10}.\ref{prop2-2}, a contradiction.

Thus $(w_1,X)_D=\{w_1x_1\}$. Namely $(w_1,Y')_D\neq \emptyset$.
If $|\mathcal{X}|=4$, then $|(Y',X'-x_1)_D|\ge6$ as $\lambda(D)\ge2$. Since $d_{X'}^+(y_1)=1$, $d_{X'}^+(w_2)\ge3$, which implies that $D$ has a good pair by Proposition~\ref{prop10}.\ref{prop2-2}, a contradiction.
Hence $|\mathcal{X}|=3$. Set $\mathcal{X}=\{X_0,X_1,X_2\}$. Recall that $X_0=x_1$. Since there is no $C_2$ in $X$ and $(w_1,X)_D=\{w_1x_1\}$, $|X_1|=|X_2|=1$.
If $y_1w_1\in A$, then $w_1$ is non-adjacent to $y_2$ by Claims~\ref{lem11}.2 and \ref{lem11}.3. By Proposition~\ref{prop10}.\ref{prop2-3}, $w_1$ is an initial strong component in $\mathcal{X}$, which implies that $|[y_1\cup w_2,w_1]_D|\ge2$, i.e., $D[\{y_1\}\cup W]\supseteq E_3$, a contradiction.
If $y_2w_1\in A$, then likewise $w_1$ is non-adjacent to $y_1$. Now $(w_2,X-x_1)_D=\emptyset$, or else we get $D$ has a good pair by changing $X'$ from $X\cup \{w_1\}$ to $X\cup \{w_2\}$ and $Y'$ from $Y\cup \{w_2\}$ to $Y\cup \{w_1\}$, a contradiction. Thus $w_2w_1\in A$. By Proposition~\ref{prop10}.\ref{prop2-4}, $w_1$ is an initial strong component in $\mathcal{X}$. Since $w_1x_1\in A$, $w_1y_2\notin A$ by Proposition~\ref{prop10}.\ref{prop2-4}, which implies that $w_1w_2\in A$. Then $D[W\cup \{x_1\}]\supseteq E_3$, a contradiction.
Therefore,  $(Y,w_1)_D=\emptyset$, which implies that $w_2w_1\in A$. It follows that $w_1$ is an initial strong component in $\mathcal{X}$ by Proposition~\ref{prop10}.\ref{prop2-4}, namely $N^-(w_1)=\{w_2\}$, a contradiction to $\lambda(D)\ge2$.

{\bf Subcase 2.2} $e_x=y_2x_1$.

Since $\lambda(D)\ge2$ and $y_2$ is not adjacent to $w_2$ and $y_1$, $|(Y',X')_D|\ge3$. Then $|\mathcal{X}|\ge2$ by Proposition~\ref{prop10}.\ref{prop2-2}.

~\\
\begin{tabular}{l|c|c}\hline
	Case & Good pair of $D$ & Notation \\\hline
	\multicolumn{3}{l}{$|\mathcal{X}|=2$}\\\hline
	$|X_1|=3$ & by Proposition~\ref{prop10}.\ref{prop2-2}\\\hline
	$|X_1|=2$ & $B_++y_2x_1+y_1v_1+w_2^-w_2,B_-+y_1w_2x_3+y_2v_2$ & $X_1=w_1x_2w_1;w_2^-\neq y_1$\\\hline
	$|X_1|=1$ & $B_++y_1v+y_2x_1+w_2^-w_2+u_1^-u_1+u_2^-u_2,B_-+w_1x_1+y_2v+y_1w_2w_2^+$ & $w_2^+\neq y_1;w_2^-\neq y_1$\\\hline
	\multicolumn{3}{l}{Note that $X_1=\{v\},X'-\{x_1,v\}=\{u_1,u_2\}$ and $u_1^-u_1,u_2^-u_2\in D[X']$}\\\hline
	\multicolumn{3}{l}{$|\mathcal{X}|=3$ (say $X_1=x_2$)}\\\hline
	\multicolumn{3}{l}{{\bf A.} $|X_2|=2$ (set $y_ix_2\in A$, $P^1=y_ix_2$ and $P^2=y_1w_2x_2+y_2x_1+w_1x_3$)}\\\hline
	$y_{3-i}x_3\in A$ & $B_++P^1+y_{3-i}x_3w_1x_1+w_2^-w_2,B^-+P^2$ & $w_2^-\neq y_1$\\\hline
	$y_{3-i}w_1\in A$ & $B_++P^1+y_{3-i}w_1x_1+w_2^-w_2x_3,B^-+P^2$ & $w_2^-\neq y_2,x_3$\\\hline
	\multicolumn{3}{l}{{\bf B.} $|X_2|=1$}\\\hline
	\multicolumn{3}{l}{$\bullet$ $X_2=w_1$}\\\hline
\end{tabular}\\
\vspace{1mm}
\begin{tabular}{l|c|c}\hline
	$y_1w_1,y_2x_2\in A$ & $B_++y_1w_2w_1x_1+y_2x_2+x_3^-x_3,B_-+y_2x_1+w_2x_2+w_1w_1^+$ & $w_1^+\notin \{x_1,y_1\};x_3^-\neq w_1$\\\hline
	$y_1x_2,y_2w_1\in A$\\$w_1w_2\notin A$ &$B_++y_1x_2+y_2x_1+x_3^-x_3+w_2^-w_2,B_-+y_1w_2x_2+y_2w_1x_1$ & $x_3^-\in X'$\\\hline
	\multicolumn{3}{l}{$y_1x_2,y_2w_1,w_1w_2\in A$}\\\hline
	$x_1q_2\in A$ & $q_1q_2y_2w_1x_1+w_1w_2x_2+y_1^-y_1+x_3^-x_3,y_2x_1q_2q_1y_1w_2w_1+x_2x_2^++x_3x_3^+$\\\hline
	& \multicolumn{2}{l}{where $x_2^+,x_3^+\in Q,y_1^-\neq q_1,x_3^-\neq y_1$}\\\hline
	$x_1q_1,x_2q_2\in A$ & $q_1q_2y_2w_1x_1+w_1w_2x_2+y_1^-y_1+x_3^-x_3,y_2x_1q_1y_1w_2w_1+x_2q_2q_1+x_3x_3^+$ & $x_3^+\in Q;y_1^-\neq q_1;x_3^-\neq y_1$\\\hline
	\multicolumn{3}{l}{$\bullet$ $X_2=x_3$}\\\hline
	$y_1x_3,y_2x_2\in A$ & $B_++y_1w_2x_3+y_2x_2+w_1^-w_1,B_-+y_1x_3+y_2x_1+w_2x_2+w_1w_1^+$ & $w_1^+,w_1^-\neq x_1$\\\hline
	$y_1x_2,y_2x_3\in A$ & $B_++y_1w_2x_2+y_2x_3+w_1^-w_1,B_-+y_1x_2+y_2x_1+w_2x_3+w_1w_1^+$ & $w_1^+,w_1^-\neq x_1$\\\hline
\end{tabular}

Henceforth assume $|\mathcal{X}|=4$, i.e., $w_1,x_2,x_3$ are all initial strong components in $\mathcal{X}$. Then $|(Y',X'-x_1)_D|\ge6$. Since $d_{X'}^+(y_1)=1$, $|(\{y_2,w_2\},X'-x_1)_D|\ge5$, which implies that $d_{X'}^+(y_2)\ge 3$. Then by Proposition~\ref{prop10}.\ref{prop2-2}, $D$ has a good pair, a contradiction.

\subsection{Lemma~\ref{lem12}}
\subsubsection{Case 1}
\begin{tabular}{l|c|c}\hline
	Case & Contradiction & Reference\\\hline
	\multicolumn{3}{l}{$H_3$}\\\hline
	$d^+_{Y'}(w_j)=1,\forall j\in \{2,3\}$ & at least one of the vertices $w_2$ and $w_3$ has an out-neighbour in $X$ & Claim~\ref{lem12}.1\\\hline
	$d^+_{Y'}(w_2)\geq 2$ & $w_3$ has an out-neighbour in $X$ & Claim~\ref{lem12}.1\\\hline
	$H_4$ & at least one of the vertices $w_2$ and $w_3$ has an out-neighbour in $X$ & Claim~\ref{lem12}.1\\\hline
	$H_5$ & $B^-_D=F^-(H_5)+y_1w_2(=a_5)+w_2y_2(=e_{w_2})+e_{w_3}+e_{x_1}+e_{x_2}$ & Note~\ref{lem12}.2\\\hline
	& \multicolumn{2}{l}{where $e_{w_3}$ is an out-arc of $w_3$ with head in $Y'$}\\\hline
	$H_6$ ($w_2y_1$ or $w_2w_1$ is in $A$ ) & $B^-_D=F^-(H_6)+y_2w_2(=a_6)+w_2s(=e_{w_2})+e_{w_3}+e_{x_1}+e_{x_2}$ & Note~\ref{lem12}.2\\\hline
	& \multicolumn{2}{l}{where $e_{w_3}$ is an out-arc of $w_3$ with head in $Y'$}\\\hline
	\multicolumn{3}{l}{$w_2y_2\in A$ and $d^+_{Y'}(w_2)=1$}\\\hline
	$\exists w_2^+\in X$ & $D$ has a good pair & Claim~\ref{lem12}.1\\\hline
	$w_2w_3\in A$ & $B^-_D=F^-(H_6)+y_2w_2(=a_6)+w_2w_3(=e_{w_2})+e_{w_3}+e_{x_1}+e_{x_2}$ & Note~\ref{lem12}.2\\\hline
\end{tabular}

\subsubsection{Case 2}
{\bf Subcase 2.1:} $d^+_X(w_2)=2$ and $d^+_X(w_3)=2$.

That is, $\hat{A}=\{w_2x_1,w_2x_2,w_3x_1,w_3x_2\}\subseteq A$.
Let $e_{w_2}=w_2x_2$ and $e_{w_3}=w_3x_1$. Moreover, let $D'=(V,A')$ with $A'=A-\{e_{w_2},e_{w_3}\}$ and $F$ be a subdigraph of $D$ such that $V(F)=X'$ and $A(F)=\hat{A}$.

Note that each initial strong component in $D'[X']$ has at least two in-arcs from $Y'$. It follows that for each $1\leq i\leq 6$, $a_i$ and $D''=(V,A'')$ with $A''=A'-a_i$, each initial strong component in $D''[X']$ has an in-arc from $Y'$. Moreover, $F^-(H_i)+e_{x_1}+e_{x_2}+e_{w_2}+e_{w_3}+a_i$ is an in-branching of $D$. Thus $D$ has a good pair by Note~\ref{lem12}.2.

{\bf Subcase 2.2:} $d^+_X(w_2)=1$ and $d^+_X(w_3)=1$.

Since $d^+_{X'}(w_2)\geq 2$ and $d^+_{X'}(w_3)\geq 2$, $w_2w_3,w_3w_2\in A$. Thus, the arc from $w_2$ to $X$ and the arc from $w_3$ to $X$ has different heads. W.l.o.g., assume $w_2x_1,w_3x_2\in A$. Let $\hat{A}=\{w_2w_3,w_3w_2,w_2x_1,w_3x_2\}$ and $F$ be a subdigraph of $D$ such that $V(F)=X'$ and $A(F)=\hat{A}$. Since $D[X']$ has no good pair, by symmetry, $D[X']=F$ or $F+x_1x_2$. Let $D'=(V,A')$ with $A'=A-\{a_i,e_{w_2},e_{w_3}\}$.

Note that each initial strong component of $D'[X']$ has an in-arc from $Y'$ in $D'$. Moreover, $F^-(H_i)+e_{x_1}+e_{x_2}+e_{w_2}+e_{w_3}+a_i$ is always an in-branching of $D$. Thus $D$ has a good pair by Note~\ref{lem12}.2.

{\bf Subcase 2.3:} $d^+_X(w_j)=2$ and $d^+_X(w_{5-j})=1$ for $j=2$ or $3$.

W.l.o.g., assume $d^+_X(w_2)=2$, i.e., $w_2x_1,w_2x_2\in A$, and $d^+_X(w_3)=1$. By symmetry, let $w_3x_2\in A$. Since $d^+_{X'}(w_3)\geq 2$,
$w_3w_2\in A$. Moreover, since $D[\{w_2,w_3,x_2\}]$ can contain at most three arcs, $w_2w_3,x_2w_3\notin A$. Note that $d^-_{D}(w_3)\geq 2$. Thus, for any $a_i$, there is an arc $e\neq a_i$ from $Y'\cup x_1$ to $w_3$.

~\\
\begin{tabular}{l|c|c}\hline
	Case & $O_D$ & $I_D$ \\\hline
	$x_2x_1\in A$ & $B^+(H_i)+e+w_3w_2x_2x_1$ & $F^-(H_i)+e_{x_1}+e_{x_2}+w_2x_1+w_3x_2+a_i$\\\hline
	\multicolumn{3}{l}{Condition ($\ast$): in $D-a_i$, $x_1$ has an in-arc $\hat{e}$ from $Y'$}\\\hline
	\multicolumn{3}{l}{for each $H_i$ where $1\leq i\leq 6$ and $i\neq 5$ or for $H_5$ and $a_5$ satisfies ($\ast$)}\\\hline
	($\ast$) holds & $B^+(H_i)+\hat{e}+e+w_3w_2x_2$ & $F^-(H_i)+e_{x_1}+e_{x_2}+w_2x_1+w_3x_2+a_i$\\\hline
	\multicolumn{3}{l}{For $H_5$ and $a_5$ doesn't satisfy the condition ($\ast$)}\\\hline
	$y_2w_2\in A$ & $B^+(H_i)+y_2w_2+w_2x_1+e+w_3x_2$ & $F^-(H_i)+e_{x_1}+e_{x_2}+y_1x_1(=a_5)+w_3w_2x_2$\\\hline
	$y_2x_2\in A$ & $B^+(H_i)+e+w_3w_2x_1+y_2x_2$ & $F^-(H_i)+e_{x_1}+e_{x_2}+y_1x_1(=a_5)+w_2x_2+w_3x_2$\\\hline
\end{tabular}

\subsubsection{Claim~\ref{lem12}.2}
For $H_i$ ($1\leq i\leq 3$), $a_i$ has two choices, both of which are from a same vertex, which implies that they have different heads. Hence, there exists an $a_i$ with head in $\{x_1,x_2,w_2\}$. For $H_4$ to $H_6$, we show contradictions below.

~\\
\begin{tabular}{l|c}\hline
	Case & Contradiction \\\hline
	$H_4$ & $D[Y'\cup\{w_3\}]$ is either a tournament of order $4$ or contains a subdigraph on $3$ vertices with $4$ arcs\\\hline
	$H_5$ & $D$ has a good pair by Note~\ref{lem12}.2\\\hline
	$H_6$ & $D$ contains no subdigraph with a good pair on at least 3 vertices\\\hline
\end{tabular}

\subsubsection{Case 3}
{\bf Subcase 3.2:} $a_i$ with head $x_2$.

For both situations, let $e_{w_2}=w_2x_1$ and then $F^-(H_i)+e_{x_1}+e_{x_2}+a_i+e_{w_2}+e_{w_3}$ is an in-branching of $D$ for any $1\leq i\leq 6$. If there exists an initial strong component $X_0$ in $D'[X']$ such that $d^-_{Y'}(X_0)=1$, then $x_1\in X_0$ and $w_2\notin X_0$. For $w_2$ and Situation $(a)$, $x_2\notin X_0$, since $w_2x_2\in A(D')$. For Situation $(b)$, if $D[X]$ is a digon, then $D[Q\cup X]$ is $E_4$, which has a good pair, a contradiction. So, $D[X]$ is not a digon. Thus, if $x_2\in X_0$, then $w_3\in X_0$. But, $w_2w_3\in A(D')$, contradicting the assumption that $X_0$ is initial. Hence, we also have $x_2\notin X_0$. Since the head of $a_i$ is $x_2$, in $D''$, each initial strong component of $D''[X']$ has always at least one in-arc from $Y'$. By Note~\ref{lem12}.2, $D$ has a good pair.

{\bf Subcase 3.2:} $a_i$ with head $w_2$.

For both situations, let $e_{w_2}=w_2x_1$ and then $F^-(H_i)+e_{x_1}+e_{x_2}+a_i+e_{w_2}+e_{w_3}$ is an in-branching of $D$ for any $1\leq i\leq 6$. If there exists an initial strong component $X_0$ in $D'[X']$ such that $d^-_{Y'}(X_0)=1$, then $x_1\in X_0$ and $w_2\notin X_0$. Since the head of $a_i$ is $w_2$, in $D''$, each initial strong component of $D''[X']$ has always at least one in-arc from $Y'$. By Note~\ref{lem12}.2, $D$ has a good pair.

\subsection{Lemma~\ref{lem13}}
\subsubsection{Claim~\ref{lem13}.1}
We find a good pair $(B^++P_+,B^-+P_-)$ of a subdigraph $H$ of $D$ as follows, which implies that $D$ has a good pair by Lemma~\ref{lem1}.

{\bf Case 2:} $w_1y_2\in A$ and $N_{D-y_2}^+(w_1)\subseteq W$.

W.l.o.g., assume $w_1w_2\in A$. Note that there exists an arc from $w_2$ to $X\cup Y-y_2$, say $e_1$, as $D\nsupseteqq E_3$.
If $(X\cup Y,w_2)_D\neq \emptyset$ and $(D-w_3,w_1)_D\neq \emptyset$, then set $e_2\in (X\cup Y,w_2)_D$ and $e_3\in (D-w_3,w_1)_D$.

~\\
\begin{tabular}{l|c|c|c}\hline
	Case & $P_+,P_-$ & $H$ & Notation \\\hline
	$\exists e_2,e_3$ & $P^1+e_2+e_3,P^2+y_2w_1w_2+e_1$ & $D-w_3$\\\hline
	$\exists e_2,\nexists e_3$ & $P^1+e_2+w_3^-w_3w_1,P^2+y_2w_1w_2+e_1+w_3w_3^+$ & $D$ & $w_3^-,w_3^+\neq w_1$\\\hline
	$\nexists e_2$ and $w_3w_1\notin A$ & $P^1+w_3^-w_3w_2+w_1^-w_1,P^2+y_2w_1w_2+e_1+w_3w_3^+$ & $D$ & $w_3^-,w_3^+\neq w_2;w_1^-\neq y_2$\\\hline
	\multicolumn{4}{l}{Thus $\nexists e_2$ and $w_3w_1\in A$.}\\\hline
	\multicolumn{4}{l}{{\bf Subcase 2.1:} $|(Y,X)_D|=3$.}\\\hline
\end{tabular}\\
\vspace{1mm}
\begin{tabular}{l|c|c|c}\hline
	$w_3x_1\in A$ & $P^2+w_3^-w_3x_1+y_2w_1,P^1+w_3w_1y_2$ & $D-w_2$ & $w_3^-\in X\cup Y-x_1$\\\hline
	$w_2x_1\in A$ & $P^2+y_2w_1w_2x_1,P^1+w_1y_2+w_2w_2^+$ & $D-w_3$ & $w_2^+\in X\cup Y$\\\hline
	$w_1x_1\in A$ & $P^2+y_2w_1x_1,P^1+w_1y_2$ & $D-\{w_2,w_3\}$\\\hline
	\multicolumn{4}{l}{{\bf Subcase 2.2:} $|(Y,X)_D|=2$. (Let $P_*^1=x_2x_1$ and $P_*^2=(Y,X)_D$.)}\\\hline
	$w_3x_2\in A$ & $P_*^1+w_3^-w_3x_2+y_2w_1,P_*^2+w_3w_1y_2$ & $D-w_2$ & $w_3^-\in Y$\\\hline
	$w_2x_2\in A$ & $P_*^1+y_2w_1w_2x_2,P_*^2+w_1y_2+w_2w_2^+$ & $D-w_3$ & $w_2^+\notin W$\\\hline
	$w_1x_2\in A$ & $P_*^1+y_2w_1x_2,P_*^2+w_1y_2$ & $D-w_2-w_3$\\\hline
\end{tabular}

{\bf Case 3:} $e_1$ exists but $N_{D-y_2}^-(w_1)\subset W$.

W.l.o.g., assume $w_2w_1\in A$. Let $P'_-=y_2w_1+e_1$.
If $(X\cup Y,w_2)_D\neq \emptyset$ and $(w_2,X\cup Y)_D\neq \emptyset$, then set $e_3\in (X\cup Y,w_2)_D$ and $e_4\in (w_2,X\cup Y)_D$.

~\\
\begin{tabular}{l|c|c|c}\hline
	Case & $P_+,P_-$ & $H$ & Notation \\\hline
	$\exists e_3,e_4$ & $P^1+e_3+w_2w_1,P^2+P'_-+e_4$ & $D-w_3$\\\hline
	$\nexists e_3$ & $P^1+w_3^-w_3w_2w_1,P^2+P'_-+e_4+w_3w_3^+$ & $D$ & $w_3^-,w_3^+\in X\cup Y$\\\hline
	$\exists e_3,\nexists e_4$ & $P^1+e_3+w_2w_1+w_3^-w_3,P^2+P'_-+e_4+w_2w_3w_3^+$ & $D$ & $w_3^-\neq w_2;w_3^+\notin W$\\\hline
\end{tabular}

\subsubsection{Claim~\ref{lem13}.2}
{\bf Subcase 2.1:} $(Y,X)_D=y_1x_1$ ($(Y,X)_D=y_2x_2$).

By the digraph duality, it suffices to prove the case of $(Y,X)_D=y_1x_1$.

Now $d_W^+(y_2)\ge2$ and $d_W^-(x_1),d_W^-(x_2)\ge1$. W.l.o.g., assume $y_2w_1,y_2w_2\in A$. Note that there exists at least one arc from $w_1$ or $w_2$ to $X$ by the fact that $|[X,w_i]_D|\le1$. W.l.o.g., assume $w_1x_i\in A,~i\in[2]$. Let $P_+=y_1x_1x_2+y_2w_2$ and $P_-=y_1y_2w_1x_i$. Let $w_1w_1^+$ and $w_1^-w_1$ be respectively an out- and in-arc of $w_1$ such that $w_1^+\neq x_i$ and $w_1^-\neq y_2$. This implies that $w_1^+,w_1^-\in\{w_2,w_3\}$.

~\\
\vspace{1mm}
\begin{tabular}{l|c|c}\hline
	Case & Good Pair of $D$ & Notation \\\hline
	$w_1^-=w_2$ & $B^++P_++w_2w_1w_3,B^-+P_-+w_3w_3^++w_2w_2^+$ & $w_3^+\neq w_2;w_2^+\neq w_1$\\\hline
	$w_1^-=w_3$ & $B^++P_++w_3^-w_3w_1,B^-+P_-+w_2w_2^++w_3w_3^+$ & $w_3^-,w_3^+\neq w_1;w_2^+\neq w_3$\\\hline
\end{tabular}

{\bf Subcase 2.2:} $(Y,X)_D=y_1x_2$.

Likewise $d_W^+(y_2)\ge2$, say $y_2w_1,y_2w_2\in A$. Note that $d_W^-(x_1)\ge2$ and $N^-(x_1)\cap\{w_1,w_2\}\neq \emptyset$, say $w_1x_1\in A$. Let $B_{q_2}^+=q_2q_1y_1y_2w_1x_1x_2+y_2w_2$ and $B_{y_2}^-$ be an in-tree rooted at $y_2$ such that $A(B_{y_2}^-)\subseteq \{q_1q_2y_2,y_1x_2\}\cup (X,Q)_D$. Let $w_1w_1^+$ and $w_1^-w_1$ be respectively an out- and in-arc of $w_1$ such that $w_1^+\neq x_1$ and $w_1^-\neq y_2$. This implies that $w_1^+,w_1^-\in\{w_2,w_3\}$ by the fact that $|[Y,w_i]_D|\le1$ and $|[X,w_i]_D|\le1$.

~\\
\vspace{1mm}
\begin{tabular}{l|c|c}\hline
	Case & Good Pair of $D$ & Notation \\\hline
	$w_2w_1,w_2w_3\in A$ & $B_{q_2}^++w_2w_3,B_{y_2}^-+w_2w_1w_3w_3^+$ & $w_3^+\notin W$\\\hline
	$\exists w_2^+\notin W$ & $B_{q_2}^++w_2w_2^++w_3w_3^++w_1w_1^+,B_{y_2}^-+w_3^-w_3$ & $w_3^+,w_3^-\neq w_1;\forall w_1^+\in N_D^+(w_1)$\\\hline
\end{tabular}

{\bf Subcase 2.3:} $(Y,X)_D=y_2x_1$.

Now $d_W^+(y_i)\ge1$ for any $i\in[2]$, w.l.o.g., say $y_1w_1,y_2w_2\in A$. Let $B_{q_2}^+=q_2q_1y_1y_2x_1x_2+y_2w_2$ and $B_{y_2}^-$ be an in-tree rooted at $y_2$ such that $A(B_{y_2}^-)\subseteq q_1q_2y_2\cup (X,Q)_D$.

~\\
\begin{tabular}{l|c|c}\hline
	Case & Good Pair of $D$ & Notation \\\hline
	$(w_2,X)_D\neq \emptyset$ & $B_{q_2}^++w_3^-w_3+w_1^-w_1,B_{y_2}^-+w_2x_i+y_1w_1w_1^++w_3w_3^+$ & $w_3^+,w_3^-\neq w_1;w_1^-\neq y_1;w_1^+\neq w_3,y_1;i\in[2]$\\\hline
	$(w_2,X)_D=\emptyset$ & $B_{q_2}^++w_2w_1+w_3^-w_3,B_{y_2}^-+y_1w_1x_i+w_2w_3w_3^+$ & $w_3^-,w_3^+\neq w_2;i\in[2]$\\\hline
\end{tabular}

\subsubsection{Claim~\ref{lem13}.3}
We find a good pair $(B^++P_+,B^-+P_-)$ of a subdigraph $H$ of $D$ as follows, which implies that $D$ has a good pair by Lemma~\ref{lem1}.

{\bf Subcase 1.2:} $w_2x_2,w_3x_2\in A$.

\textbf{First assume $w_1w_2\in A$.} Let $P'_+=y_2w_1w_2x_2$.
If $(w_1,X\cup Y)_D\neq \emptyset$ and $(w_2,D-w_3)_D\neq \emptyset$, then set $e_1\in (w_1,X\cup Y)_D$ and $e_2\in (w_2,D-w_3)_D$.

~\\
\vspace{1mm}
\begin{tabular}{l|c|c|c}\hline
	Case & $P_+,P_-$ & $H$ & Notation \\\hline
	$\exists e_1,e_2$ &   $P^1+P'_+,P^2+e_1+e_2$ & $D-w_3$\\\hline
	$\nexists e_1$ & $P^1+P'_++w_3^-w_3,P^2+w_2w_2^++w_1w_3x_2$ & $D$ & $w_2^+\neq x_2;w_3^-\neq w_1$\\\hline
	$\exists e_1,\nexists e_2, w_2w_1\in A$ & $P^1+P'_+,P^2+w_2w_1+e_1$ & $D-w_3$\\\hline
	$\exists e_1,\nexists e_2, w_2w_3\in A$ & $P^1+P'_++w_3^-w_3,P^2+e_1+w_2w_3x_2$ & $D$ & $w_3^-\neq w_2$\\\hline
\end{tabular}

The case of $w_1w_3\in A$ can be proved analogously.

\textbf{Next assume $N^+(w_1) \nsubseteq W$.} Since $D\nsupseteqq E_3$, at least one of $w_2$ and $w_3$ has two in-neighbours, w.l.o.g., say $w_2$. This implies that there exists an arc $e_1$ from $Y$ to $w_2$. Let $P'_+=e_1+w_2x_2$.

If $(w_2,x_1\cup Y)_D\neq \emptyset$, then set $e_2\in (w_2,x_1\cup Y)_D$.

~\\
\vspace{1mm}
\begin{tabular}{l|c|c|c}\hline
	Case & $P_+,P_-$ & $H$ & Notation \\\hline
	$\exists e_2$ & $P^1+P'_+,P^2+e_2$ & $D-\{w_1,w_3\}$\\\hline
	$\nexists e_2, w_2w_1\in A$ & $P^1+P'_++y_2w_1,P^2+w_2w_1w_1^+$ & $D-w_3$ & $w_1^+\in X\cup Y$\\\hline
	$\nexists e_2, w_2w_3\in A$ & $P^1+P'_++w_3^-w_3,P^2+w_2w_3x_2$ & $D-w_1$ & $w_3^-\in X\cup Y$\\\hline
\end{tabular}

{\bf Case 2:} $y_2w_1,w_1x_2\in A$ ($y_1w_1,w_1x_2\in A$).

By the digraph duality, it suffices to prove the case of $y_2w_1,w_1x_2\in A$.

By Case 1, $y_1w_1\notin A$, then w.l.o.g., assume $y_1w_2\in A$. If $(w_1,x_1\cup Y)_D\neq \emptyset$, then set $e_1\in (w_1,x_1\cup Y)_D$.

~\\
\begin{tabular}{l|c|c|c}\hline
	Case & $P_+,P_-$ & $H$ & Notation\\\hline
	$\exists e_1$ & $P^1+y_2w_1x_2,P^2+e_1$ & $D-\{w_2,w_3\}$\\\hline
	\multicolumn{3}{l}{{\bf Subcase 2.1:} $\nexists e_1, w_2x_2\in A$.}\\\hline
	$w_1w_2\in A$ ($w_2w_1\in A$ is analogous) & $P^1+y_2w_1x_2+y_1w_2,P^2+w_1w_2x_2$ & $D-w_3$\\\hline
	$w_1w_3,w_2w_3\in A$ & $P^1+y_2w_1x_2+y_1w_2w_3,P^2+w_1w_3w_3^++w_2x_2$ & $D$ & $w_3^+\notin W$\\\hline
	\multicolumn{3}{l}{{\bf Subcase 2.2:} $\nexists e_1, w_3x_2\in A$.}\\\hline
	$w_1w_2\in A,\exists w_2w_2^+\in (w_2,X\cup Y)_D$ & $P^1+y_2w_1x_2+y_1w_2,P^2+w_1w_2w_2^+$ & $D-w_3$\\\hline
	$w_1w_2\in A,N^+(w_2)\subset W$ & $P^1+y_2w_1x_2+y_1w_2+w_3^-w_3,P^2+w_1w_2w_3x_2$ & $D$ & $w_3^-\notin W$\\\hline
	$w_1w_3\in A,\exists w_3^-w_3\in (X\cup Y,w_3)_D$ & $P^1+y_2w_1x_2+w_3^-w_3,P^2+w_1w_3x_2$ & $D-w_2$\\\hline
	$w_1w_3,w_2w_3\in A;\exists w_3^+\notin w_2 \cup X$ & $P^1+y_2w_1w_3x_2,P^2+w_1x_2+w_3w_3^+$ & $D-w_2$\\\hline
	$w_1w_3,w_2w_3,w_3w_2\in A$ & $P^1+y_2w_1w_3x_2+y_1w_2,P^2+w_3w_2w_2^+$ & $D$ & $w_2^+\neq w_3$\\\hline
\end{tabular}

\subsubsection{Claim~\ref{lem13}.4}
{\bf Subcase 1.2:} $(Y,X)_D=y_2x_i,~i\in \{1,2\}$.

W.l.o.g., assume $i=1$. Note that $(y_1,W)_D\neq \emptyset$, say $y_1w_1\in A$.

\textbf{First assume $w_1x_1,w_1x_2\in A$.} Now $(\{w_2,w_3\},x_2)_D\neq \emptyset$, w.l.o.g., say $w_2x_2\in A$. Set $P_+=y_1w_1x_1$ and $P_-=y_1y_2x_1+w_1x_2$.

~\\
\vspace{1mm}
\begin{tabular}{l|c|c}\hline
	Case & Good Pair of $D$ & Notation\\\hline
	$x_2w_3,w_2w_3\in A$ & $B^++P_++w_2^-w_2x_2w_3,B^-+P_-+w_2w_3w_3^+$ & $w_2^-,w_3^+\notin \{x_2,w_2,w_3\}$\\\hline
	$\exists w_3^-\neq w_2,x_2$ & $B^++P_++w_3^-w_3+w_2^-w_2x_2,B^-+P_-+w_3w_3^++w_2w_2^+$ & $w_2^-,w_2^+\neq x_2;w_3^+\neq w_2$\\\hline	
\end{tabular}

\textbf{Next assume that $w_1x_1\in A$ but $w_1x_2\notin A$.} Since $d_W^-(x_2)\ge2$, $w_2x_2,w_3x_2\in A$. Note that $|(w_1,\{w_2,w_3\})_D|\ge1$ as $w_1x_2\notin A$ by Lemma~\ref{lem12}. W.l.o.g., assume $w_1w_2\in A$.

~\\
\vspace{1mm}
\begin{tabular}{l|l}\hline
	Case & Good Pair of $D$\\\hline
	$w_3w_2\notin A$ & $B^++y_1w_1x_1+w_3^-w_3x_2+w_2^-w_2,B^-+y_1y_2x_1+w_1w_2x_2+w_3w_3^+$\\\hline
	&where $w_2^-\neq w_1;w_3^+\neq x_2;w_3^-\neq x_2,w_2$\\\hline
	$w_3w_2\in A;w_2w_1\notin A$ & $B^++y_1w_1x_1+w_3^-w_3w_2x_2,B^-+y_1y_2x_1+w_3x_2+w_1w_2w_2^+$\\\hline
	&where $w_2^+\neq x_2;w_3^-\neq x_2,w_2$\\\hline
	$w_3w_2,w_2w_1\in A$ & $B^++y_1w_1w_2x_2+y_2x_1+w_3^-w_3,B^-+w_3w_2w_1x_1+y_1y_2y_2^+$\\\hline
	&where $w_3^-\neq y_2,y_2^+\in W$\\\hline
\end{tabular}

\textbf{Henceforth assume that $w_1x_1\notin A$ but $w_1x_2\in A$.} Note that $w_1$ and $y_2$ respectively have an out-neighbour $w_1^+$ and an out-neighbour $y_2^+$ such that $w_1^+,y_2^+\in \{w_2,w_3\}$ by $\lambda(D)\ge2$ and Lemma~\ref{lem12}. W.l.o.g., assume $w_1w_2\in A$.

~\\
\vspace{1mm}
\begin{tabular}{l|c|c}\hline
	Case & Good Pair of $D$ & Notation\\\hline
	\multicolumn{3}{l}{{\bf A.} $y_2^+=w_2$ (Let $w_2^+$ be an arbitrary out-neighbour of $w_2$.)}\\\hline
	$w_2^+\in\{y_1,y_2,x_2\}$ & $B^++y_1w_1x_2+y_2w_2+w_3^-w_3+x_1^-x_1,B^-+y_1y_2x_1+w_1w_2w_2^++w_3w_3^+$ & $w_3^-,w_3^+\neq x_1;x_1^-\neq y_2$\\\hline
	$w_2^+=x_1$ & $B^++y_1w_1w_2+y_2x_1+w_3^-w_3+x_2^-x_2,B^-+y_1y_2w_2x_2+w_1x_2+w_3w_3^+$ & $w_3^-,w_3^+\neq x_2;x_2^-\neq w_1$\\\hline
	$w_2w_1,w_2w_3\in A$ & $B^++y_1w_1w_2w_3+y_2x_1+x_2^-x_2,B^-+y_1y_2w_2w_1x_2+w_3w_3^+$ & $x_2^-\neq w_1;w_3^+\neq x_2$\\\hline
	\multicolumn{3}{l}{{\bf B.} $y_2^+=w_3$ ($\exists w_1^-\in X\cup W$)}\\\hline
	\multicolumn{3}{l}{Let $B_{q_2}^+=q_2q_1y_1y_2$ and $B_{y_2}^-$ be an in-tree rooted at $y_2$ such that $A(B_{y_2}^-)\subseteq q_1q_2y_2\cup (X,Q)_D$.}\\\hline
	$w_1^-=x_1$ & $B_{q_2}^++y_2x_1w_1w_2+y_2w_3+x_2^-x_2,B_{y_2}^-+y_1w_1x_2+w_iw_i^++w_{5-i}w_{5-i}^+$ & $w_i^+\neq w_{5-i};w_{5-i}^+\neq x_2$\\\hline
	&\multicolumn{2}{l}{note that $x_2^-\neq w_1,w_i$, where $i\in\{2,3\}$}\\\hline
	$w_1^-=x_2$ & $B_{q_2}^++y_2x_2+y_2w_3x_2w_1w_2,B_{y_2}^-+y_1w_1x_2+w_2w_2^++w_3w_3^+$ & $w_2^+\neq w_3,w_3^+\neq x_2$\\\hline
	$w_1^-=w_2$ & $B_{q_2}^++y_2x_1+y_2w_3+w_2^-w_2w_1x_2,B_{y_2}^-+w_3x_2+y_1w_1w_2w_2^+$\\\hline
	&\multicolumn{2}{l}{where $w_2^-\notin \{x_2,w_1\};w_2^+\notin \{y_1,w_1\}$}\\\hline
	\multicolumn{3}{l}{Hence $w_1^-=w_3$. Now $x_1$ and $x_2$ respectively has an in-neighbour $x_1^-$ and an in-neighbour $x_2^-$ which is in $\{w_2,w_3\}$.}\\\hline
	$w_3\in \{x_1^-,x_2^-\}$ & $B^++y_1w_1w_2+y_2w_3+x_1^-x_1+x_2^-x_2,B^-+y_1y_2x_1+w_3w_1x_2+w_2w_2^+$ & $w_2w_2^+\notin \{x_1^-x_1,x_2^-x_2\}$\\\hline
	$w_2x_1,w_2x_2\in A$ & $B_{q_2}^++y_2w_3w_1w_2x_2+y_2x_1,B_{y_2}^-+y_1w_1x_2+w_2x_1+w_3w_3^+$ & $w_3^+\neq w_1$\\\hline
\end{tabular}

\textbf{Therefore, $w_1x_1,w_1x_2\notin A$.} By Lemma~\ref{lem12}, $N^+(w_1)\subset W$, i.e., $w_1w_2,w_1w_3\in A$. This implies that $w_2x_2,w_3x_2\in A$, $N^-(x_1)\cap \{w_2,w_3\}\neq \emptyset$ and $N^+(y_2)\cap \{w_2,w_3\}\neq \emptyset$. W.l.o.g., assume $w_2x_1\in A$.

~\\
\vspace{1mm}
\begin{tabular}{l|c|c}\hline
	Case & Good Pair of $D$ & Notation\\\hline
	$y_2w_2\in A$ & $B^++y_2w_2x_1+y_1w_1w_3x_2,B^-+y_1y_2x_1+w_1w_2x_2+w_3w_3^+$ & $w_3^+\neq x_2$\\\hline
	$y_2w_3\in A$ & $q_2q_1y_1y_2x_1+y_2w_3+w_2^-w_2x_2+w_1^-w_1,q_1q_2y_2+y_1w_1w_2x_1x_1^++w_3x_2x_2^++x_3x_3^+$\\\hline
	&\multicolumn{2}{l}{where $x_ix_i^+\in (X,Q)_D,~i\in[3];w_2^-\neq w_1,x_2;w_1^-\neq y_1$}\\\hline
\end{tabular}

{\bf Case 2:} $|(Y,X)_D|=0$.

Now $|(Y,W)_D|\ge3$ and $|(W,X)_D|\ge4$. This implies that there exists a $w_i\in W$ such that $d_X^+(w_i)\ge2$. By Lemma~\ref{lem12}, assume $y_1w_1,y_2w_2,y_2w_3\in A$. We find a good pair $(B^++P_+,B^-+P_-)$ of $D$ as follows, a contradiction.

~\\
\begin{tabular}{l|c|c}\hline
	Case & $P_+,P_-$ & Notation\\\hline
	\multicolumn{3}{l}{{\bf Subcase 2.1:} $w_1x_1,w_1x_2\in A$. (W.l.o.g., assume $w_2x_1\in A$)}\\\hline
	$w_3w_2\in A$ & $y_1w_1x_2+y_2w_3w_2x_1,w_1x_1+y_1y_2w_2w_2^++w_3w_3^+$ & $w_3^+\neq w_2;w_2^+\in \{w_1,x_2\}$\\\hline
	$w_3w_2\notin A$ & $y_1w_1x_1+y_2w_3+w_2^-w_2+x_2^-x_2,y_1y_2w_2x_1+w_1x_2+w_3w_3^+$ & $w_3^+\neq x_2;x_2^-\neq w_1;w_2^-\neq y_2$\\\hline
	\multicolumn{3}{l}{{\bf Subcase 2.2:} $w_ix_1,w_ix_2\in A,~i\in\{2,3\}$. (W.l.o.g., assume $w_2x_1,w_2x_2\in A$)}\\\hline
	$w_1x_1,w_1w_2\in A$ & $y_1w_1w_2x_1+y_2w_3x_2,y_1y_2w_2x_2+w_1x_1+w_3w_3^+$ & $w_3^+\neq x_2$\\\hline
	$w_1x_1,w_3w_2\in A$ & $y_1w_1x_1+y_2w_3w_2x_2,y_1y_2w_2x_1+w_3x_2+w_1w_1^+$ & $w_1^+\neq x_1$\\\hline
	$w_1x_1,w_1x_2\notin A$ & $y_1w_1w_3x_2+y_2w_2x_1,y_1y_2w_3x_1+w_1w_2x_2$\\\hline
\end{tabular}

\subsubsection{Claim~\ref{lem13}.5}
We find a good pair $(B^++P_+,B^-+P_-)$ of a subdigraph $H$ of $D$ as follows, which implies that $D$ has a good pair by Lemma~\ref{lem1}.

~\\
\begin{tabular}{l|c|c|c}\hline
	Case & $P_+,P_-$ & $H$ & Notation\\\hline
	\multicolumn{4}{l}{{\bf Subcase 1.3:} $(w_1,X)_D=\emptyset$ and $w_2x_1,w_2x_2\in A$. ($P^1=y_1w_1+y_2x_2+w_2x_1,P^2=y_1x_1+y_2w_1+w_2x_2$)}\\\hline
	$w_1w_2\notin A$ & $P^1+w_2^-w_2,P^2+w_1w_1^+$ & $D-w_3$ & $w_1^+\in X;w_2^-\in Y$\\\hline
	$w_1w_2\in A$, $\exists w_2^-\in X\cup Y$ & $P^1+w_2^-w_2,P^2+w_1w_2$ & $D-w_3$\\\hline
	$w_1w_2,w_3w_2\in A$ & $P^1+w_3^-w_3w_2,P^2+w_1w_2+w_3w_3^+$ & $D$ & $w_3^-\in X\cup Y;w_3^+\neq w_2$\\\hline
	\multicolumn{4}{l}{{\bf Subcase 1.4:} $w_2x_1,w_3x_2\in A$.}\\\hline
\end{tabular}\\
\vspace{1mm}
\begin{tabular}{l|c|c|c}\hline
	$\exists w_2^-\in X\cup Y;\exists w_3^+\notin \{x_2,w_2\}$ & $y_1x_1+y_2w_1w_3x_2+w_2^-w_2,y_1w_1w_2x_1+w_3w_3^+$ & $D-w_3$ \\\hline
	$\exists w_3^-\in X\cup Y;\exists w_2^+\notin \{x_1,w_3\}$ & $y_1w_1w_2x_1+w_3^-w_3,y_1x_1+y_2w_1w_3x_2+w_2w_2^+$ & $D-w_3$ \\\hline
\end{tabular}

{\bf Case 2:} $y_2w_1,w_1x_2\in A$ ($y_1w_1,w_1x_1\in A$).

By the digraph duality, it suffices to prove the case of $y_2w_1,w_1x_2\in A$.

{\bf Subcase 2.1:} $y_1w_3,w_3x_1\in A$.

Set $e_1\in (w_1,Y)_D$, $e_2\in (w_3,Y)_D$, $e_3\in (X,w_1)_D$ and $e_4\in (X,w_3)_D$ (if such sets are not empty).

~\\
\vspace{2mm}
\begin{tabular}{l|c|c|c}\hline
	Case & $P_+,P_-$ & $H$ & Notation\\\hline
	$\exists~e_1,e_2$ & $y_1w_3x_1+y_2w_1x_2,y_1x_1+y_2x_2+e_1+e_2$ & $D-w_2$ \\\hline
	$\exists~e_3,e_4$ & $y_1x_1+y_2x_2+e_3+e_4,y_1w_3x_1+y_2w_1x_2$& $D-w_2$ \\\hline
	$\exists~e_1,e_4$ & $y_2w_1x_2+y_1x_1+e_4,y_1w_3x_1+y_2x_2+e_1$ & $D-w_2$ \\\hline
	$\exists~e_2,e_3$ & $y_1w_3x_1+y_2x_2+e_3,y_2w_1x_2+y_1x_1+e_2$ & $D-w_2$ \\\hline
	\multicolumn{4}{l}{First assume $e_1$ exists, then $e_2$ and $e_4$ do not exist.}\\\hline
	$D[\{w_2,w_3\}]=C_2$ & $y_1w_3x_1+y_2w_1x_2+w_2^-w_2,y_1x_1+y_2x_2+w_3w_2w_2^++e_1$ & $D$ & $w_2^-,w_2^+\neq w_3$\\\hline
	$w_1w_3\in A$ & $y_1x_1+y_2w_1x_2+w_1w_3,y_1w_3x_1+y_2x_2+e_1$ & $D-w_2$\\\hline
	$w_3w_1\in A$ & $y_1w_3x_1+y_2w_1x_2,y_1x_1+y_2x_2+w_3w_1+e_1$ & $D-w_2$\\\hline
\end{tabular}

The case when $e_2$ exists can be proved analogouslsy.

Henceforth assume that $e_1$ and $e_2$ do not exist. This implies that both $N^+(w_1)\cap W$ and $N^+(w_3)\cap W$ are not empty.
We first show that $D[\{w_1,w_3\}]\neq C_2$. W.l.o.g., assume $x_1q_1,x_2q_2\in A$. If $q_1y_1,q_2y_2\in A$, then let $B_{y_2}^+=y_2x_2q_2q_1y_1w_3x_1+w_3w_1$ and $B_{w_3}^-=y_1x_1q_1q_2y_2w_1w_3$. If $N_{D-\{w_2,q_2\}}^+(x_2)\neq \emptyset$, say $x_2^+\in N_{D-\{w_2,q_2\}}^+(x_2)$, then $(B_{y_2}^+,B_{w_3}^-+x_2x_2^+)$ is a good pair of $D-w_3$, a contradiction. Hence $x_2w_2\in A$. Then $(B_{y_2}^++w_2^-w_2,B_{w_3}^-+x_2w_2w_2^+)$ is a good pair of $D$, where $w_2^+,w_2^-\neq x_2$ as $\lambda(D)\ge2$, a contradiction.
If $q_2y_1,q_1y_2\in A$, then let $B_{x_2}^+=x_2q_2q_1y_2w_1w_3x_1$ and $B_{x_2}^-=y_2x_2+x_1q_1q_2y_1w_3w_1x_2$. If $N_{D-\{w_2,q_2\}}^-(y_1)\neq \emptyset$, say $y_1^-\in N_{D-\{w_2,q_2\}}^-(y_1)$, then $(B_{x_2}^++y_1^-y_1,B_{x_2}^-)$ is a good pair of $D-w_2$, a contradiction. Hence $w_2y_1\in A$. Then $(B_{x_2}^++w_2^-w_2y_1,B_{x_2}^-+w_2w_2^+)$ is a good pair of $D$, where $w_2^-,w_2^+\neq y_1$, a contradiction.
This implies that $D[\{w_1,w_3\}]\neq C_2$.

If $w_1w_2,w_3w_2\in A$, then at least one of $e_3$ and $e_4$ exists by Fact~\ref{lem13}.3, w.l.o.g., say $e_3$. Let $P_+=y_1w_3x_1+y_2x_2+e_3+w_1w_2$ and $P_-=y_1x_1+y_2w_1x_2+w_3w_2w_2^+$, where $w_2^+\neq w_3$ as $\lambda(D)\ge2$. It follows that $(B^++P_+,B^-+P_-)$ is a good pair of $D$, a contradiction.
If $w_1w_3,w_3w_2\in A$, then let $P_+=y_1w_3x_1+y_2w_1x_2+w_2^-w_2$ and $P_-=y_1x_1+y_2x_2+w_1w_3w_2w_2^+$, where $w_2^-\neq w_3,w_2^+\notin W$ by $\lambda(D)\ge2$ and Fact~\ref{lem13}.3. It follows that $(B^++P_+,B^-+P_-)$ is a good pair of $D$, a contradiction.

The case of $w_3w_1,w_1w_2\in A$ can be proved analogously.

{\bf Subcase 2.2:} $y_1w_2,w_3x_1\in A$.

Set $e_1\in (w_1,Y)_D$, $e_2\in (X,w_1)_D$, $e_3\in (w_3,Y)_D$, $e_4\in (x_2\cup Y,w_3)_D$, $e_5\in (X,w_2)_D$ and $e_6\in (w_2,y_2\cup X)_D$ (if such sets are not empty). If one of the following holds, then $D-w_2$ or $D-w_3$ has a good pair $(B^++P_+,B^-+P_-)$, a contradiction.

~\\
\vspace{2mm}
\begin{tabular}{l|c|c}\hline
	Case & $P_+$ & $P_-$ \\\hline
	$\exists~e_1,e_3,e_4$ & $y_2w_1x_2+e_4+w_3x_1$ & $y_2x_2+y_1x_1+e_3+e_3$\\\hline
	$\exists~e_2,e_3,e_4$ & $y_2x_2+e_2+e_4+w_3x_1$ & $y_2w_1x_2+y_1x_1+e_3$\\\hline
	$\exists~e_1,e_5,e_6$ & $y_2w_1x_2+e_5+y_1x_1$ & $y_2x_2+e_1+y_1w_2+e_6$\\\hline
	$\exists~e_2,e_5,e_6$ & $y_2x_2+e_2+y_1x_1+e_5$ & $y_2w_1x_2+y_1w_2+e_6$\\\hline
\end{tabular}

First assume $N^-(w_3)\subset W$.
If $e_2$ exists, then let $P_+=y_1x_1+y_2x_2+e_2+w_1w_3+w_2^-w_2$ and $P_-=y_1w_2w_3x_1+y_2w_1x_2$, where $w_2^-$ is an arbitrary in-neighbour of $w_2$. It follows that $(B^++P_+,B^-+P_-)$ is a good pair of $D$, a contradiction.
Hence assume that $e_2$ does not exist.
If $w_2w_1\in A$, then $e_3$ exists by Fact~\ref{lem13}.3. Let $P_+=y_1w_2w_1w_3x_1+y_2x_2$ and $P_-=y_1x_1+y_2w_1x_2+w_2w_3+e_3$. It follows that $(B^++P_+,B^-+P_-)$ is a good pair of $D$, a contradiction.

Thus $w_3w_1\in A$. W.l.o.g., assume $q_1y_1,q_2y_2\in A$.
If $x_1q_1,x_2q_2\in A$, then $D$ has a good pair $(B_{w_1}^+,B_{w_1}^-)$ with $B_{w_1}^+=w_1w_3x_1+w_1x_2q_2q_1y_1w_2+y_2^-y_2$ and $B_{w_1}^-=y_1x_1q_1q_2y_2w_1+w_3w_1+x_2x_2^+$, where $y_2^-,x_2^+\neq q_2$ as $\lambda(D)\ge2$, a contradiction.
If $x_1q_2,x_2q_1\in A$, then $D$ has a good pair $(B_{q_1}^+,B_{w_1}^-)$ with $B_{q_1}^+=q_1q_2y_2w_1w_3x_1+w_1x_2+y_1^-y_1+w_2^-w_2$ and $B_{w_1}^-=y_2x_2q_1y_1w_2w_3w_1+x_1q_2q_1$, where $y_1^-\neq q_1$ and $w_2^-\neq y_1$ as $\lambda(D)\ge2$, a contradiction.

By the digraph duality, we also get a contradiction when $N^+(w_2)\subset W$.
Hence $(w_2,X\cup Y)_D,(X\cup Y,w_3)_D\neq \emptyset$.

Next assume that $D[\{x_1,w_3\}]=C_2$ and $N_{D-x_1}^-(w_3)\subset W$. We find $D$ has a good pair $(B^++P_+,B^-+P_-)$ as follows, a contradiciton.

~\\
\vspace{2mm}
\begin{tabular}{l|c|c}\hline
	Case & $P_+,P_-$ & Notation\\\hline
	$D[\{y_1,w_2\}]=C_2;N_{D-y_1}^+(w_2)\subset W$ & $y_1w_2w_1w_3x_1+y_2x_2,w_2y_1x_1+y_2w_1x_2+w_3w_3^+$ & $w_3^+\neq x_1$\\\hline
	$w_2w_3\in A;w_1w_2\notin A$ & $y_1x_1w_3+y_2w_1x_2+w_2^-w_2,y_1w_2w_3x_1+y_2x_2+w_1w_1^+$ & $w_2^-\neq y_1;w_1^+\neq x_2$\\\hline
	$w_2w_3,w_1w_2\in A$ & $y_1w_2w_3x_1+y_2w_1x_2,y_1x_1+y_2x_2+w_1w_2+e_6+w_3w_3^+$ & $w_3^+\neq x_1$\\\hline
	$w_1w_3\in A$ & $y_1x_1w_3+y_2w_1x_2+w_2^-w_2,y_1w_2+e_6+y_2x_2+w_1w_3x_1$ & $w_2^-\neq y_1$\\\hline
\end{tabular}

By the digraph duality, we also get a contradiction when $D[\{y_1,w_2\}]=C_2$ and $N_{D-y_1}^+(w_2)\subset W$.
Thus $e_4$ and $e_6$ exist.

Henceforth assume that $e_1$ and $e_2$ do not exist, namely $N_{D-y_2}^-(w_1)\cup N_{D-x_2}^+(w_1)\subset W$. Then $D$ has a good pair $(B^++P_+,B^-+P_-)$ as follows, a contradiciton.

~\\
\vspace{2mm}
\begin{tabular}{l|c|c|c}\hline
	Case & $P_+,P_-$ & Notation & Analogous case\\\hline
	$w_1w_2\in A$ & $y_1w_2+y_2w_1x_2+e_4+w_3x_1,y_1x_1+y_2x_2+w_1w_2+e_6+w_3w_3^+$ & $w_3^+\neq x_1$ & $w_3w_1\in A$ \\\hline
	$w_2w_1\in A$ & $y_1w_2w_1+y_2x_2+e_4+w_3x_1,y_1x_1+y_2w_1x_2+e_6+w_3w_3^+$ & $w_3^+\neq x_1$ & $w_1w_3\in A$ \\\hline
\end{tabular}

Therefore $e_1$ or $e_2$ exists, w.l.o.g., say $e_1$. Let $P_+=y_1w_2+y_2w_1x_2+e_4+w_3x_1$ and $P_-=y_1x_1+y_2x_2+e_1+e_6+w_3w_3^+$, where $w_3^+\neq x_1$. It follows that $(B^++P_+,B^-+P_-)$ is a good pair of $D$, a contradiction.

{\bf Case 3:} $y_2w_1,w_1x_1\in A$ ($y_1w_1,w_1x_2\in A$).

By the digraph duality, it suffices to prove the case of $y_2w_1,w_1x_1\in A$.

{\bf Subcase 3.1:} $y_1w_3,w_3x_2\in A$.

Set $e_1\in (X,w_1)_D$, $e_2\in (X,w_3)_D$, $e_3\in (w_1,Y)_D$ and $e_4\in (w_3,Y)_D$ (if such sets are not empty).

~\\
\vspace{2mm}
\begin{tabular}{l|c|c|c}\hline
	Case & $P_+,P_-$ & $H$ & Notation\\\hline
	$\exists~e_1,e_2$ & $y_1x_1+y_2x_2+e_1+e_2,y_1w_3x_2+y_2w_1x_1$ & $D-w_2$\\\hline
	$\exists~e_3,e_2$ & $y_1w_3x_2+y_2w_1x_1,y_1x_1+y_2x_2+e_3+e_4$ & $D-w_2$\\\hline
	\multicolumn{4}{l}{\textbf{First assume that $N_{D-y_2}^-(w_1),N_{D-x_1}^+(w_1)\subset W$.}}\\\hline
	$w_1w_3,w_3w_1\in A$ & $y_1w_3w_1x_1+y_2x_2,y_1x_1+y_2w_1w_3x_2$ & $D-w_2$\\\hline
	$w_1w_3\in A$\\$w_3w_1,w_3w_2\notin A$ & $y_1w_3x_2+y_2w_1x_1,y_1x_1+y_2x_2+w_1w_3w_3^+$ & $D-w_2$ & $w_3^+\neq w_1,w_2$\\\hline
	$w_1w_3,w_3w_2\in A$\\$w_3w_1\notin A$ & $y_1w_3x_2+y_2w_1x_1+w_2^-w_2,y_1x_1+y_2x_2+w_1w_3w_2w_2^+$ & $D$ & $w_2^+\notin W;w_2^-\neq w_3$\\\hline
	\multicolumn{4}{l}{The case of $w_3w_1$ can be proved analogously. Thus $w_1$ is not adjacent to $w_3$ and $D[\{w_1,w_2\}]=C_2$.}\\\hline
	$w_3w_2\in A$ & $y_1x_1+y_2x_2+w_3^-w_3w_2w_1,y_1w_3x_2+y_2w_1x_1+w_2w_2^+$ & $D$ & $w_3^-\in X\cup y_2;w_2^+\neq w_1$\\\hline
	\multicolumn{4}{l}{The case of $w_2w_3\in A$ can be proved analogously.}\\\hline
	$N(w_2)\cap (X\cup Y)\neq \emptyset$ & $y_1w_3x_2+y_2w_1x_1+w_2^-w_2,y_1x_1+y_2x_2+w_1w_2w_2^++w_3w_3^+$ & $D$ & $w_3^+\neq x_2$\\\hline
	& \multicolumn{3}{l}{Note that $w_2^+\in N^+(w_2)\cap (X\cup Y)$ and $w_2^- \in N^-(w_2)\cap (X\cup Y)$.}\\\hline
	\multicolumn{4}{l}{The case when $N_{D-y_1}^-(w_3),N_{D-x_2}^+(w_3)\subset W$ can be proved analogously.}\\\hline
	\multicolumn{4}{l}{\textbf{Next assume that $e_1$ and $e_4$ exist but $e_2$ and $e_3$ do not.}}\\\hline
	$w_1w_3\in A$ & $y_1x_1+y_2x_2+e_1+w_1w_3,y_1w_3x_2+y_2w_1x_1$ & $D-w_2$\\\hline
	\multicolumn{4}{l}{The case of $w_3w_1\in A$ can be proved analogously.}\\\hline
	$w_1w_2,w_2w_3\in A$ & $y_1x_1+y_2x_2+e_1+w_1w_2w_3,y_1w_3x_2+y_2w_1x_1+w_2w_2^+$ & $D$ & $w_2^+\neq w_3$\\\hline
	\multicolumn{4}{l}{The case when $e_2$ and $e_3$ exist but $e_1$ and $e_4$ do not can be proved analogously.}\\\hline
\end{tabular}

{\bf Subcase 3.2:} $y_1w_2,w_3x_2\in A$.

Set $e_1\in (w_1,Y)_D$, $e_2\in (X,w_1)_D$, $e_3\in (w_3,Y)_D$, $e_4\in (x_1\cup Y,w_3)_D$, $e_5\in (X,w_2)$ and $e_6\in (w_2,X\cup y_2)$ (if such sets are not empty). If one of the following holds, then $D-w_2$ or $D-w_3$ has a good pair $(B^++P_+,B^-+P_-)$, a contradiction.

~\\
\vspace{2mm}
\begin{tabular}{l|c|c}\hline
	Case & $P_+$ & $P_-$ \\\hline
	$\exists~e_1,e_3,e_4$ & $y_2w_1x_1+e_4+w_3x_2$ & $y_1x_1+y_2x_2+e_1+e_3$\\\hline
	$\exists~e_2,e_5,e_6$ & $y_1x_1+y_2x_2+e_2+e_5$ & $y_1w_2+e_6+y_2w_1x_1$\\\hline
\end{tabular}

First assume $N^-(w_3)\subset W$. Then $D$ has a good pair $(B^++P_+,B^-+P_-)$ as follows, a contradiction.

~\\
\vspace{2mm}
\begin{tabular}{l|c|c}\hline
	Case & $P_+,P_-$ & Notation\\\hline
	$w_2w_1\in A$ & $y_1w_2w_3x_2+y_2w_1x_1,y_1x_1+y_2x_2+w_2w_1w_3w_3^+$ & $w_3^+\notin W$\\\hline
	otherwise & $y_1w_2w_3+w_1^-w_1x_1+y_2x_2,y_1x_1+y_2w_1w_3x_2+w_2w_2^+$ & $w_1^-\neq y_2,x_1;w_2^+\neq w_3$\\\hline
\end{tabular}

The case of $N^+(w_2)\subset W$ can be proved analogously.
This implies that $N^-(w_3)\cap (X\cup Y)\neq \emptyset$ and $N^+(w_2)\cap (X\cup Y)\neq \emptyset$.

Next assume that $D[\{x_2,w_3\}]=C_2$ and $N_{D-x_2}^-(w_3)\subset W$.
If $w_2w_3\in A$, then $D$ has a good pair $(B^++P_+,B^-+P_-)$ below.

~\\
\vspace{2mm}
\begin{tabular}{l|c|c|c}\hline
	Case & $P_+$ & $P_-$ & Notation \\\hline
	$D[\{w_1,w_2\}]\neq C_2$& $y_1x_1+y_2x_2w_3+w_2^-w_2+w_1^-w_1$ &$y_1w_2w_3x_2+y_2w_1x_1$& $w_1^-\neq y_2,w_2^-\neq y_1$\\\hline
	$D[\{w_1,w_2\}]=C_2$& $y_1w_2w_1x_1+y_2x_2w_3$& $y_1x_1+y_2w_1w_2w_3x_2$\\\hline
\end{tabular}

Thus $w_1w_3\in A$ and $w_2w_3\notin A$. We find a good pair $(B^++P_+,B^-+P_-)$ of a subdigraph $H$ of $D$ as follows, which implies that $D$ has a good pair by Lemma~\ref{lem1}.

~\\
\vspace{2mm}
\begin{tabular}{l|c|c|c}\hline
	Case & $P_+,P_-$ & $H$ & Notation\\\hline
	$x_1w_1\notin A$ & $w_1^-w_1x_1+y_2x_2w_3,y_1x_1+y_2w_1w_3x_2$ & $D-w_2$ & $w_1^-\neq y_2$\\\hline
	$D[\{x_1,w_1\}]=C_2$ & $y_1x_1w_1+y_2x_2w_3+w_2^-w_2,y_1w_2w_2^++y_2w_1w_3x_2$ & $D$ & $w_2^-,w_2^+\neq y_1$\\\hline
\end{tabular}

The case when $D[\{y_1,w_2\}]=C_2$ and $N_{D-y_1}^+(w_2)\subset W$ can be proved analogously.
Henceforth $e_4$ and $e_6$ exist. We find that $D$ has a good pair $(B^++P_+,B^-+P_-)$ below.

~\\
\begin{tabular}{l|c|c}\hline
	Case & $P_+,P_-$ & Notation\\\hline
	$D[\{w_1,w_2\}]\neq C_2$ & $y_1x_1+y_2x_2+e_4+w_2^-w_2+w_1^-w_1,y_1w_2+e_6+y_2w_1x_1+w_3x_2$ & $w_1^-\neq y_2;w_2^-\neq y_1$\\\hline
	$D[\{w_1,w_2\}]= C_2$ & $y_1x_1+y_2w_1w_2+e_4+w_3x_2,y_1w_2w_1x_1+y_2x_2+w_3w_3^+$ & $w_3^+\neq x_2$\\\hline
\end{tabular}

\subsubsection{Claim~\ref{lem13}.6}
{\bf Case 2:} $d_Y^-(w_2)=2$.

That is $y_1w_2,y_2w_2\in A$. Let $P'_+=y_1w_2x_2+y_2x_1$ and $P'_-=y_1x_1+y_2w_2$. By Lemma~\ref{lem12} and $\lambda(D)\ge2$, $w_2$ has an out-neighbour $w_2^+\notin Y\cup x_2$. We find a good pair $(B^++P_+,B^-+P_-)$ of a subdigraph $H$ of $D$, which implies that $D$ has a good pair by Lemma~\ref{lem1}.

~\\
\vspace{2mm}
\begin{tabular}{l|c|c|c}\hline
	Case & $P_+,P_-$ & $H$ & Notation\\\hline
	$w_2^+=x_1$ & $P'_+,P'_-+w_2x_1$ & $D-\{w_1,w_3\}$\\\hline
	$w_2^+=w_1;\exists w_1^-\notin \{w_2,w_3\}$ & $P'_++w_1^-w_1,P'_-+w_2w_1x_2$ & $D-w_3$\\\hline
	$w_2^+=w_1;w_3w_1\in A$ & $P'_++w_3^-w_3w_1,P'_-+w_2w_1x_2+w_3w_3^+$ & $D$ & $w_3^+,w_3^-\neq w_1$\\\hline
	$w_2^+=w_3;w_3^+=w_1$ & $P'_++w_1^-w_1+w_3^-w_3,P'_-+w_2w_3w_1x_2$ & $D$ & $w_1^-\neq w_3;w_3^-\neq w_2$\\\hline
	$w_3^+\in X\cup \{y_1\};\exists w_3^-\neq w_1,w_2$ & $P'_++w_3^-w_3,P'_-+w_2w_3w_3^+$ & $D-w_1$\\\hline
	$w_3^+\in X\cup \{y_1\};w_1w_3\in A$ & $P'_++w_1^-w_1w_3,P'_-+w_2w_3w_3^++w_1x_2$ & $D$ & $w_1^-\neq w_3$\\\hline
\end{tabular}

{\bf Case 3:} $d_Y^-(w_2)=1$.

W.l.o.g., assume $y_1w_2\in A$, then $y_2w_3\in A$. Since $D\nsupseteqq E_3$, $w_1$ has an in-neighbour $w_1^-\notin \{x_2,w_2\}$. Set $P'_+=y_1x_1$ and $P'_-=y_1w_2x_2+y_2x_1$.

{\bf Subcase 3.1:} $x_2w_1\in A$.

Since $\lambda(D)\ge2$ and $D\nsupseteqq E_3$, $w_1$ and $w_2$ respectively have an in-neighbour $w_1^-$ and $w_2^-$ such that $w_1^-\notin \{x_2,w_2\}$ and $w_2^-\notin \{y_1,x_2,w_1\}$. We find a good pair $(B^++P_+,B^-+P_-)$ of a subdigraph $H$ of $D$ as follows, which implies that $D$ has a good pair by Lemma~\ref{lem1}.

~\\
\vspace{2mm}
\begin{tabular}{l|c|c|c}\hline
	Case & $P_+,P_-$ & $H$ & Notation\\\hline
	$w_1^-,w_2^-\neq w_3;\exists w_1^+\notin \{x_2,w_3\}$ & $P'_++w_1^-w_1x_2+w_2^-w_2,P'_-+w_1w_1^+$ & $D-w_3$\\\hline
	$w_1^-,w_2^-\neq w_3;w_1w_3\in A$ & $P'_++w_1^-w_1x_2+w_2^-w_2+y_2w_3,P'_-+w_1w_3w_3^+$ & $D$ & $w_3^+\neq w_1$\\\hline
	$w_1^-=w_3;w_2^-\neq w_3$ & $P'_++y_2w_3w_1x_2+w_2^-w_2,P'_-+w_1w_1^++w_3w_3^+$ & $D$ & $w_1^+\notin \{x_2,w_3\};w_3^+\neq w_1$\\\hline\
	$w_1^-\neq w_3;w_2^-=w_3$ & $P'_++y_2w_3w_2+w_1^-w_1x_2,P'_-+w_3w_3^++w_1w_1^+$ & $D$ & $w_1^+\neq x_2;w_3^+\neq w_2$\\\hline
	$w_1^-=w_2^-=w_3;\exists w_2^+\notin \{x_2,y_1\}$ & $y_1x_1+y_2w_3w_2x_2w_1,y_1w_2w_2^++y_2x_1+w_3w_1x_2$ & $D$ & $w_2^+\neq y_1$\\\hline
	$w_1^-=w_2^-=w_3;w_2y_1\in A$ & $y_1w_2x_2w_1+y_2x_1+w_3^-w_3,y_2w_3w_2y_1x_1+w_1x_2$ & $D$ & $w_3^-\neq y_2$\\\hline
\end{tabular}

{\bf Subcase 3.2:} $x_1w_1\in A$.

~\\
\vspace{2mm}
\begin{tabular}{l|c|c|c}\hline
	$w_1w_2,w_2w_3\in A$ & $y_1x_1+w_1w_2x_2+y_2w_3,y_1w_2w_3w_3^++y_2x_1+w_1x_2$ & $D$ & $w_3^+\notin \{y_1,w_2\}$\\\hline
	$w_1w_2\in A;w_2w_3\notin A$ & $y_1x_1+w_1w_2x_2,y_1w_2w_2^++y_2x_1+w_1x_2$ & $D-w_3$ & $w_2^+\neq x_2$\\\hline
	$w_1w_3,w_3w_2,w_3y_2\in A$ & $y_1x_1w_1x_2+y_2w_3w_2,y_1w_2x_2+w_1w_3y_2x_1$ & $D$\\\hline
	$w_1w_3,w_3w_2\in A;w_3y_2\notin A$ & $y_2x_1w_1w_3w_2x_2,y_1x_1+y_2w_3w_3^++w_2w_2^++w_1x_2$ & $D$ & $w_3^+\neq w_2;w_2^+\neq x_2$\\\hline
	$w_1w_3,w_3w_2\notin A$ & $y_1x_1w_1x_2+w_2^-w_2,y_1w_2x_2+y_2x_1+w_1w_1^+$ & $D-w_3$ & $w_2^-\neq y_1;w_1^+\neq x_2$\\\hline
	$w_1w_3\in A$ & $y_1x_1w_1x_2+y_2w_3+w_2^-w_2,y_1w_2x_2+y_2x_1+w_1w_3w_3^+$ & $D$ & $w_2^-\neq y_1;w_3^+\neq w_1$\\\hline
	$w_3w_2\in A$ & $y_1x_1w_1x_2+y_2w_3w_2,y_1w_2x_2+y_2x_1+w_1w_1^++w_3w_3^+$ & $D$ & $w_1^+\neq x_2;w_3^+\neq w_2$\\\hline
\end{tabular}

{\bf Subcase 3.3:} $N^-(w_1)\subset W$.

~\\
\vspace{2mm}
\begin{tabular}{c|c|c}\hline
	$P_+,P_-$ & $H$ & Notation\\\hline
	$y_1x_1+y_2w_3w_1x_2+w_2^-w_2,y_1w_2x_2+y_2x_1+w_3w_3^++w_1w_1^+$ & $D$ & $w_2^-\neq y_1;w_3^+\neq w_1;w_1^+\neq x_2$\\\hline
\end{tabular}

{\bf Case 4:} $(N^-(w_1)\cup N^-(w_2))\cap Y=\emptyset$.

That is $y_1w_3,y_2w_3\in A$. By Lemma~\ref{lem12}, $N^-(w_i)\cap W\neq \emptyset$ for any $i\in[2]$ and $d_Y^+(w_3)=0$.

{\bf Subcase 4.1:} $w_3x_i\in A,~i\in[2]$.

Since $D\nsupseteqq E_3$, $w_2$ has an in-neighbour $w_2^-\notin \{w_1,x_2\}$. Set $P_+=y_1x_1+y_2w_3+w_2^-w_2x_2$ and $P_-=y_1w_3x_i+y_2x_1+w_2w_2^+$, where $w_2^+\neq x_2$ as $\lambda(D)\ge2$.
If $w_2^+=w_1$, then $(B^++P_++w_1^-w_1,B^-+P_-+w_1x_2)$ is a good pair of $D$, a contradiction.
If $w_2^+\neq w_1$, then $(B^++P_+,N^-+P_-)$ is a good pair of $D-w_1$, a contradiction.

{\bf Subcase 4.2:} $N^+(w_3)\subset W$.

That is $w_3w_2,w_3w_1\in A$. Since $D\nsupseteqq E_3$, $D[\{w_1,w_2\}]\neq C_2$. Now assume that $w_iw_{3-i}\notin A$ for some $i\in[2]$. It follows that $(B^++y_1w_3w_ix_2+y_2x_1+w_{3-i}^-w_{3-i},B^-+y_1x_1+y_2w_3w_{3-i}x_2+w_iw_i^+)$ is a good pair of $D$, where $w_{3-i}^-\neq w_3$ and $w_i^+\neq x_2$ as $\lambda(D)\ge2$, a contradiction.

\subsubsection{Claim~\ref{lem13}.7}
We find a good pair $(B^++P_+,B^-+P_-)$ of $D$ or $D-w_3$ as follows, a contradiction.

~\\
\begin{tabular}{l|c|c}\hline
	Case & $P_+,P_-$ & Notation \\\hline
	\multicolumn{3}{l}{{\bf B.} $w_2x_1,w_2x_2\in A$.}\\\hline
	$N^+(w_1)\cap \{w_2,x_1,y_1\}=\emptyset$ & $y_1w_1x_2+y_2w_2x_1,y_1x_1+y_2w_1w_1^++w_2x_2$\\\hline
	$w_1w_3\in A$ & $y_1w_1x_2+y_2w_2x_1+w_3^-w_3,y_1x_1+y_2w_1w_3w_3^++w_2x_2$ & $w_3^-\neq w_1;w_3^+\neq w_1,y_2$\\\hline
	$w_1y_2\in A,w_3w_2\notin A$ & $y_1x_1+y_2w_1x_2+w_2^-w_2,y_1w_1y_2w_2x_2$& $w_2^-\neq y_2$\\\hline
	$w_1y_2,w_3w_2\in A$& $y_1x_1+y_2w_1x_2+w_3^-w_3w_2,y_1w_1y_2w_2x_2+w_3w_3^+$& $w_3^-,w_3^+\neq w_2$\\\hline
	\multicolumn{3}{l}{{\bf C.} $w_2x_2,w_1x_1\in A$.}\\\hline
	$w_2w_3\notin A$ & $y_1w_1x_1+y_2w_2x_2,y_1x_1+y_2w_1x_2+w_2w_2^+$& $w_2^+\neq x_2$\\\hline
	$w_3w_2\notin A$& $y_1x_1+y_2w_1x_2+w_2^-w_2,y_1w_1x_1+y_2w_2x_2$& $w_2^-\neq y_2$\\\hline
\end{tabular}\\
\vspace{2mm}
\begin{tabular}{l|c|c}\hline
	$D[\{w_2,w_3\}]=C_2$& $y_1w_1x_1+y_2w_2x_2+w_3^-w_3,y_1x_1+y_2w_1x_2+w_2w_3w_3^+$& $w_3^-,w_3^+\neq w_2$\\\hline
	\multicolumn{3}{l}{{\bf D.} $w_1x_1,w_1x_2\in A$.}\\\hline
	$D[\{x_2,w_3\}]=C_2$\\$D[\{y_2,w_2\}]=C_2$\\$w_2w_3\in A$ & $y_2w_1x_2w_3x_1+w_2^-w_2,y_1x_1+y_2w_2w_3x_2+w_1w_1^+$ & $w_2^-\notin \{y_2,w_1\};w_1^+\notin \{w_2,x_2\}$\\\hline
	otherwise & $y_1x_1+y_2w_1+w_2^-w_2+w_3^-w_3,y_1w_1x_2+y_2w_2w_2^++w_3x_1$ & $w_3^-\neq x_2;w_2^-,w_2^+\neq y_2$\\\hline
	\multicolumn{3}{l}{{\bf E.} $w_2x_1,w_3x_2\in A$.}\\\hline
	$w_2w_3\notin A$ & $y_1w_1+y_2w_2x_1+w_3^-x_2,y_1x_1+y_2w_1x_2+w_2w_2^++w_2w_3^+$& $w_3^-,w_3^+\neq x_2;w_2^+\neq x_1$\\\hline
	$w_2w_3,x_2w_3\in A$& $y_1x_1+y_2w_1+w_2^-w_2w_3x_2,y_1w_1x_2+y_2w_2x_1+w_3w_3^+$& $w_2^-\notin \{w_2,x_2,y_2\};w_3^+\neq x_2$\\\hline
	$w_2w_3\in A$\\$x_2w_3,w_3w_2\notin A$ & $y_1w_1+y_2w_2x_1+w_3^-w_3x_2,y_1x_1+y_2w_1x_2+w_2w_3w_3^+$& $w_3^-\neq w_2;w_3^+\neq x_2$\\\hline
	$D[\{w_2,w_3\}]=C_2$\\$x_2w_3,w_1w_3\notin A$ & $y_2w_1x_2+w_3^-w_3w_2x_1,y_1x_1+y_2w_2w_3x_2+w_1w_1^+$ & $w_3^-\notin \{w_2,x_1\};w_1^+\neq x_2$\\\hline
\end{tabular}

This implies that $D[\{w_2,w_3\}]=C_2$, $x_2w_3\notin A$ and $w_1w_3\in A$.
Now let $Q'=\{w_2,w_3\}$, $X'=N^-(Q')$ and $Y'=N^+(Q')$. Then $X'=\{y_2,w_1\}$ and $Y'=\{x_1,x_2\}$. Since $y_2w_1\in A$, by Claim~\ref{lem13}.4, $D$ has a good pair, a contradiction.

~\\
\vspace{2mm}
\begin{tabular}{l|c|c}\hline
	\multicolumn{3}{l}{{\bf F.} $w_1x_1,w_3x_2\in A$.}\\\hline
	$w_2w_3\notin A$ & $y_1w_1x_1+y_2w_2+w_3^-w_3x_2,y_1x_1+y_2w_1x_2+w_2w_2^++w_3w_3^+$ & $w_3^-,w_3^+\neq x_2;\forall w_2^+\in N_D^+(w_2)$\\\hline
	otherwise & $y_1w_1x_1+y_2w_2w_3x_2,y_1x_1+y_2w_1x_2+w_2w_2^++w_3w_3^+$ & $w_2^+\neq w_3;w_3^+\neq x_2$\\\hline
\end{tabular}

{\bf Subcase 1.2:} $w_2x_2,w_3x_2\in A$.

We find a good pair $(B^++P_+,B^-+P_-)$ of $D$ as follows, a contradiction.

~\\
\vspace{2mm}
\begin{tabular}{l|c|c}\hline
	Case & $P_+,P_-$ & Notation \\\hline
	$w_1x_1\in A;w_3w_2\notin A$ & $y_1x_1+y_2w_1+w_2^-w_2+w_3^-w_3x_2,y_1w_1x_1+y_2w_2x_2+w_3w_3^+$ & $w_2^-\neq y_2;w_3^-,w_3^+\neq x_2$\\\hline
	$w_1x_1,w_3w_2\in A,w_2y_2\notin A$ & $y_1x_1+y_2w_1+w_3^-w_3w_2x_2,y_1w_1x_1+y_2w_2w_2^++w_3x_2$ & $x_3^-\notin \{w_2,x_2\};w_2^+\neq x_2$\\\hline
	$w_1x_1,w_3w_2,w_2y_2\in A$ & $y_1w_1x_1+y_2w_2x_2+w_3^-w_3,y_1x_1+w_2y_2w_1w_1^++w_3x_2$\\\hline
	&\multicolumn{2}{l}{where $w_1^+\notin \{w_2,y_2\};\forall w_3^-\in N_D^-(w_3)$}\\\hline
	\multicolumn{3}{l}{$w_2x_1\in A$}\\\hline
	$D[\{w_3,x_2\}]=C_2,$\\$D[\{w_1,y_2\}]=C_2$ & $y_1x_1+y_2w_1+w_2^-w_2x_2w_3,y_1w_1y_2w_2x_1+w_3x_2$ & $w_2^-\notin \{w_3,x_2\}$\\\hline
	otherwise & $y_1w_1+y_2w_2x_1+w_3^-w_3x_2,y_1x_1+y_2w_1w_1^++w_2x_2+w_3w_3^+$ & $w_3^-,w_3^+\neq x_2;w_1^+\neq y_2$\\\hline
	\multicolumn{3}{l}{$w_3x_1\in A$}\\\hline
	$D[\{w_3,x_1\}]=C_2,w_1w_2\notin A$ & $y_1x_1w_3x_2+y_2w_1+w_2^-w_2,y_1w_1w_1^++y_2w_2x_2+w_3x_1$ & $w_2^-\neq y_2;w_1^+\neq y_1$\\\hline
	$D[\{w_3,x_1\}]=C_2,w_2w_3\notin A$ & $y_1w_1w_2x_2+w_3^-w_3x_1,y_2w_1y_1x_1+w_2w_2^++w_3x_2$ & $w_3^-\neq x_1;w_2^+\neq x_2$\\\hline
\end{tabular}

Hence $w_1w_2,w_2w_3\in A$.
Now w.l.o.g., assume $x_1q_1,x_2q_2\in A$.
If $q_1y_1,q_2y_2\in A$, then $D$ has a good pair $(B_{q_2}^+,B_{w_3}^-)$ with $B_{q_2}^+=q_2q_1y_1w_1w_2w_3x_2+w_3x_1+y_2^-y_2$ and $B_{w_3}^-=w_2x_2q_2y_2w_1y_1x_1w_3+q_1q_2$, where $y_2^-\neq q_2$, a contradiction.
If $q_2y_1,q_1y_2\in A$, then $D$ has a good pair $(B_{y_1}^+,B_{w_3}^-)$ with $B_{y_1}^+=y_1w_1w_2w_3x_2+w_3x_1q_1q_2+y_2^-y_2$ and $B_{w_3}^-=w_2x_2q_2q_1y_2w_1y_1x_1w_3$, where $y_2^-\neq q_1$, a contradiction.
Hence $D[\{w_3,x_1\}]\neq C_2$. Analogously, $D[\{w_1,y_1\}]\neq C_2$.

~\\
\vspace{2mm}
\begin{tabular}{l|c|c}\hline
	Case & $P_+,P_-$ & Notation \\\hline
	$D[\{w_1,y_2\}]=C_2$ & $y_1x_1+y_2w_1+w_2^-w_2+w_3^-w_3x_2,y_1w_1y_2w_2x_2+w_3x_1$ & $w_3^-\neq x_2;w_2^-\neq y_2$\\\hline
	\multicolumn{3}{l}{The case of $D[\{w_3,x_2\}]=C_2$ can be proved analogously.}\\\hline
	$N^-(w_3)\cup N^+(w_1)\subseteq W$ & $y_1w_1w_3x_1+y_2w_2x_2,y_1x_1+y_2w_1w_2w_3x_2$\\\hline
	Otherwise & $y_1w_1+y_2w_2x_2+w_3^-w_3x_1,y_1x_1+y_2w_1w_1^++w_2w_2^++w_3x_2$ & $w_3^-,w_1^+\notin W;w_2^+\neq x_2$\\\hline
\end{tabular}

{\bf Case 2:} $|N_W^+(Y)|=3$ ($|N_W^-(X)|=3$).

By the digraph duality, it suffices to prove the case of $|N_W^+(Y)|=3$.

W.l.o.g., assume $y_1w_3,y_2w_1,y_2w_2\in A$. We find a good pair $(B^++P_+,B^-+P_-)$ of $D$ as follows, a contradiction.

{\bf Subcase 2.1:} $w_1x_2,w_2x_2,w_3x_1\in A$.

Since $D\nsupseteqq E_3$, $D[\{w_1,w_2\}]\neq C_2$.

~\\
\vspace{2mm}
\begin{tabular}{l|c|c}\hline
	Case & $P_+,P_-$ & Notation \\\hline
	$w_1w_2,w_3w_1\in A$ & $y_1w_3x_1+y_2w_1w_2x_2,y_1x_1+y_2w_2w_2^++w_3w_1x_2$ & $w_2^+\notin \{x_2,y_2\}$\\\hline
	$w_1w_2\in A;w_3w_1\notin A$ & $y_1w_3x_1+y_2w_2+w_1^-w_1x_2,y_1x_1+y_2w_1w_2x_2+w_3w_3^+$ & $w_1^-\notin \{x_2,y_2\};w_3^+\neq x_1$\\\hline
	\multicolumn{3}{l}{The case of $w_2w_1\in A$ can be proved analogously.}\\\hline
	$w_2w_3\in A;w_3w_2\notin A$ & $y_1x_1+y_2w_1x_2+w_2^-w_2w_3,y_1w_3x_1+y_2w_2x_2+w_3w_3^+$ & $w_2^-\neq y_2;w_3^+\neq x_2$\\\hline
\end{tabular}

Hence $w_3w_2\in A$.
If $D[\{w_1,y_2\}]=C_2$, then let $V(Q')=\{w_1,y_2\}$, $X'=N^-(Q')$ and $Y'=N^+(Q')$. Note that $Y'=\{w_2,x_2\}$. Since $w_2x_2\in A$, by Claim~\ref{lem13}.4, $D$ has a good pair, a contradiction.
The case of $D[\{w_1,x_2\}]= C_2$ is analogous. Just let $V(Q')=\{w_1,x_2\}$. Then $X'=\{w_2,y_2\}$ with $y_2w_2\in A$, which implies that $D$ has a good pair by Claim~\ref{lem13}.4, a contradiction.

~\\
\vspace{2mm}
\begin{tabular}{l|c|c}\hline
	Case & $P_+,P_-$ & Notation \\\hline
	$D[\{w_1,y_2\}]\neq C_2$\\$D[\{w_1,x_2\}]\neq C_2$ & $y_2w_2w_3x_1+w_1^-w_1x_2,y_1w_3w_2x_2+y_2w_1w_1^+$ & $w_1^-,w_1^+\notin \{x_2,y_2\}$\\\hline
	\multicolumn{3}{l}{The following cases can be proved analogously: $w_1w_3\in A$; $w_3w_1\in A$; $w_3w_2\in A$.}\\\hline
	$w_1,w_2\notin N_D(w_3)$ & $y_1w_3x_1+y_2w_2x_2+w_1^-w_1,y_1x_1+y_2w_1x_2+w_2w_2^++w_3w_3^+$ & $w_1^-\neq y_2;w_2^+\neq x_2;w_3^+\neq x_1$\\\hline
\end{tabular}

{\bf Subcase 2.2:} $w_1x_1,w_2x_2,w_3x_2\in A$.

~\\
\vspace{2mm}
\begin{tabular}{l|c|c}\hline
	Case & $P_+,P_-$ & Notation \\\hline
	$w_1w_2\in A;w_2w_3\notin A$ & $y_1x_1+y_2w_1w_2x_2+w_3^-w_3,y_1w_3x_2+y_2w_2w_2^++w_1x_1$ & $w_3^-\neq y_1;w_2^+\notin \{x_2,y_2\}$\\\hline
	$w_1w_2\in A;D[\{w_1,x_1\}]\neq C_2$ & $y_1w_3+y_2w_2x_2+w_1^-w_1x_1,y_1x_1+y_2w_1w_2w_3x_2$ & $w_1^-\neq y_2$\\\hline
	$w_1w_2\in A;D[\{w_3,y_1\}]\neq C_2$ & $y_1x_1+y_2w_1w_2w_3x_2,y_1w_3w_3^++y_2w_2x_2+w_1x_1$ & $w_3^+\neq x_2$\\\hline
\end{tabular}

Hence both $D[\{w_1,x_1\}]$ and $D[\{w_3,y_1\}]$ are $C_2$.
W.l.o.g., assume $x_1q_2,x_2q_2\in A$.
If $q_1y_1,q_2y_2\in A$, then $D$ has a good pair $(B_{w_1}^+,B_{w_3}^-)$ with $B_{w_1}^+=w_1x_1q_1q_2y_2w_2w_3y_1$ and $B_{w_3}^-=y_2w_1w_2x_2q_2q_1y_1w_3+x_1y_1$, a contradiction.
If $q_2y_1,q_1y_2\in A$, then $D$ has a good pair $(B_{w_1}^+,B_{x_2}^-)$ with $B_{w_1}^+=w_1x_1q_1y_2w_2w_3x_2q_2+y_1^-y_1$ and $B_{x_2}^-=w_3y_1x_1w_1w_2x_2+y_2w_1+q_1q_2y_1$, where $y_1^-\neq q_2$, a contradiction.

By the digraph duality, we also prove the case of $w_2w_3\in A$. Thus $w_1w_2,w_2w_3\notin A$.

~\\
\begin{tabular}{l|c|c|c}\hline
	Case & $P_+$ & $P_-$ & Notation \\\hline
	$D[\{w_1,w_3\}]=C_2$ & $y_1w_3w_1x_1+y_2w_2x_2$ & $y_1x_1+y_2w_1w_3x_2+w_2w_2^+$ & $w_2^+\neq x_2$\\\hline
	$D[\{w_1,w_3\}]\neq C_2$ & $y_1w_3x_2+y_2w_1x_1+w_2^-w_2$ & $y_1x_1+y_2w_2x_2+w_1w_1^++w_3w_3^+$ & $w_2^-\neq y_2;w_1^+\neq x_1;w_3^+\neq x_2$\\\hline
	\multicolumn{3}{l}{The case when $w_1x_2,w_3x_2,w_2x_1\in A$ can be proved analogously.}\\\hline
\end{tabular}

\subsubsection{Case 2}
We find a good pair $(B^++P_+,B^-+P_-)$ of $D$, a contradiction.

{\bf Subcase 2.2:} Exactly one of $w_2$ and $w_3$ has two out-neighbours in $X$.

W.l.o.g., assume $w_3x_1,w_3x_2\in A$. Then $d_X^+(w_1)\le1$ and $d_X^+(w_1)\ge1$ as $|(W,X)_D|=4$, w.l.o.g., say $w_1x_1\in A$. Note that $w_3$ has an in-neighbour $w_3^-\in \{w_1,w_2\}$ by Lemma~\ref{lem12} and $d_Y^-(w_3)=1$.

~\\
\vspace{2mm}
\begin{tabular}{l|c|c}\hline
	Case & $P_+,P_-$ & Notation \\\hline
	$w_3^-=w_1$ & $y_1w_1x_1+y_2w_3x_2+w_2^-w_2,y_2w_1w_3x_1+y_1w_2w_2^+$ & $w_2^-,w_2^+\neq y_1$\\\hline
	$w_3^-=w_2$ & $y_1w_2w_3x_2+y_2w_1x_1,y_2w_3x_1+y_1w_1w_1^++w_2w_2^+$ & $w_1^+\notin \{x_1,y_1\};w_2^+\neq w_3$\\\hline
\end{tabular}

{\bf Subcase 2.3:} $d_X^+(w_2)=d_X^+(w_3)=1$.

That is $w_1x_1,w_1x_2\in A$. W.l.o.g., assume $w_2x_1,w_3x_2\in A$.

First assume $w_3w_2\in A$. Let $w_2^+$ be an out-neighbour of $w_2$ such that $w_2^+\neq x_1$ as $\lambda(D)\ge2$. If $w_2^+=y_1$, then $D[\{w_2,y_1\}]=C_2$. Set $Q'=D[\{w_2,y_1\}]$ and $Y'=N^+(Q')$. Now $Y'=\{w_1,x_1\}$ with $w_1x_1\in A$, which implies that $D$ has a good pair by the cases with $Y$ is not an independent set.

~\\
\begin{tabular}{l|c|c}\hline
	Case & $P_+,P_-$ & Notation \\\hline
	$w_2^+\neq y_1$ & $y_1w_1x_2+y_2w_3w_2x_1,y_2w_1x_1+w_3x_2+y_1w_2w_2^+$\\\hline
	$w_3w_2\notin A$ & $y_1w_1x_1+y_2w_3x_2+w_2^-w_2,y_1w_2x_1+y_2w_1x_2+w_3w_3^+$ & $w_2^-\neq y_1;w_3^+\neq x_2$\\\hline
\end{tabular}

\subsection{Proposition~\ref{h3-1}}
By contradiction, suppose that $P_8$ is the longest dipath of $D$ by Proposition~\ref{h2-2}. In fact there exist arcs between $C^1$ and $C^2$ from both directions, otherwise $D$ has a $P_9$ as $\lambda(D)\ge2$. W.l.o.g., assume $|C^1|\ge |C^2|$. Let $y$ be the vertex in $V-V(C^1\cup C^2)$ and $x_i,~\in[8]$ be the vertex in $V(C^1\cup C^2)$.
From the longestness of $P_8$ in $D$, we have the following facts.
\begin{description}
	\item[Fact~\ref{h3-1}.1] At least one of $(C^{i},y)_D$ and $(y,C^{3-i})_D$ is empty for any $i\in[2]$.
	\item[Fact~\ref{h3-1}.2] At least one of arcs $x_iy$ and $yx_{i+1}$ is not in $A$ for any $i\in[7]$.
\end{description}
We distinguish two cases as follows.

{\bf Case 1:} $|C^1|=|C^2|=4$.

Let $C^1=x_1x_2x_3x_4x_1$ and $C^2=x_5x_6x_7x_8x_5$ with $x_4x_5\in A$. Since $\lambda(D)\ge2$, $y$ have at least two in-arcs from and two out-arcs into $C^1\cup C^2$. W.l.o.g., assume that there exists an arc from $y$ to $C^1$. By Fact 1, $N(y)\subset C^1$. It follows that $N^+(y)=\{x_1,x_2\}$ and $N^-(y)=\{x_3,x_4\}$ by Fact 2.
Then we find a Hamilton dipath as $yx_1\in A$, a contradiction.

{\bf Case 2:} $|C^1|=5$ and $|C^2|=3$.

Let $C^1=x_1x_2x_3x_4x_5x_1$ and $C^2=x_6x_7x_8x_6$ with $x_5x_6\in A$.
Analogously to Case 1, $N(y)\subset C^1$ as $|C^2|=3$ and $D$ is oriented. Consider several cases below.

~\\
\vspace{2mm}
\begin{tabular}{l|c}\hline
	Case & Contradiction\\\hline
	$x_1y,yx_3\in A$ & $D$ has a Hamilton dipath $x_2x_3x_1yx_4x_5x_6x_7x_8$\\\hline
	$x_1y,x_3x_1\in A$ & $D$ has a Hamilton dipath $x_2x_3x_1yx_4x_5x_6x_7x_8$\\\hline
	$x_1y,x_1x_3,x_3x_5\in A$ & $N^+(x_4)=x_5$\\\hline
	$y\notin N_D(x_1)$ & $D$ has a Hamilton dipath $yx_2x_3x_1x_4x_5x_6x_7x_8$\\\hline
\end{tabular}

This compeltes the proof of Proposition~\ref{h3-1}.

\subsection{Lemma~\ref{h3-2}}
Suppose that $P=x_1x_2\ldots x_8$ is the longest dipath of $D$ by Proposition~\ref{h2-2}. Let $X=V(P)$ and $y=V-X$.
\begin{description}
	\item[Claim~\ref{h3-2}.1] If $x_4x_1,x_8x_6\in A$, then $x_5$ is not adjacent to $y$.
\end{description}
{\it Proof.}
Suppose to the contrary that $x_5$ is adjacent to $y$ when $x_4x_1,x_8x_6\in A$. Let $C'=x_1x_2x_3x_4x_1$ and $C'=x_6x_7x_8x_6$.

First assume $x_5x_9\in A$. Note that $N^+(y)\subset C'-x_1$ as $P_8$ is the longest dipath in $D$. Then $N^+(y)=\{x_3,x_4\}$ by Proposition~\ref{h3-1}. Since $\lambda(D)\ge2$, $N_{D-x_4}^-(x_5)\subset \{x_1,x_2\}$ by Proposition~\ref{h3-1} and $D$ contains no $K_4$.
If $x_2x_5\in A$, then we find a Hamilton dipath of $D$, $yx_3x_4x_1x_2x_5x_6x_7x_8$, a contradiction. Thus $x_1x_5\in A$. Now we find a new dicycle $Z'=x_1x_5yx_3x_4x_1$, which implies that $D$ has a Hamilton dipath by Proposition~\ref{h3-1}.

Henceforth, $x_9x_5\in A$. Note that $N^-(y)=\{x_6,x_7\}$ by the longestness of $P$ in $D$. It follows that we find a new dicycle $Z''=x_5x_6x_7yx_5$. Then by Proposition~\ref{h3-1}, $D$ has a Hamilton dipath, a contradiction.
\hfill $\lozenge$

\begin{description}
	\item[Note~\ref{h3-2}.1] Let $B^+$ and $T^-$ respectively be an out-branching and an in-tree of $D[X]$. If $V(T^-)=X-v$ for some $v\in X$ and $v$ has an out-arc which is not in $A(B^+)$, then $D$ has a good pair.\\
	Analogously, let $B^-$ and $T^+$ respectively be an in-branching and an out-tree of $D[X]$. If $V(T^+)=X-v$ for some $v\in X$ and $v$ has an in-arc which is not in $A(B^-)$, then $D$ has a good pair.
	\item[Note~\ref{h3-2}.2] Let $B^+$ and $T^-$ respectively be an out-branching and an in-tree of $D[X]$ with $V(T^-)=X-\{v_1,v_2\}$ for some $v_1,v_2\in X$. If for any $i\in[2]$, $v_i$ has an out-arc which is not in $A(B^+)$ and at least one of the two arcs is in $D[X]$, then $D$ has a good pair.\\
	Analogously, let $B^-$ and $T^+$ respectively be an in-branching and an out-tree of $D[X]$ with $V(T^+)=X-\{v_1,v_2\}$ for some $v_1,v_2\in X$. If for any $i\in[2]$, $v_i$ has an in-arc which is not in $A(B^-)$ and at least one of the two arcs is in $D[X]$, then $D$ has a good pair.
\end{description}
{\it Proof.}
Note~\ref{h3-2}.1 is trivial as $\lambda(D)\ge2$. It suffices to prove Note~\ref{h3-2}.2.
Let $e_i$ be the out-arc of $v_i$ which is described in Note~\ref{h3-2}.2. If some $e_i$ is adjacent to some vertex in $T^-$, then it is just the case in Note~\ref{h3-2}.1. Henceforth assume that $e_i$ is adjacent to $y$ or $v_{3-i}$, for any $i\in[2]$. W.l.o.g., assume that $e_1=v_1v_2$ and $e_2=v_2y$. Now let $B^-=T^-+e_1+e_2$. Since $\lambda(D)\ge2$, $y$ can be added to $B^+$ and $B^-$ respectively by an in-arc of $y$ which is different from $e_2$ and an arbitrary out-arc of $y$.
\hfill $\lozenge$~\\

We have the fact below as $P$ is the longest dipath of $D$.
\begin{description}
	\item[Fact~\ref{h3-2}.1.] $N^-(x_1),N^+(x_8)\subset X$ and $x_8x_1\notin A$, i.e., $D$ has no $C_8$.
\end{description}
We distinguish several cases as follows.

{\bf Case 1:} $x_6\in N^+(x_8)\cap N^-(x_1)$.

~\\
\vspace{2mm}
\begin{tabular}{l|c|c}\hline
	Case & Contradiction & Reference\\\hline
	\multicolumn{3}{l}{{\bf Subcase 1.1:} $x_4x_1\in A$}\\\hline
	$x_8x_4,x_3x_8\in A$ & $B_{D[X]}^+=x_1x_2x_3x_8x_4x_5x_6x_7,~T_{D[X]}^-=x_3x_4x_1+x_7x_8x_6x_1$ & Note~\ref{h3-2}.2\\\hline
	$x_8x_4,x_3x_6\in A;x_2x_5\notin A$ & $B_{D[X]}^+=x_8x_4x_1x_2x_3x_6x_7+x^-_5x_5,~T_{D[X]}^-=x_3x_4x_5x_6x_1+x_7x_8x_6$ & Note~\ref{h3-2}.1\\\hline
	&\multicolumn{2}{l}{where $x_5^-\in D[X]-\{x_2,x_4\}$}\\\hline
	$x_8x_4,x_3x_6,x_2x_5\in A$ & $B_{D[X]}^+=x_1x_2x_3x_6x_7x_8x_4x_5,~T_{D[X]}^-=x_3x_4x_1+x_2x_5x_6x_1+x_8x_6$ & Note~\ref{h3-2}.1\\\hline
	$x_8x_3,x_2x_4\in A$ & $B_{D[X]}^+=x_1x_2x_4x_5x_6x_7x_8x_3,~T_{D[X]}^-=x_2x_3x_4x_1+x_8x_6x_1$ & Note~\ref{h3-2}.2\\\hline
	$x_8x_3,x_2x_6\in A;x_7x_5\notin A$ & $B_{D[X]}^+=x_2x_6x_7x_8x_3x_4x_1+x^-_5x_5,~T_{D[X]}^-=x_4x_5x_6x_1x_2x_3+x_8x_6$ & Note~\ref{h3-2}.1\\\hline
	& \multicolumn{2}{l}{where $x_5^-\in D[X]-\{x_4,x_7\}$}\\\hline
	$x_8x_3,x_2x_6,x_7x_5\in A$ & $B_{D[X]}^+=x_1x_2x_6x_7x_8x_3x_4x_5,~T_{D[X]}^-=x_4x_1+x_8x_6x_1+x_7x_5x_6$ & Note~\ref{h3-2}.2\\\hline
	$x_8x_3,x_2x_8,x_7x_5\in A$ & $x_1x_2x_8x_3x_4x_1$ and $x_5x_6x_7x_5$ are two cycles covering $8$ vertices & Proposition~\ref{h3-1}\\\hline
	$x_8x_3,x_2x_8\in A;x_7x_5\notin A$ & $B_{D[X]}^+=x_2x_8x_3x_4x_1+x_5^-x_5x_6x_7,~B_{D[X]}^-=x_7x_8x_6x_1x_2x_3+x_4x_5x_5^+$\\\hline
	& \multicolumn{2}{l}{where $x_5^-\in D[X]-\{x_4,x_6,x_7\},~x_5^+ \in D[X]-\{x_4,x_6\}$}\\\hline
	$x_8x_2,x_1x_3\in A$ & $B_{D[X]}^+=x_1x_3x_4x_5x_6x_7x_8x_2,~T_{D[X]}^-=x_4x_1x_2x_3+x_8x_6x_1$ & Note~\ref{h3-2}.2\\\hline
	$x_8x_2,x_1x_8\in A$ & $B_{D[X]}^+=x_1x_8x_2x_3x_4x_5x_6x_7,~T_{D[X]}^-=x_4x_1x_2+x_7x_8x_6x_1$ & Note~\ref{h3-2}.2\\\hline
	\multicolumn{3}{l}{{\bf Subcase 1.2:} $x_3x_1\in A$}\\\hline
	$x_8x_2,x_1x_5\in A$ & $B_{D[X]}^+=x_8x_6x_1x_2x_3x_4x_5+x_7^-x_7,~T_{D[X]}^-=x_3x_1x_5x_6x_7x_8x_2$ & Note~\ref{h3-2}.1\\\hline
	& \multicolumn{2}{l}{where $x_7^-\in D[X]-\{x_4,x_6\}$}\\\hline
	$x_8x_2,x_1x_8\in A$ & $B_{D[X]}^+=x_1x_8x_2x_3x_4x_5x_6x_7,~T_{D[X]}^-=x_3x_1x_2+x_7x_8x_6x_1$ & Note~\ref{h3-2}.2\\\hline
	$x_8x_3,x_1x_4\in A$ & $B_{D[X]}^+=x_4x_5x_6x_7x_8x_3x_1x_2,~T_{D[X]}^-=x_8x_6x_1x_4+x_2x_3x_4$ & Note~\ref{h3-2}.2\\\hline
	$x_8x_3,x_6x_4,x_5x_3\in A$ & $B_{D[X]}^-=x_5x_6x_1x_2x_3x_4+x_7x_8x_3,~T_{D[X]}^+=x_8x_6x_4x_5x_3x_1+x_6x_7$ & Note~\ref{h3-2}.1\\\hline
	$x_8x_3,x_6x_4,x_5x_8\in A$ & $B_{D[X]}^-=x_5x_6x_1x_2x_3x_4+x_7x_8x_6,~T_{D[X]}^+=x_6x_4x_5x_8x_3x_1+x_6x_7$ & Note~\ref{h3-2}.1\\\hline
	$x_8x_5,x_7x_4\in A$ & $B_{D[X]}^+=x_7x_8x_5x_6x_1x_2x_3x_4,~T_{D[X]}^-=x_8x_6x_7x_4x_5x_3x_1$ & Note~\ref{h3-2}.1\\\hline
	$x_8x_5\in A;x_7x_4\notin A$ & $B_{D[X]}^+=x_6x_7x_8x_5x_3x_1x_2+x_4^-x_4,~T_{D[X]}^-=x_2x_3x_4x_5x_6x_1+x_8x_6$ & Note~\ref{h3-2}.1\\\hline
	& \multicolumn{2}{l}{where $x_4^-$ in $D[X]-\{x_3,x_7\}$}\\\hline
	{\bf Subcase 1.3:} $x_7x_1\in A$ & $C_8=x_5x_8x_6x_7x_1x_2x_3x_4x_5$ & Fact~\ref{h3-2}.1\\\hline
\end{tabular}

The case when $x_3\in N^+(x_8)\cap N^-(x_1)$ is analogous.

{\bf Case 2} $x_5\in N^+(x_8)\cap N^-(x_1)$.

~\\
\vspace{2mm}
\begin{tabular}{l|c|c}\hline
	Case & Contradiction & Reference\\\hline
	\multicolumn{3}{c}{{\bf Subcase 2.1:} $x_3x_1\in A$}\\\hline
	\multicolumn{3}{l}{$x_6x_4\in A$}\\\hline
	$yx_3,yx_7\in A$ & $D$ has a Hamilton dipath $x_6x_4x_5yx_7x_8x_2x_3x_1$\\\hline
	otherwise & $B_{D[X]}^-=x_1x_2x_3x_4x_5+x_6x_7x_8x_5,~T_{D[X]}^+=x_5x_1x_8x_2+x_5x_6x_4$ & Notes~\ref{h3-2}.1 and \ref{h3-2}.2\\\hline
	$x_7x_4\in A$ & $C_8=x_2x_3x_4x_5x_1x_6x_7x_8x_2$ & Fact~\ref{h3-2}.1\\\hline
	\multicolumn{3}{c}{{\bf Subcase 2.2:} $x_6x_1\in A$}\\\hline
	\multicolumn{3}{l}{$x_1x_7\in A$}\\\hline
	$x_2y,x_3y\in A$ & $D$ has a Hamilton dipath $x_7x_8x_3x_4x_5x_6x_1x_2y$\\\hline
	otherwise & $B_{D[X]}^+=x_5x_6x_7+x_5x_1x_2x_3x_4x_8,~T_{D[X]}^-= x_4x_5+x_6x_1x_7x_8x_5$ & Notes~\ref{h3-2}.1 and \ref{h3-2}.2\\\hline
	$x_2x_7\in A$ & $P_8$: $x_3x_4x_5x_6x_1x_2x_7x_8$ & Case 1\\\hline
	$x_3x_7\in A$ & $B_{D[X]}^+=x_5x_1x_2x_3x_4x_8+x_5x_6x_7,~T_{D[X]}^-=x_4x_5+x_3x_7x_8x_5+x_6x_1x_1^+$ & Note~\ref{h3-2}.1\\\hline
	& \multicolumn{2}{l}{where $x_1^+\in \{x_3,x_4,x_5,x_7,x_8\}$}\\\hline
\end{tabular}\\
\vspace{2mm}
\begin{tabular}{l|c}\hline
	Case & Contradiction\\\hline
	\multicolumn{2}{l}{{\bf Subcase 2.3:} $x_7x_1\in A$}\\\hline
	\multicolumn{2}{l}{$x_i^+$: an out-neighbour of $x_i$ which is not the successor of $x_i$ in $P$, $\forall i\in[8]$}\\\hline
	\multicolumn{2}{l}{$x_i^-$: an in-neighbour of $x_i$ which is not the predesessor of $x_i$ in $P$, $\forall i\in[8]$}\\\hline
	\multicolumn{2}{l}{$x_1x_6\in A$}\\\hline
	$x_2^+=x_3^+=y$ & $B_D^+=x_5x_1x_2x_3yx_7+x_5x_6+x_8^-x_8+x_4^-x_4,~B_D^-=x_3x_4x_5+x_2yx_5+x_1x_6x_7x_8x_5$\\\hline
	$x_i^+=y,~x_{5-i}^+=x_6$ & $B_D^+=x_5x_1x_2x_3x_4x_7+x_5x_6+x_8^-x_8+y^-y,~B_D^-=x_4x_5+x_iyx_5+x_1x_6x_7x_8x_5+x_{5-i}x_6$\\\hline
	& \multicolumn{1}{l}{where $y^-\in D\setminus x_i$}\\\hline
	otherwise & $B_{D[X]}^+=P$ and an in-tree $T_{D[X]}^-=x_8x_5x_1+x_4x_7x_1$ by Notes~\ref{h3-2}.1 and \ref{h3-2}.2\\\hline
	\multicolumn{2}{l}{$x_2x_6\in A$}\\\hline
	$x_6^+\neq x_3$ & $B_{D[X]}^+=P,~T_{D[X]}^-=x_8x_5x_1+x_4x_7x_1+x_2x_6x_6^+$ by Note~\ref{h3-2}.1\\\hline
	$x_6^+=x_3;x_3^+\neq y$ & $D\setminus y$ has a good pair\\\hline
	$x_6^+=x_3;x_3^+=y$ & $D$ has a Hamilton dipath $x_4x_7x_8x_5x_1x_2x_6x_3y$\\\hline
	$x_3x_6\in A$ & $B_D^+=x_5x_1+x_5x_6x_2x_3x_4x_7+y^-y+x_8^-x_8,~B_D^-=x_4x_5+x_3x_6x_7x_8x_5+x_1x_2x_5$ ($y^-\in D\setminus x_2$)\\\hline
\end{tabular}

The case when $x_4\in N^+(x_8)\cap N^-(x_1)$ is analogous.

{\bf Case 3} $x_3\in N^-(x_1)$.

~\\
\vspace{2mm}
\begin{tabular}{l|c|c}\hline
	Case & Contradiction & Reference\\\hline
	\multicolumn{3}{l}{{\bf Subcase 3.1:} $x_4x_1\in A$}\\\hline
	$x_3x_5\in A;x_6y\notin A$ & $B_{D[X]}^+=x_3x_4x_5x_6x_7x_8x_2+x_3x_1,~T_{D[X]}^-=x_4x_1x_2x_3x_5+x_8x_6x_6^+$ & Note~\ref{h3-2}.1\\\hline
	& \multicolumn{2}{l}{where $x_6^+ \in V(T^-)$}\\\hline
	$x_3x_5,x_6y\in A$ & $B_D^+=x_3x_4x_5x_6x_7x_8x_2+x_3x_1+y^-y,~B_D^-=x_4x_1x_2x_3x_5+x_8x_6yy^++x_7x_7^+$\\\hline
	&\multicolumn{2}{l}{where $y^-\neq x_6$, $y^+ \in V(T^-)$ and $x_7^+\neq x_8$}\\\hline
	$x_7x_5,x_5x_8\in A$ & $x_4x_1x_2x_3x_4$ and $x_8x_6x_7x_5x_8$ are two cycles covering 8 vertices & Proposition~\ref{h3-1}\\\hline
	$x_7x_5\in A;x_5x_8\notin A$ & $B_{D[X]}^+=x_3x_4x_5x_6x_7x_8x_2+x_3x_1,~T_{D[X]}^-=x_4x_1x_2x_3+x_7x_5x_5^+$ & Note~\ref{h3-2}.2\\\hline
	& \multicolumn{2}{l}{where $x_5^+ \in \{x_1,x_2,x_3\}$}\\\hline
	\multicolumn{3}{l}{{\bf Subcase 3.2:} $x_6x_1\in A$}\\\hline
	 & $B_{D[X]}^-=x_4x_5x_6x_7x_8x_2+x_3x_1x_5,~T_{D[X]}^+=x_6x_1x_2x_3x_4+x_3x_7$ & Note~\ref{h3-2}.2\\\hline
	\multicolumn{3}{l}{{\bf Subcase 3.3:} $x_7x_1\in A$}\\\hline
	$\{x_2,x_5\}\subseteq N^+(x_8)$ & $B_{D[X]}^+=x_3x_4x_5x_6x_7x_8x_2+x_3x_1,~T_{D[X]}^-=x_7x_1x_2x_3x_8x_5$ & Note~\ref{h3-2}.2\\\hline
	$N^+(x_8)=\{x_2,x_6\}$ & $x_5x_6x_7x_1x_5$ and $x_2x_3x_4x_8x_2$ are two cycles covering 8 vertices & Proposition~\ref{h3-1}\\\hline
	$N^+(x_8)=\{x_5,x_6\}$ & analogous to the case when $N^-(x_1)\supset \{x_3,x_4\}$\\\hline
\end{tabular}

The case of $x_6\in N^+(x_8)$ is analogous.

{\bf Cases 4 to 6}

~\\
\vspace{2mm}
\begin{tabular}{l|c|c}\hline
	Case & Contradiction & Reference\\\hline
	\multicolumn{3}{l}{{\bf Case 4:} $x_4\in N^-(x_1)$}\\\hline
	$x_5x_1\in A$ & $x_5x_6x_7x_8x_5$ and $x_1x_2x_3x_4x_1$ are two cycles covering 8 vertices & Proposition~\ref{h3-1}\\\hline
	$x_6x_1\in A$ & $N^+(x_2)=\{x_3\}$ & $\lambda(D)\ge2$\\\hline
	$x_7x_1\in A$ & $N^+(x_1)=\{x_2\}$ & $\lambda(D)\ge2$\\\hline
	\multicolumn{3}{l}{The case of $x_5\in N^+(x_8)$ is analogous.}\\\hline
	\multicolumn{3}{l}{{\bf Case 5:} $x_5\in N^-(x_1)$}\\\hline
	$x_6x_1,x_8x_4\in A$ & $N^+(x_7)=x_8$ & $\lambda(D)\ge2$\\\hline
	$x_6x_1\in A;N^+(x_8)=\{x_2,x_3\}$ & $N^+(x_2)=\{x_3,x_8\}$ & $D$ is oriented\\\hline
	$x_7x_1,x_8x_2\in A$ & $N^+(x_1)=x_2$ & $\lambda(D)\ge2$\\\hline
	$x_7x_1\in A;N^+(x_8)=\{x_3,x_4\}$ & $N^+(x_2)=x_3$ & $\lambda(D)\ge2$\\\hline
	\multicolumn{3}{l}{The case of $x_4\in N^+(x_8)$ is analogous.}\\\hline
	\multicolumn{3}{l}{{\bf Case 6:} $N^-(x_1)=\{x_6,x_7\}$ and $N^+(x_8)=\{x_2,x_3\}$}\\\hline
	& $N^+(x_2)=x_3$ & $\lambda(D)\ge2$\\\hline
\end{tabular}

This completes the proof of Lemma~\ref{h3-2}.

\end{document}